\documentclass[letterpaper]{article} 
\usepackage{authblk}
\usepackage{arxiv}  
\usepackage{times}  
\usepackage{helvet}  
\usepackage{courier}  
\usepackage[hyphens]{url}  
\usepackage{graphicx} 
\usepackage[numbers,sort&compress]{natbib}
\usepackage[table]{xcolor}
\usepackage{color-edits}
\usepackage{float} 
\usepackage{xspace}
\usepackage{mathtools}
\usepackage{bm}
\usepackage{wrapfig}
\usepackage{booktabs}
\usepackage{tikz}
\usepackage[utf8]{inputenc} 
\usepackage[T1]{fontenc}    
\usepackage{booktabs}       
\usepackage{amsfonts}       
\usepackage{nicefrac}       
\usepackage{microtype}      
\usepackage{lipsum}         
\usepackage{subcaption}
\usepackage{amsmath}
\usepackage{algorithm}
\usepackage{algpseudocode}
\usepackage{pgfgantt}
\usepackage{appendix}
\usepackage{mathrsfs}
\usepackage{multirow}

\usepackage{adjustbox}
\usepackage{threeparttable}
\usepackage{hyperref}       
\usepackage{tabularx}
\usepackage{amsthm, comment}
\usepackage{pifont}
\usepackage{enumitem}
\makeatletter
\newsavebox{\@tabnotebox}

\makeatother

\newcommand{\name}{PDHCG-CQP\xspace}
\title{GPU-Accelerated Conic Quadratic Programming with Local Linear Convergence under Strict Complementarity}
\hypersetup{
  hidelinks,
  pdftitle={GPU-Accelerated Conic Quadratic Programming with Local Linear Convergence under Strict Complementarity},
  pdfauthor={Hongpei Li, Yicheng Huang, Huikang Liu, Dongdong Ge, and Yinyu Ye},
  pdfkeywords={conic quadratic programming, restarted averaged PDHG, smoothed duality gap, quadratic growth, strict complementarity, GPU, multi-GPU}
}
\newtheorem{theorem}{Theorem}[section] 
\newtheorem{lemma}[theorem]{Lemma} 
\newtheorem{proposition}[theorem]{Proposition}
\newtheorem{corollary}[theorem]{Corollary}

\theoremstyle{definition}
\newtheorem{assumption}[theorem]{Assumption}
\newtheorem{definition}[theorem]{Definition}

\usepackage{siunitx}
\theoremstyle{remark}
\author[1]{Hongpei Li}
\author[2]{Yicheng Huang}
\author[3]{Huikang Liu}
\author[3]{Dongdong Ge}
\author[3,4]{Yinyu Ye}
\affil[1]{Northwestern University}
\affil[2]{Shanghai University of Finance and Economics}
\affil[3]{Shanghai Jiao Tong University}
\affil[4]{Stanford University}

\AtBeginDocument{%
  \addtolength\abovedisplayskip{-0.05\baselineskip}%
  \addtolength\belowdisplayskip{-0.1\baselineskip}%
  \addtolength\abovedisplayshortskip{-0.05\baselineskip}%
  \addtolength\belowdisplayshortskip{-0.1\baselineskip}%
}

\usepackage{setspace}
\usepackage{titlesec}
\titlespacing*{\section}{0pt}{-0.05\baselineskip}{-0.05\baselineskip}
\titlespacing*{\subsection}{0pt}{-0.075\baselineskip}{-0.075\baselineskip}
\titlespacing*{\subsubsection}{0pt}{-0.025\baselineskip}{-0.05\baselineskip}
\newcommand{\papertablesetup}{\small\setstretch{1.0}\renewcommand{\arraystretch}{1.2}}
\newcommand{\AVERAGEDPDHG}{}
\begin{document}
\maketitle

\begin{abstract}
We present \textbf{\name}, a GPU-accelerated first-order solver for large-scale conic convex quadratic programming. \name supports affine constraints and Cartesian products of nonnegative, second-order, rotated second-order, exponential, and three-dimensional power cones. At its core is a restarted averaged primal-dual hybrid gradient (PDHG) method, whose primal update is computed inexactly by solving a conic quadratic proximal subproblem with projected gradient iterations. We establish local linear convergence of the restarted averaged scheme with both exact and inexact primal proximal evaluations under a uniform local quadratic-growth condition on the smoothed primal-dual gap. We further show that this condition holds under strict complementarity by exploiting a rotated second-order-cone lifting together with local primal and dual regularity conditions. Our C/CUDA implementation combines matrix-free linear algebra, batched cone projections, adaptive inner solves, reflected-Halpern acceleration, and fully device-resident KKT residual computations. It also supports multi-GPU execution through a two-dimensional partitioning of the problem data. Extensive experiments on standard and large-scale quadratic programming (QP), convex quadratically constrained quadratic programming (QCQP), second-order cone programming (SOCP), and quasilinear Fisher equilibrium benchmarks demonstrate that \name achieves state-of-the-art robustness among first-order solvers while scaling efficiently to 8 GPUs and instances with up to $4.4\times10^8$ stored primal coordinates. \name is open source and available at \url{https://github.com/Lhongpei/PDHCG}.
\end{abstract}

\section{Introduction}

Conic convex quadratic programming (CQP) is a fundamental optimization
problem class with applications in machine learning
\cite{tibshirani1996regression}, control \cite{mayne2000constrained},
signal processing \cite{lobo1998applications}, finance
\cite{markowitz1952portfolio}, and market design
\cite{eisenberg1959consensus}.  A quadratic objective represents
curvature, regularization, or risk, while conic constraints model geometric
structures such as norms, perspectives, exponential relations, and power
laws.

Formally, we consider the conic convex quadratic program
\begin{equation}
\label{eq:cqp}
\begin{aligned}
    \min_{x\in E}\quad
    & f(x):=\frac12\langle x,Qx\rangle+\langle c,x\rangle\\
    \mathrm{s.t.}\quad
    & Ax=b,\qquad x\in K,
\end{aligned}
\end{equation}
where $E$ and $Y$ are finite-dimensional Euclidean spaces,
$A:E\to Y$ is linear, $Q:E\to E$ is self-adjoint and positive semidefinite,
$c\in E$, $b\in Y$, and $K\subseteq E$ is a nonempty closed convex cone.  The
associated equality-conic saddle function is
\begin{equation}
\label{eq:saddle}
    \mathcal L(x,y)
    :=
    \frac12\langle x,Qx\rangle+\langle c,x\rangle
    +\langle y,b-Ax\rangle,
    \qquad x\in K,\quad y\in Y,
\end{equation}
and a saddle point of \eqref{eq:saddle} recovers a primal-dual solution of
\eqref{eq:cqp}.

Classical conic and quadratic-programming solvers are dominated by interior-point and active-set methods. Modern interior-point solvers such as Clarabel \cite{goulart2024clarabel}, CuClarabel \cite{chen2024cuclarabel}, and QOCO-GPU \cite{chari2026qocogpu} provide reliable high-accuracy solutions, with the latter two accelerating sparse factorizations on GPUs. Their scalability, however, is limited by repeated Newton--KKT solves, whose sparse factorizations can become expensive in both time and memory because of fill-in and are difficult to distribute efficiently across multiple accelerators. First-order splitting methods alleviate this bottleneck: SCS \cite{o2016conic,scs}, ABIP+ \cite{lin2021admm,deng2025enhanced}, and OSQP \cite{stellato2020osqp} reuse a single factorization or employ indirect linear solves, but the former can still be prohibitive at scale while the latter requires multiple matrix--vector products per iteration. Fully matrix-free restarted primal-dual hybrid gradient (PDHG) methods eliminate linear-system solves altogether and have demonstrated strong scalability for LP in PDLP \cite{applegate2021practical} and cuPDLP \cite{lu2023cupdlpc}, and for box-constrained convex QP in PDQP \cite{lu2023practical}, HPR-QP \cite{chen2025hpr}, and PDHCG \cite{huang2025restarted}.

For conic feasible sets, this matrix-free toolbox is still incomplete.
PDCS \cite{lin2025pdcs}, the closest PDHG-based conic solver, treats a
linear objective and handles quadratic objectives through an epigraph
reformulation, which requires a factorization $Q=B^*B$, roughly triples
the outer iteration count, and adds the projection of a rotated second-order-cone (RSOC) block of
dimension at least $\operatorname{rank}(Q)+2$ to every iteration;
Section~\ref{sec:rh-restart-theory} quantifies these effects.  Moreover,
the current PDCS implementation does not yet convert the matrix-free
design into a practical advantage: in the conic benchmarks of
Section~\ref{sec:experiments}, it is outperformed by a wide margin --- even
by factorization-based solvers on the large instances for which
first-order methods are intended.  No available solver combines a native
quadratic objective, general vector-level cones, and matrix-free GPU
execution.  We therefore seek a method with the following properties.
\begin{itemize}
    \item \textbf{Matrix-free computation.}
    Each outer iteration uses only applications of $A$, $A^*$, and $Q$,
    together with product-cone projections; no factorization is required.

    \item \textbf{Native conic quadratic modeling.}
    The quadratic objective is kept explicitly, and the principal
    vector-level cones are supported without an objective-epigraph
    reformulation.

    \item \textbf{Linear convergence with implementable proximal steps.}
    The convergence theory covers both the exact primal proximal map and
    practical inexact solutions of the conic quadratic subproblems.

    \item \textbf{Scalable implementation.}
    The computational kernels map naturally to a single GPU and admit a
    distributed extension when the problem exceeds the memory of one
    device.
\end{itemize}

This paper develops \name, a matrix-free solver for \eqref{eq:cqp} with
these four properties.  The algorithm is restarted averaged PDHG applied
directly to the saddle problem \eqref{eq:saddle}: each iteration performs a
primal-first PDHG update whose primal step is a strongly convex conic
quadratic proximal subproblem, and each epoch restarts from the Ces\`aro
average of its iterates.  The proximal subproblem generally has no closed
form; it is solved by direct weighted cone projections when a separable
oracle is available and by matrix-free projected-gradient iterations
otherwise, and the convergence analysis covers both exact and inexact
proximal evaluations, with the latter satisfying a stated error budget.

The analysis rests on a single local regularity property: uniform
quadratic growth of the smoothed duality gap near the solution set.  A
single PDHG iterate provides no usable gap estimate, but the one-step
energy inequalities telescope into an $O(1/T)$ ergodic bound at the epoch
average, and quadratic growth converts this bound into a geometric
contraction of the distance to the KKT set; a perturbation argument
extends the contraction to inexact proximal solves.  Quadratic growth, in
turn, is not automatic for nonpolyhedral cones.  We characterize it
exactly through a primal and a dual error bound and verify both under
strict complementarity, so that the ``global convergence plus
strict-complementarity-driven local linear convergence'' principle,
recently established for semidefinite programming (SDP)
\cite{jiang2026pdhgsdp,kang2025admmsdp},
extends to the general conic program with a quadratic objective.

The contributions of this paper can be summarized as follows.
\begin{itemize}
    \item \textbf{Quadratic growth for conic convex QP under strict
    complementarity.}
    We characterize uniform local quadratic growth of the smoothed
    primal-dual gap exactly through a primal and a dual error bound, and
    we verify both bounds --- through a rotated-SOC lifting, the error
    bound of \cite{ding2023strict}, slack regularity, and normal-cone
    calmness --- when $K$ is a finite product of nonnegative, second-order,
    and positive-semidefinite cones whose local KKT centers are strictly
    complementary.

    \item \textbf{Local linear convergence of restarted averaged PDHG with
    an inexact inner solver.}
    Under this quadratic growth, restarting from epoch averages contracts
    the distance to the KKT set geometrically.  A solver-independent
    relative proximal-error budget preserves the rate, and a fixed number
    of warm-started projected-gradient steps per subproblem satisfies the
    budget.

    \item \textbf{A state-of-the-art GPU solver.}
    Our matrix-free CUDA implementation combines structured quadratic
    products, batched cone projections, adaptive inner tolerances, and a
    multi-GPU extension, delivering state-of-the-art performance on the
    tested large-scale conic QP benchmarks.
\end{itemize}

The experiments in Section~\ref{sec:experiments} support these claims.  On
the Maros-M{\'e}sz{\'a}ros benchmark
\cite{maros1999repository}, \name solves 126 instances at tolerance
$10^{-6}$, the most among the tested solvers, with the best average
runtime; on the Mittelmann QP benchmark
\cite{mittelmann2021decision} it solves 17 instances at both $10^{-6}$ and
$10^{-8}$, more than any other tested solver.  On the public convex quadratically constrained
quadratic-programming (QCQP) benchmark it is the only first-order solver that solves every instance
through $10^{-6}$, and on the Mittelmann second-order cone-programming (SOCP) benchmark
\cite{mittelmann2024socp} it attains the best solved counts among the
tested first-order methods.  On large sparse Lasso QPs it is the fastest
solver on seven of the nine instances. On large-scale quasilinear Fisher
equilibrium instances, the distributed implementation achieves up to a
$7.97\times$ speedup on eight GPUs and is the only solver reported as reaching
optimality at $n=10^7$ buyers; the large Lasso instances similarly benefit
from the multi-GPU implementation of \name.

\subsection{Related literature}

\paragraph{Conic and quadratic-programming solvers.}
Interior-point and active-set methods remain the standard approaches for
obtaining high-accuracy solutions of conic and quadratic programs.  Mature
commercial solvers, including MOSEK \cite{mosek2025manual}, Gurobi
\cite{gurobi2024manual}, and COPT \cite{ge2022cardinal}, provide highly
optimized interior-point implementations for conic and convex quadratic
programs and serve as high-accuracy baselines in our experiments.
Clarabel \cite{goulart2024clarabel} provides an interior-point method for
quadratic objectives over general convex cones, while CuClarabel
\cite{chen2024cuclarabel}, QOCO-GPU \cite{chari2026qocogpu}, and NVIDIA's cuOpt \cite{nvidia2025cuopt} investigate
GPU acceleration of sparse interior-point linear algebra.  These approaches
are complementary to the matrix-free regime studied here.

\paragraph{First-order and splitting-based solvers.}
SCS \cite{o2016conic,scs} and ABIP+ \cite{lin2021admm,deng2025enhanced} solve homogeneous
self-dual formulations using ADMM, while OSQP \cite{stellato2020osqp}
specializes ADMM to convex QPs.  Their direct
implementations rely on a reusable factorization of a fixed linear system;
indirect variants replace the factorization by iterative linear solves but
typically require several matrix-vector products per outer iteration.  PDCS \cite{lin2025pdcs} instead provides a matrix-free, GPU-oriented PDHG
method for linear-objective conic programs and handles convex QPs through
conic reformulation.  Our method retains the quadratic objective in the
primal PDHG proximal step and therefore avoids the corresponding
objective-epigraph lifting.

\paragraph{Restarted primal-dual methods.}
Restarted PDHG underlies the PDLP family \cite{applegate2021practical}
for large-scale LP, and GPU implementations such as cuPDLP
\cite{lu2023cupdlp} and cuPDLP-C\cite{lu2023cupdlpc} demonstrate the scalability of this approach.  PDQP \cite{lu2023practical} extends related first-order ideas to convex
QP, while HPR-QP \cite{chen2025hpr} and the earlier PDHCG method
\cite{huang2025restarted} provide further restarted primal-dual QP
algorithms.  Restarted Halpern PDHG and its reflected variant \cite{lu2024restarted}
have also been analyzed for LP.  Concurrent work by
\cite{liu2026reflected} studies a broader reflected-Halpern framework under
fixed-point sharpness.  The present paper instead analyzes restarted
averaged PDHG for the conic saddle problem through the smoothed duality
gap, with explicit perturbation guarantees for inexact quadratic proximal
solves, and establishes the required quadratic growth from conic primal and
dual geometry.

\subsection{Notation}

The adjoint of $A$ is denoted by
$A^*$, the dual cone of $K$ by $K^*$, and the normal cone of $K$ at $x$ by
$N_K(x)$.  For a primal-dual pair $(x,y)$, define the dual slack $s(x,y):=Qx+c-A^*y.$ 
A point $(x^\star,y^\star)$ is a saddle point if and only if, with
$s^\star=s(x^\star,y^\star)$,
\begin{equation}
\label{eq:kkt-cqp}
    Ax^\star=b,\qquad
    x^\star\in K,\qquad
    s^\star\in K^*,\qquad
    \langle x^\star,s^\star\rangle=0.
\end{equation}
We denote the saddle-point set and its primal and dual projections by
\[
    Z^\star
    :=
    \{(x,y):(x,y)\text{ satisfies \eqref{eq:kkt-cqp}}\},
    \qquad
    X^\star:=\operatorname{proj}_E Z^\star,
    \qquad
    Y^\star:=\operatorname{proj}_Y Z^\star.
\]
Throughout the convergence analysis, we assume $Z^\star\neq\emptyset$ and
fix a reference KKT point $\bar z\in Z^\star$.  All unqualified norms are
Euclidean, and product spaces use the corresponding product norm. 
For a nonempty set $S$, let
\[
    \operatorname{dist}(u,S):=\inf_{v\in S}\|u-v\|,
    \qquad
    \Pi_S(u):=\operatorname*{argmin}_{v\in S}\|u-v\|
\]
whenever the projection is single-valued. For $r>0$ and $z\in E\times Y$, define
\[
    B_r(z)
    :=
    \{w\in E\times Y:\|w-z\|<r\},
    \qquad
    \overline B_r(z)
    :=
    \{w\in E\times Y:\|w-z\|\le r\}.
\]


\ifdefined\AVERAGEDPDHG
\section{PDHCG for Conic Quadratic Programming}
\label{sec:averaged-restart-theory}
\label{sec:rh-restart-theory}

In this section, we present PDHCG-CQP, which directly applies the restarted
averaged PDHG method to the conic saddle-point problem \eqref{eq:saddle}.

\begin{algorithm}[H]
\caption{\name: restarted averaged PDHG for the original conic QP}
\label{alg:pdhcq1}
\begin{algorithmic}[1]
\State \textbf{Input:} $z^{0,0}=(x^{0,0},y^{0,0})$, stepsizes $\tau,\sigma>0$, an epoch length $T$.
\For{$n=0,1,2,\ldots$}
    \State Initialize the running average $\bar z^{n,0}=z^{n,0}$.
    \For{$k=0,1,\ldots,T-1$}
        \State Compute exactly or approximately
        \[
            x^{n,k+1}
            \approx
            \arg\min_{x\in K}
            \left\{
                \frac12\langle x,Qx\rangle+\langle c,x\rangle
                +\frac{1}{2\tau} \left\|x-\left( x^{n,k}+\tau A^*y^{n,k}\right) \right\|^2
            \right\}.
        \]
        \State Set
        $x_r^{n,k+1}=2x^{n,k+1}-x^{n,k}$ and
        $y^{n,k+1}=y^{n,k}+\sigma(b-Ax_r^{n,k+1})$.
        \State Update the running average:
        \[
            \bar z^{n,k+1}
            =
            \frac{k}{k+1}\bar z^{n,k}
            +
            \frac{1}{k+1}z^{n,k+1}.
        \]
    \EndFor
    \State Restart from the epoch average:
    $z^{n+1,0}=\bar z^{n,T}$.
\EndFor
\end{algorithmic}
\end{algorithm}

Algorithm~\ref{alg:pdhcq1} has two nested loops. Within an epoch, the inner
loop performs primal-first PDHG updates and maintains the running Ces\`aro
average of its iterates; the outer loop restarts each epoch from the previous
epoch average. Since every exact or feasible inexact primal iterate lies in
$K$ and $K$ is convex, the averaged primal component remains conically
feasible. The displayed minimization is either exact, when it admits a closed
form, or inexact, computed by an arbitrary inner solver. The theory below
covers both cases.

\paragraph{Benefit of PDHCG.}
The predecessor of the present method, PDHCG \cite{huang2025restarted}, was proposed for large-scale
convex QP as a restarted primal-dual method that solves the strongly convex
primal proximal subproblem accurately by conjugate-gradient-type inner
iterations, rather than taking a single forward gradient step.  Treating the quadratic term through its
proximal subproblem moves the curvature of $Q$ from the outer loop into
cheap inner iterations and thereby reduces the number of outer PDHG
iterations substantially.  The same design principle underlies \name.

Retaining the quadratic objective natively is also preferable to
eliminating it.  If $Q=B^*B$, problem \eqref{eq:cqp} admits the rotated
second-order-cone (RSOC) epigraph reformulation
\[
    \min_{x,t}\ t+\langle c,x\rangle
    \quad\mathrm{s.t.}\quad
    Ax=b,\quad x\in K,\quad (t,1,Bx)\in\mathcal Q_r,
\]
where
$\mathcal Q_r:=\{(u,v,w):u,v\ge0,\ 2uv\ge\|w\|^2\}$, after which any conic
solver with a linear objective applies.  Three considerations argue against
this route.  First, unless a factor $B$ is supplied, the reformulation
requires a potentially expensive and fill-inducing factorization of $Q$.
Second, the reformulation degrades the iteration path itself: the outer
iteration count grows --- roughly threefold in
Table~\ref{tab:native-q-vs-rsoc} --- and every outer iteration must
additionally project onto an RSOC block of dimension at least
$\operatorname{rank}(Q)+2$, which also adds storage and communication
costs.  These per-iteration lifted projections are typically more expensive
than the product-set projections performed inside the projected-gradient
inner solves of the native formulation. Third, the native formulation exposes the curvature of
$Q$ directly through the proximal Hessian $Q+\tau^{-1}I$, which could
improve practical convergence when this curvature is informative, although
the benefit is problem-dependent.

The test problems in Table~\ref{tab:native-q-vs-rsoc} are drawn from the
QPLIB-QCQP benchmark collection maintained by Hans Mittelmann
\cite{mittelmann2021decision}.  On these instances, the native formulation
is approximately three times faster and uses about one third as many
iterations as the RSOC reformulation.
\begin{table}[htbp]
    \centering
    \papertablesetup
    \caption{Native quadratic objective versus the RSOC reformulation
    on the Mittelmann QCQP benchmark instances.}
    \label{tab:native-q-vs-rsoc}
    \medskip
    \begin{tabular}{ccccc}
        \toprule
        \multirow{2}{*}{Target accuracy}
        & \multicolumn{2}{c}{Native quadratic}
        & \multicolumn{2}{c}{RSOC reformulation}\\
        \cmidrule(lr){2-3}\cmidrule(lr){4-5}
        & Time & Iterations & Time & Iterations\\
        \midrule
        $10^{-4}$ & 4.67  & 25,811  & 13.28  & 76,836\\
        $10^{-6}$ & 37.83 & 310,074 & 118.04 & 868,239\\
        $10^{-8}$ & 82.24 & 687,301 & 252.46 & 1,911,390\\
        \bottomrule
    \end{tabular}
\end{table}

The rest of this section proves eventual local linear convergence of
PDHCG-CQP.  Its local convergence analysis first uses the exact PDHG operator to
establish the basic epoch contraction and then passes to an implementable
inexact PDHG operator, in which the primal proximal subproblem may be approximated by
any inner method satisfying the error conditions stated below.  The conic
lifting introduced later in Section~\ref{sec:conic-geometry} is used only to
verify a quadratic-growth property.

\subsection{Preliminaries}

For saddle-point problems, the primal-dual gap is a standard merit function:
it compares the Lagrangian at the current primal and dual variables against a
primal-dual comparison point.  Ergodic convergence estimates for PDHG are
naturally expressed through this two-point gap
\cite{chambolle2011first}.  For $z=(x,y)\in K\times Y$ and
$\widehat z=(\widehat x,\widehat y)\in K\times Y$, define
\begin{equation}
\label{eq:ordinary-duality-gap}
    \mathcal Q(z,\widehat z)
    :=
    \mathcal L(x,\widehat y)-\mathcal L(\widehat x,y).
\end{equation}
Maximizing this comparison function over $\widehat z$ gives the ordinary
duality gap.  Although it is a natural primal-dual progress measure, this
supremum can be infinite when the primal or dual domain is unbounded.
Following the smoothed-gap framework of \cite{fercoq2022quadratic}, we
instead penalize the distance from the comparison point to a prescribed
center.  The resulting quantity remains finite and can satisfy a quadratic
error bound, thereby linking the ergodic PDHG estimate to restart
contraction.  This framework has also been used in restarted first-order
methods for convex QP, including rAPDHG and the earlier PDHCG method
\cite{lu2023practical,huang2025restarted}.

\begin{definition}[smoothed duality gap]
\label{def:smoothed-gap}
For $\xi>0$, $z=(x,y)\in K\times Y$, and
$\dot z=(\dot x,\dot y)\in E\times Y$, define
\begin{equation}
\label{eq:smoothed-gap}
    G_\xi(z;\dot z)
    :=
    \sup_{\widehat x\in K,\ \widehat y\in Y}
    \left\{
        \mathcal Q(z,\widehat z)
        -\frac{\xi}{2}\|\widehat x-\dot x\|^2
        -\frac{\xi}{2}\|\widehat y-\dot y\|^2
    \right\}.
\end{equation}
\end{definition}

The quadratic penalty makes the supremum finite and is also the bridge
between an ergodic gap estimate and distance to the KKT set.

\begin{assumption}[uniform local quadratic growth]
\label{ass:smoothed-qg}
There are a KKT point $\bar z\in Z^\star$, constants $R>0$, $\xi>0$, and
$\alpha_\xi>0$ such that
\begin{equation}
\label{eq:gap-qg}
    G_\xi(z;z^\star)
    \ge
    \alpha_\xi\operatorname{dist}^2(z,Z^\star)
\end{equation}
for every
\[
    z\in(K\times Y)\cap B_R(\bar z),
    \qquad
    z^\star\in Z^\star\cap B_R(\bar z).
\]
\end{assumption}

The center $z^\star$ in \eqref{eq:gap-qg} is allowed to vary over the local
solution stratum, while the constant $\alpha_\xi$ and the effective
neighborhood remain fixed.  Section~\ref{sec:conic-geometry} verifies this
assumption from primal and dual geometric error bounds.  Thus
Assumption~\ref{ass:smoothed-qg} is the only problem-dependent local
regularity property used in the convergence proof below.

The proof has two stages.  Stage~I assumes that every primal proximal
subproblem is solved exactly.  A one-step energy inequality gives an
$O(1/T)$ gap bound for one Ces\`aro-averaged epoch, and local quadratic
growth turns this bound into a strict contraction.  Stage~II treats the
finite accuracy of a general inexact proximal oracle as a perturbation of the
exact epoch and gives method-independent conditions under which the same
local linear rate is retained.  Projected gradient is then presented only as
one concrete oracle satisfying those conditions.

\subsection{Exact averaged epochs}

We first analyze the exact scheme, in which every primal proximal subproblem
is solved exactly; it is the reference trajectory against which the inexact
method of Stage~II is compared.  Throughout this stage the stepsizes and the
epoch length are fixed.  Set
\[
    F(x)
    :=
    \frac12\langle x,Qx\rangle+\langle c,x\rangle+\delta_K(x),
\]
and fix $\tau,\sigma>0$ such that
\begin{equation}
\label{eq:theory-step-condition}
    \tau\sigma\|A\|^2<1.
\end{equation}
One primal-first PDHG step from $z^k=(x^k,y^k)$ is
\begin{subequations}
\label{eq:theory-pdhg-step}
\begin{align}
    x^{k+1}
    &=
    \operatorname{prox}_{\tau F}(x^k+\tau A^*y^k),
    \label{eq:theory-pdhg-primal}\\
    x_r^{k+1}
    &=
    2x^{k+1}-x^k,\\
    y^{k+1}
    &=
    y^k+\sigma(b-Ax_r^{k+1}),
    \label{eq:theory-pdhg-dual}
\end{align}
\end{subequations}
denoted $\mathcal T(z^k)=z^{k+1}$.  Because the quadratic term is treated
proximally, the stability condition \eqref{eq:theory-step-condition}
involves $A$ but not $\|Q\|$.  During restart epoch $n$, perform $T$ steps
$z^{n,k+1}=\mathcal T(z^{n,k})$ and restart from the Ces\`aro average
\begin{equation}
\label{eq:theory-epoch-average}
    \bar z^{n,T}
    :=
    \frac1T\sum_{k=0}^{T-1}z^{n,k+1},
    \qquad
    z^{n+1,0}:=\bar z^{n,T}.
\end{equation}
Averaging is what produces the required gap estimate: the one-step energy
inequalities telescope into an $O(1/T)$ ergodic bound at $\bar z^{n,T}$,
whose primal component moreover remains in $K$ by convexity.

Joint primal-dual progress is measured in the metric of the symmetric
preconditioner
\begin{equation}
\label{eq:theory-preconditioner}
    P
    :=
    \begin{bmatrix}
        \tau^{-1}I&A^*\\
        A&\sigma^{-1}I
    \end{bmatrix},
\end{equation}
which is positive definite under \eqref{eq:theory-step-condition}.  In this
metric the exact update is a preconditioned resolvent of the KKT operator
and hence Fej\'er monotone with respect to $Z^\star$
(Lemma~\ref{lem:one-step-energy}).  With
\begin{equation}
\label{eq:theory-metric-constants}
    \lambda_-:=\lambda_{\min}(P),
    \qquad
    \lambda_+:=\lambda_{\max}(P),
    \qquad
    \kappa_P:=\sqrt{\lambda_+/\lambda_-},
\end{equation}
write, for the exact restart sequence,
\[
    d_n:=\operatorname{dist}(z^{n,0},Z^\star),
    \qquad
    q_0:=\sqrt{\frac{\lambda_+}{\alpha_\xi T}},
    \qquad
    R_0:=\frac{1+\kappa_P}{1-e^{-1}}\,d_0.
\]

\begin{theorem}[local linear convergence of exact restarted PDHG]
\label{thm:exact-local-linear}
\label{thm:main-local-linear}
Suppose $d_0>0$ and Assumption~\ref{ass:smoothed-qg} holds on
$B_R(\bar z)$, with
\begin{equation}
\label{eq:main-theorem-ball}
    \overline B_{R_0}(z^{0,0})
    \subseteq
    B_R(\bar z).
\end{equation}
If the restart length satisfies
\begin{equation}
\label{eq:main-restart-length}
    T
    \ge
    \max\left\{
        \frac{2\lambda_+}{\xi},
        \frac{e^2\lambda_+}{\alpha_\xi}
    \right\},
\end{equation}
then the distance to the KKT set decreases geometrically:
\begin{equation}
\label{eq:main-linear-rate}
    \operatorname{dist}(z^{n,0},Z^\star)
    \le
    e^{-n}d_0,
    \qquad n\ge0.
\end{equation}
\end{theorem}

Of the two lower bounds in \eqref{eq:main-restart-length}, the first permits
the conversion of the ordinary gap into the smoothed gap, and the second
gives $q_0\le e^{-1}$, so that each exact epoch contracts by at least
$e^{-1}$.  The complete proof is deferred to
Appendix~\ref{app:proof-exact-local-linear}.

The local initialization condition in
Theorem~\ref{thm:exact-local-linear} should be understood as an eventual
condition rather than as a requirement that a user provide a warm start.  For
the exact reference scheme, standard PDHG convergence theory ensures
convergence from an arbitrary initialization to some saddle point
$z^\infty\in Z^\star$ under the usual stepsize condition
\cite{chambolle2011first,chambolle2016ergodic}; finite Ces\`aro averaging
and restart preserve the saddle-point fixed set.  Once the iterates enter a
sufficiently small neighborhood of $z^\infty$, that iterate may be reindexed
as $z^{0,0}$ in Theorem~\ref{thm:exact-local-linear}.  Hence the substantive
local requirement is that the saddle point selected by the global dynamics
admit uniform local quadratic growth.  Under the sufficient conic regularity
conditions developed in Section~\ref{sec:conic-geometry}, strict
complementarity at the limiting KKT point is the central mechanism that
yields this property.  Accordingly, within the strict-complementarity regime
covered by Section~\ref{sec:conic-geometry}, the local initialization
hypothesis amounts, after a finite burn-in, to assuming that the 
limiting KKT point is strictly complementary.  

\paragraph{Strict complementarity.} This interpretation parallels
recent SDP results: PDHG is eventually R-linearly convergent when its limiting
KKT point satisfies strict complementarity \cite{jiang2026pdhgsdp}, and
ADMM is locally linearly convergent when its limiting primal-dual solution is
strictly complementary \cite{kang2025admmsdp}.  Thus, subject to the conic
regularity conditions in Section~\ref{sec:conic-geometry}, our result extends
the same ``global convergence followed by strict-complementarity-driven local
linear convergence'' principle from SDP to conic convex QP with a quadratic
objective.

\subsection{Inexact proximal solves}

We now separate the convergence argument from the choice of inner solver.
For an input $v\in E$, define
\[
    p(v):=\operatorname{prox}_{\tau F}(v)
\]
as the unique minimizer over $K$ of the strongly convex function
\begin{equation}
\label{eq:theory-inner-objective}
    H_v(u)
    :=
    \frac12\langle u,Qu\rangle+\langle c,u\rangle
    +\frac{1}{2\tau}\|u-v\|^2.
\end{equation}
Let $\widetilde z^k=(\widetilde x^k,\widetilde y^k)$ be an inexact orbit and
set
\begin{equation}
\label{eq:theory-shadow-prox}
    \widetilde v^k
    :=
    \widetilde x^k+\tau A^*\widetilde y^k,
    \qquad
    p_{\rm sh}^{k+1}
    :=
    p(\widetilde v^k).
\end{equation}
At step $k$, an arbitrary inner solver returns a point
$\widetilde x^{k+1}\in K$.  Its proximal error is
\begin{equation}
\label{eq:general-inexact-error}
    \delta_k
    :=
    \|\widetilde x^{k+1}-p_{\rm sh}^{k+1}\|.
\end{equation}
The extrapolated primal point and dual update are then
\[
    \widetilde x_r^{k+1}
    =
    2\widetilde x^{k+1}-\widetilde x^k,
    \qquad
    \widetilde y^{k+1}
    =
    \widetilde y^k+\sigma(b-A\widetilde x_r^{k+1}).
\]
For a relative-error parameter $\eta\ge0$, set
\[
  \varepsilon_n:=\sum_{k=0}^{T-1}\delta_{n,k},
    \qquad  q_\eta:=q_0+C_{\rm ep}\eta,
    \qquad
    \beta_\eta:=1+\kappa_P+C_{\rm ep}\eta,
    \qquad
    R_\eta:=\frac{\beta_\eta}{1-q_\eta}\,d_0,
\]
where the constant
$C_{\rm ep}$ is supplied in 
Lemma~\ref{lem:theory-epoch-perturbation} in
Appendix~\ref{app:proof-inexact-local-linear}.
\stepcounter{theorem}
\edef\epochperturbationnumber{\thetheorem}

\begin{theorem}[local linear convergence with inexact proximal steps]
\label{thm:inexact-local-linear}
Suppose Assumption~\ref{ass:smoothed-qg} holds on $B_R(\bar z)$, let
$T\ge2\lambda_+/\xi$, and suppose
\begin{equation}
\label{eq:general-relative-error}
    \varepsilon_n\le\eta d_n,
    \qquad
    q_\eta<1.
\end{equation}
\label{eq:rh-relative-epoch-error}
If the initial point $z^{0,0}$ satisfies
\begin{equation}
\label{eq:inexact-retention-ball}
    \overline B_{R_\eta}(z^{0,0})
    \subseteq
    B_R(\bar z),
\end{equation}
then the distance to the KKT set decreases geometrically:
\begin{equation}
\label{eq:general-relative-rate}
    d_n\le q_\eta^n d_0,
    \qquad n\ge0.
\end{equation}
\end{theorem}

The complete proof, together with the perturbation lemma on which it relies,
is deferred to Appendix~\ref{app:proof-inexact-local-linear}.

The relative-error condition \eqref{eq:general-relative-error} is not
directly implementable: $\varepsilon_n$ aggregates distances to unknown
exact proximal points, and $d_n$ is the distance to the unknown KKT set.
The next subsection replaces both by computable certificates and a
movement-based inner tolerance.

\subsubsection{Accuracy certificates}

The error $\delta_k$ involves the unknown exact point $p_{\rm sh}^{k+1}$,
but it is controlled by computable certificates.  Set
\[
    H_\tau:=Q+\tau^{-1}I,
    \qquad
    \mu_\tau:=\lambda_{\min}(H_\tau) \geq \tau^{-1},
    \qquad
    L_\tau:=\lambda_{\max}(H_\tau) = \|Q\| + \tau^{-1}.
\]
If the inner algorithm returns $\widetilde x\in K$ for the proximal input
$v$ together with a stationarity residual
\begin{equation}
\label{eq:general-stationarity-residual}
    r
    \in
    Q\widetilde x+c+\tau^{-1}(\widetilde x-v)+N_K(\widetilde x),
\end{equation}
then, because the operator
$u\mapsto Qu+c+\tau^{-1}(u-v)+N_K(u)$ is $\mu_\tau$-strongly monotone and
contains $0$ at $p(v)$,
\begin{equation}
\label{eq:general-residual-to-error}
    \|\widetilde x-p(v)\|
    \le
    \mu_\tau^{-1}\|r\|.
\end{equation}
Alternatively, an objective gap $H_v(\widetilde x)-H_v(p(v))\le\Delta$ and
$\mu_\tau$-strong convexity give
$\|\widetilde x-p(v)\|\le\sqrt{2\Delta/\mu_\tau}$.  Any feasible inner
method whose certificates obey a geometric or relative budget therefore
satisfies Theorem~\ref{thm:inexact-local-linear}, regardless of how the
inexact point is produced.

\subsubsection{An illustrative example: Projected Gradient}

Projected gradient (PG) is one convenient way to produce the required inexact
proximal point, but it is not required by the preceding theorem.  Choose
\begin{equation}
\label{eq:theory-pg-parameters}
    \eta_{\rm PG}
    :=
    \frac{2}{L_\tau+\mu_\tau},
    \qquad
    \rho_{\rm PG}
    :=
    \frac{L_\tau-\mu_\tau}{L_\tau+\mu_\tau}
    \in[0,1).
\end{equation}

\begin{proposition}[fixed projected-gradient work]
\label{prop:fixed-pg-work}
Suppose each proximal subproblem is warm-started at
$u^{n,k,0}=\widetilde x^{n,k}$ and solved by $J$ PG steps with stepsize
$\eta_{\rm PG}$.  There is a local constant $C_w>0$, independent of
$J\ge1$, such that
\begin{equation}
\label{eq:theory-fixed-j-error}
    \varepsilon_n
    \le
    T C_w\rho_{\rm PG}^Jd_n.
\end{equation}
Consequently, any fixed $J$ satisfying
$T C_w\rho_{\rm PG}^J<1$ gives local Q-linear convergence.
\end{proposition}

The proof is deferred to Appendix~\ref{app:proof-fixed-pg-work}. Proposition \ref{prop:fixed-pg-work} suggests that we should choose $J$ in the order of 
$$
\log_{\rho_{\rm PG}} T^{-1} = \frac{\log T}{\log \rho_{\rm PG}^{-1}} \leq \frac{\log T}{\log \left(1 + \frac{2}{\tau \|Q\|} \right)}.
$$
Thus, a logarithmic number of the restart length $T$ is sufficient to reduce the proximal error to the accuracy required by the convergence analysis. In particular, the inner work grows only mildly with $T$, while its constant depends on the contraction factor $\rho_{\rm PG}$, or equivalently on the conditioning of the quadratic proximal subproblem through $\tau|Q|$.

\else
\section{PDHCG for Conic Quadratic Programming}
\label{sec:rh-restart-theory}

In this section, we present PDHCG-CQP, which directly applies a restarted
reflected-Halpern PDHG method to the conic saddle-point problem
\eqref{eq:saddle}.  The exact-operator analysis specializes the reflected-Halpern
sharpness framework of \cite{lu2024restarted,liu2026reflected} to conic
quadratic programming, while the inexact analysis accounts explicitly for
errors in the quadratic proximal step.

\begin{algorithm}[H]
\caption{\name: restarted reflected-Halpern PDHG for the original conic QP}
\label{alg:pdhcq1}
\begin{algorithmic}[1]
\State Choose $z^{0,0}$, $\rho\in[1/2,1]$, stepsize $\{\tau_n, \sigma_n\},$ a restart-length policy $\{N_n\}$.
\For{$n=0,1,2,\ldots$}
    \State Set the initial Halpern state and the anchor $h^{n,0}=a^n=z^{n,0}$.
    \For{$k=0,1,\ldots,N_n-1$}
        \State Compute exactly or approximately
        \[
             x^{n,k+1}
            \approx
            \arg\min_{x\in K}
            \left\{
                \frac12\langle x,Qx\rangle+\langle c,x\rangle
                +\frac{1}{2\tau_n} \left\|x- \left(h_x^{n,k}+\tau_nA^*h_y^{n,k} \right) \right\|^2
            \right\}.
        \]
        \State Form the stored PDHG candidate
        \[
        \begin{aligned}
            y^{n,k+1}
            &=
            h_y^{n,k}
            +\sigma_n\!\left[
                b-A\bigl(2x^{n,k+1}-h_x^{n,k}\bigr)
            \right].
        \end{aligned}
        \]
            \State Set $
            z^{n,k+1}
            =
            (x^{n,k+1},y^{n,k+1})$ and 
            \[
                h^{n,k+1}
                =
                \frac{k+1}{k+2}
                \bigl(2\rho z^{n,k+1}+(1-2\rho)h^{n,k}\bigr)
                +
                \frac{1}{k+2}a^n .
            \]
    \EndFor
    \State Restart from the final stored candidate:
    $z^{n+1,0}=z^{n,N_n}$.
\EndFor
\end{algorithmic}
\end{algorithm}

Algorithm~\ref{alg:pdhcq1} has two nested loops.  Within an epoch, each
evaluation of the primal-first PDHG operator produces a stored candidate, and the
intermediate state is updated by an anchored reflected-Halpern recursion.  The
last operator evaluation is used only to form the next restart point.  This
distinction matters because a reflected-Halpern state need not have a primal
component in $K$, whereas every stored candidate returned by an exact or
feasible inexact proximal solve does. The displayed minimization is either exact, when it admits a closed form, or inexact, computed by an arbitrary inner solver. The theory below covers both cases.

The $O(1/k)$ sub-linear convergence of restarted reflected-PDHG has been shown in \cite{liu2026reflected}. The rest of this section proves eventual local linear convergence of
PDHCG-CQP.  The analysis first uses the exact PDHG operator and the
reflected-Halpern residual estimate to establish an epoch contraction.  It
then treats a general inexact proximal oracle as an operator perturbation and
gives method-independent conditions under which the same local linear rate is
retained.  Projected gradient is presented only as one concrete oracle
satisfying those conditions.  The smoothed-gap geometry developed separately
in Section~\ref{sec:conic-geometry} is a complementary quadratic-growth
result; the reflected-Halpern proof below uses fixed-point residual sharpness.

\subsection{Stage I: exact reflected-Halpern epochs}

We first isolate the outer restart mechanism by assuming that every primal
proximal subproblem is solved exactly.  This exact scheme provides the
reference trajectory against which the inexact method will be compared in
Stage~II.  The objective of this stage is to convert the $O(1/k)$ fixed-point
residual estimate of a reflected-Halpern orbit into a contraction of the
distance to the KKT set.

Set
\[
    F(x)
    :=
    \frac12\langle x,Qx\rangle+\langle c,x\rangle+\delta_K(x),
\]
so that $F$ incorporates both the quadratic objective and the conic
constraint.  During epoch $n$, fix $\tau_n,\sigma_n>0$.  For $z=(x,y)$,
define
\begin{subequations}
\label{eq:theory-pdhg-map}
\begin{align}
    p_n(z)
    &:=
    \operatorname{prox}_{\tau_nF}(x+\tau_nA^*y),
    \label{eq:theory-primal-step}\\
    q_n(z)
    &:=
    y+\sigma_n\left[b-A\bigl(2p_n(z)-x\bigr)\right],
    \label{eq:theory-dual-step}\\
    \mathcal T_n(z)
    &:=
    (p_n(z),q_n(z)).
    \label{eq:theory-map-definition}
\end{align}
\end{subequations}
The primal update solves the conic quadratic proximal problem, the reflected
point $2p_n(z)-x$ supplies the standard PDHG extrapolation, and the dual update
corrects the multiplier using the corresponding equality residual.  Because
the quadratic term is treated proximally, the stepsize condition below
involves $A$ but not $\|Q\|$.

Define the KKT mapping
\begin{equation}
\label{eq:kkt-map}
    \mathcal F_{\rm KKT}(x,y)
    :=
    \begin{bmatrix}
        Qx+c-A^*y+N_K(x)\\
        Ax-b
    \end{bmatrix}
\end{equation}
and the epoch metric
\begin{equation}
\label{eq:pdhg-metric}
    \mathcal M_n
    :=
    \begin{bmatrix}
        \tau_n^{-1}I&A^*\\
        A&\sigma_n^{-1}I
    \end{bmatrix}.
\end{equation}
The positive off-diagonal blocks are consistent with the saddle coupling
$\langle y,b-Ax\rangle$.

\begin{assumption}[uniform epoch metrics]
\label{ass:uniform-metrics}
There are positive lower and upper bounds on $\tau_n$ and $\sigma_n$, and
there is $\delta_{\mathcal M}>0$ such that
\[
    \tau_n\sigma_n\|A\|^2
    \le
    1-\delta_{\mathcal M}
\]
for every local epoch $n$.  Consequently $\mathcal M_n\succ0$, and there is
$\chi\ge1$ such that
\begin{equation}
\label{eq:metric-equivalence}
    \|u\|_{\mathcal M_{n+1}}
    \le
    \chi\|u\|_{\mathcal M_n}
    \qquad\text{for all }u,n.
\end{equation}
The same bounds provide constants $0<m_-\le m_+<\infty$ satisfying
\begin{equation}
\label{eq:uniform-metric-bounds}
    m_-\|u\|^2
    \le
    \|u\|_{\mathcal M_n}^2
    \le
    m_+\|u\|^2
    \qquad\text{for all }u,n.
\end{equation}
When the stepsizes do not change between epochs, one may take $\chi=1$.
\end{assumption}

The proximal optimality condition and the dual update give the resolvent
identity
\begin{equation}
\label{eq:pdhg-residual-kkt}
    \mathcal M_n\bigl(z-\mathcal T_nz\bigr)
    \in
    \mathcal F_{\rm KKT}(\mathcal T_nz).
\end{equation}
Thus $\mathcal T_n=(I+\mathcal M_n^{-1}\mathcal F_{\rm KKT})^{-1}$ is firmly
nonexpansive in the $\mathcal M_n$-metric and
$\operatorname{Fix}(\mathcal T_n)=Z^\star$.  This identity is also the bridge
between a fixed-point residual and a KKT residual.

Define
\begin{equation}
\label{eq:fp-residual}
    R_n(z)
    :=
    \|z-\mathcal T_nz\|_{\mathcal M_n}.
\end{equation}

\begin{assumption}[uniform local fixed-point sharpness]
\label{ass:fp-sharpness}
There are a neighborhood $U$ of a KKT point $\bar z$, a constant
$\alpha_{\rm fp}>0$, and an epoch index $n_0$ such that
\begin{equation}
\label{eq:fp-sharpness}
    \alpha_{\rm fp}
    \operatorname{dist}_{\mathcal M_n}(z,Z^\star)
    \le
    R_n(z)
\end{equation}
for every $z\in U$ and every $n\ge n_0$.
\end{assumption}

The constant and the neighborhood in Assumption~\ref{ass:fp-sharpness} are
uniform over the local epoch metrics.  A standard sufficient condition is
uniform local metric subregularity of the KKT mapping: if, for every nearby
$w$ and every local epoch,
\begin{equation}
\label{eq:kkt-msr}
    \operatorname{dist}_{\mathcal M_n}(w,Z^\star)
    \le
    \kappa_{\mathcal F}
    \operatorname{dist}_{\mathcal M_n^{-1}}
        (0,\mathcal F_{\rm KKT}(w)),
\end{equation}
then \eqref{eq:pdhg-residual-kkt} and the triangle inequality give
\[
    \operatorname{dist}_{\mathcal M_n}(z,Z^\star)
    \le
    (1+\kappa_{\mathcal F})R_n(z).
\]
Hence one may take
$\alpha_{\rm fp}=(1+\kappa_{\mathcal F})^{-1}$.  For polyhedral $K$, such
residual error bounds follow on bounded sets from standard
piecewise-polyhedral and Hoffman arguments
\cite{robinson1981continuity}.  For nonpolyhedral cones, they require the
corresponding local conic regularity conditions; implementation support for a
cone alone is not claimed to imply \eqref{eq:kkt-msr}.

For $\rho\in[1/2,1]$, define the relaxed reflection
\begin{equation}
\label{eq:relaxed-reflection-map}
    \mathcal S_{\rho,n}
    :=
    (1-\rho)I+\rho(2\mathcal T_n-I)
    =
    2\rho\mathcal T_n+(1-2\rho)I.
\end{equation}
It is nonexpansive in the $\mathcal M_n$-metric and has fixed-point set
$Z^\star$.  Starting from the anchor $a^n$, set $h^{n,0}=a^n$ and
\begin{equation}
\label{eq:theory-rh-step}
    h^{n,k+1}
    =
    \frac{k+1}{k+2}\mathcal S_{\rho,n}(h^{n,k})
    +
    \frac{1}{k+2}a^n.
\end{equation}
The standard Halpern residual estimate for a nonexpansive map
\cite{lu2024restarted} and the identity
$z-\mathcal S_{\rho,n}z=2\rho(z-\mathcal T_nz)$ yield
\begin{equation}
\label{eq:rh-residual-decay}
    R_n(h^{n,k})
    \le
    \frac{1}{\rho(k+1)}
    \|a^n-z^\star\|_{\mathcal M_n},
    \qquad
    z^\star\in Z^\star,\quad k\ge1.
\end{equation}
The choice $\rho=1/2$ gives standard Halpern iteration on $\mathcal T_n$,
whereas $\rho=1$ gives Halpern iteration on the reflected map
$2\mathcal T_n-I$ and improves the residual constant by a factor of two.

In epoch $n$, the first $N_n-1$ evaluations generate the states in
\eqref{eq:theory-rh-step}; the final evaluation is stored as the restart
candidate,
\begin{equation}
\label{eq:exact-restart-map}
    a^{n+1}
    :=
    \mathcal T_n(h^{n,N_n-1}).
\end{equation}
Let $N_{\min}\ge2$ be a uniform lower bound with
$N_n\ge N_{\min}$, and write
\[
    d_n
    :=
    \operatorname{dist}_{\mathcal M_n}(a^n,Z^\star),
    \qquad
    q
    :=
    \frac{\chi}{\rho\alpha_{\rm fp}N_{\min}},
    \qquad
    R_0
    :=
    \frac{2d_0}{\sqrt{m_-}(1-q)}.
\]

\begin{theorem}[local linear convergence of exact reflected-Halpern PDHG]
\label{thm:exact-local-linear}
\label{thm:main-local-linear}
Suppose Assumptions~\ref{ass:uniform-metrics} and
\ref{ass:fp-sharpness} hold, relabel the first local epoch $n_0$ as epoch zero,
and let $N_n\ge N_{\min}\ge2$.  If
\begin{equation}
\label{eq:restart-length-condition}
    q
    =
    \frac{\chi}{\rho\alpha_{\rm fp}N_{\min}}
    <1
\end{equation}
and $\overline B_{R_0}(a^0)\subseteq U$, then every exact
reflected-Halpern trajectory remains in $U$ and
\begin{equation}
\label{eq:exact-contraction}
    d_{n+1}
    \le
    \frac{\chi}{\rho\alpha_{\rm fp}N_n}d_n
    \le
    qd_n.
\end{equation}
Consequently $d_n\le q^nd_0$, and the anchors converge R-linearly to a point
of $Z^\star$.
\end{theorem}

\paragraph{Strict complementarity.}
The local initialization condition in
Theorem~\ref{thm:exact-local-linear} should be understood as an eventual
condition rather than as a requirement that a user provide a warm start.
Global convergence results for Halpern and reflected-Halpern fixed-point
iterations \cite{lu2024restarted,liu2026reflected} imply that, under their
standard global hypotheses, an arbitrary initialization converges to a
fixed point of the PDHG operator.  Once the global dynamics select a
limit $z^\infty\in Z^\star$ and enter a sufficiently small neighborhood of
that point, the corresponding epoch can be reindexed as $a^0$ in
Theorem~\ref{thm:exact-local-linear}.  Thus the substantive local requirement
for the present theorem is fixed-point sharpness at the selected limit.
Strict complementarity, together with the relevant conic nondegeneracy and
second-order regularity, is a standard route to the KKT metric subregularity
in \eqref{eq:kkt-msr}.  This interpretation parallels recent SDP results:
PDHG \cite{jiang2026pdhgsdp} is eventually R-linearly convergent when its
limiting KKT point satisfies strict complementarity, and ADMM
\cite{kang2025admmsdp} is locally linearly convergent when its limiting
primal-dual solution is strictly complementary.  Accordingly, under the additional regularity needed
for \eqref{eq:kkt-msr}, the present result extends this
strict-complementarity-driven local linear convergence principle from SDP to
general conic convex QP.  The smoothed-gap quadratic growth established in
Section~\ref{sec:conic-geometry} is closely related geometric information, but
does not by itself replace fixed-point sharpness.

For fixed metrics, $\chi=1$.  A sufficient epoch length for contraction by
$1/e$ is
\[
    N_{\min}
    \ge
    \frac{e\chi}{\rho\alpha_{\rm fp}}.
\]
Thus the sufficient length at $\rho=1$ is half that at $\rho=1/2$.  This is a
factor-of-two improvement in the residual-complexity bound, not a claim that
every realized runtime is exactly halved.  The complete proof of
Theorem~\ref{thm:exact-local-linear} is deferred to
Appendix~\ref{app:proof-exact-local-linear}.

\subsection{Stage II: general inexact proximal solves}

We now separate the convergence argument from the choice of inner solver.
Let $\widetilde{\mathcal T}_{n,k}$ denote a PDHG step in which the primal
proximal problem is solved approximately, and write
\begin{equation}
\label{eq:operator-error}
    \widetilde{\mathcal T}_{n,k}(z)
    =
    \mathcal T_n(z)+e_{n,k}.
\end{equation}
Within an epoch, $k=0,\ldots,N_n-1$ indexes its $N_n$ operator evaluations;
the last evaluation produces the stored restart candidate.  If
$\widetilde p$ is the returned primal point and
$\Delta p=\widetilde p-p_n(z)$, then the associated perturbation of the dual
candidate is $-2\sigma_nA\Delta p$.  The off-diagonal terms in
\eqref{eq:pdhg-metric} cancel exactly, giving
\begin{equation}
\label{eq:prox-error-to-operator-error}
    \|e_{n,k}\|_{\mathcal M_n}
    =
    \|(\Delta p,-2\sigma_nA\Delta p)\|_{\mathcal M_n}
    =
    \tau_n^{-1/2}\|\Delta p\|.
\end{equation}
This identity lets any proximal-error certificate control the perturbation of
the full primal-dual operator.

Let $a^n$ now denote the current, possibly inexact, anchor, and define
\begin{equation}
\label{eq:aggregate-operator-error}
    d_n
    :=
    \operatorname{dist}_{\mathcal M_n}(a^n,Z^\star),
    \qquad
    E_n
    :=
    \sum_{k=0}^{N_n-1}\|e_{n,k}\|_{\mathcal M_n}.
\end{equation}
Within each epoch, the exact trajectory used below is a comparison trajectory
started from this same anchor.

\stepcounter{theorem}
\edef\epochperturbationnumber{\thetheorem}
The one-epoch stability estimate stated and proved as
Lemma~\ref{lem:theory-epoch-perturbation} in
Appendix~\ref{app:proof-inexact-local-linear} shows that the actual restart
candidate differs from its exact comparison candidate by at most
$2\rho E_n$ in the $\mathcal M_n$-metric.  Whenever the exact comparison
trajectory lies in $U$, combining this estimate with
Theorem~\ref{thm:exact-local-linear} gives the perturbed recurrence
\begin{equation}
\label{eq:inexact-recurrence}
    d_{n+1}
    \le
    qd_n+2\rho\chi E_n.
\end{equation}

For a relative-error parameter $\eta\ge0$, set
\[
    q_\eta
    :=
    q+2\rho\chi\eta,
    \qquad
    R_\eta
    :=
    \frac{2(1+\rho\eta)d_0}
         {\sqrt{m_-}(1-q_\eta)}.
\]

\begin{theorem}[local linear convergence with inexact proximal steps]
\label{thm:inexact-local-linear}
Suppose Assumptions~\ref{ass:uniform-metrics} and
\ref{ass:fp-sharpness} hold, $N_n\ge N_{\min}\ge2$, and $q<1$ as in
Theorem~\ref{thm:exact-local-linear}.  If
\begin{equation}
\label{eq:rh-relative-epoch-error}
    E_n\le\eta d_n,
    \qquad
    q_\eta<1,
\end{equation}
and $\overline B_{R_\eta}(a^0)\subseteq U$, then
\begin{equation}
\label{eq:inexact-relative-rate}
    d_n\le q_\eta^nd_0,
    \qquad n\ge0.
\end{equation}
In particular, the KKT distance converges Q-linearly to zero and the anchors
converge R-linearly to a KKT point.
\end{theorem}

The complete proof, together with the epoch-stability lemma on which it
relies, is deferred to Appendix~\ref{app:proof-inexact-local-linear}.  The same
recurrence also covers a geometric error budget $E_n\le D\vartheta^n$: the
anchors then converge R-linearly, with the usual convolution rate determined
by $q$ and $\vartheta$, provided the resulting trajectories remain in $U$.

The relative-error condition \eqref{eq:rh-relative-epoch-error} contains two
unobservable quantities: $E_n$ aggregates distances to unknown exact
proximal points, whereas $d_n$ is the distance to the unknown KKT set.  The
next subsection replaces them with computable certificates.  Stationarity
residuals or objective gaps control the individual operator errors and hence
$E_n$; a movement-based tolerance then scales these residuals by observable
primal-dual displacements and, under a block-comparability condition,
recovers a relative bound of the form $E_n\le\eta_\gamma d_n$.  This turns the
abstract hypothesis of Theorem~\ref{thm:inexact-local-linear} into a
verifiable inner stopping rule.

\subsubsection{Accuracy certificates}

For state $z=(x,y)$, write $v_n(z)=x+\tau_nA^*y$ and define
\begin{equation}
\label{eq:theory-inner-objective}
    H_{n,z}(u)
    :=
    \frac12\langle u,Qu\rangle+\langle c,u\rangle
    +\frac{1}{2\tau_n}\|u-v_n(z)\|^2.
\end{equation}
Its unique minimizer over $K$ is $p_n(z)$.  Set
\[
    H_n:=Q+\tau_n^{-1}I,
    \qquad
    \mu_n:=\lambda_{\min}(H_n)>0,
    \qquad
    L_n:=\lambda_{\max}(H_n).
\]
If an inner algorithm returns $\widetilde p\in K$ and supplies a stationarity
residual
\begin{equation}
\label{eq:general-stationarity-residual}
    r
    \in
    Q\widetilde p+c+\tau_n^{-1}(\widetilde p-v_n(z))
    +N_K(\widetilde p),
\end{equation}
then strong monotonicity gives
\begin{equation}
\label{eq:general-residual-to-error}
    \|\widetilde p-p_n(z)\|
    \le
    \mu_n^{-1}\|r\|.
\end{equation}
Alternatively, if
$H_{n,z}(\widetilde p)-H_{n,z}(p_n(z))\le\Delta$, strong convexity yields
\[
    \|\widetilde p-p_n(z)\|
    \le
    \sqrt{\frac{2\Delta}{\mu_n}}.
\]
Together with \eqref{eq:prox-error-to-operator-error}, these estimates convert
stationarity residuals or objective gaps directly into bounds on
$\|e_{n,k}\|_{\mathcal M_n}$ and therefore on $E_n$.  Thus the outer theorem
is independent of the inner algorithm: first- or second-order methods,
active-set methods, and splitting methods are covered whenever their returned
points and certificates satisfy the stated error budget.

\paragraph{Solver-independent movement-based tolerance.}
Suppose the chosen inner solver returns residuals $r_{n,k}$ satisfying
\eqref{eq:general-stationarity-residual}.  For the inexact
reflected-Halpern states, define the full primal-dual movement
\begin{equation}
\label{eq:theory-full-movement}
    \mathfrak m_{n,k}
    :=
    \|\widehat h^{n,k}-\widehat h^{n,k-1}\|_{\mathcal M_n},
    \qquad
    k=1,\ldots,N_n-1.
\end{equation}
Partition the evaluations into blocks $\{\mathcal B_j\}$, and let
$\bar m_{n,j}$ be the stale movement statistic used on block
$\mathcal B_j$.  Assume
\begin{equation}
\label{eq:theory-block-comparability}
    \sum_j|\mathcal B_j|\bar m_{n,j}
    \le
    C_B\left(
        d_n+\sum_{k=1}^{N_n-1}\mathfrak m_{n,k}
    \right)
\end{equation}
and accept an inner solve whenever
\begin{equation}
\label{eq:theory-periodic-residual}
    \|r_{n,k}\|
    \le
    \tau_n\gamma\bar m_{n,j},
    \qquad
    k\in\mathcal B_j.
\end{equation}
Uniform metric and stepsize bounds, together with
\eqref{eq:general-residual-to-error} and
\eqref{eq:prox-error-to-operator-error}, first give
\[
    E_n
    \le
    C_r\gamma\left(
        d_n+\sum_{k=1}^{N_n-1}\mathfrak m_{n,k}
    \right).
\]
Assume also that $N_n\le N_{\max}<\infty$.  Then the finite
reflected-Halpern composition and the one-epoch stability estimate give
\[
    \sum_{k=1}^{N_n-1}\mathfrak m_{n,k}
    \le
    C_m(d_n+E_n).
\]
After absorbing constants,
$E_n\le C_0\gamma(d_n+E_n)$.  Hence, whenever $C_0\gamma<1$,
\begin{equation}
\label{eq:theory-periodic-absorbed}
    E_n
    \le
    \eta_\gamma d_n,
    \qquad
    \eta_\gamma
    :=
    \frac{C_0\gamma}{1-C_0\gamma}.
\end{equation}
Choosing $\gamma$ sufficiently small that
$q+2\rho\chi\eta_\gamma<1$ makes
\eqref{eq:theory-periodic-absorbed} a valid relative-error budget in
Theorem~\ref{thm:inexact-local-linear} and preserves local linear convergence.
The latter conclusion is understood under the corresponding retention
condition of that theorem.

\subsubsection{An illustrative example: Projected Gradient}

Projected gradient is one convenient way to produce the required inexact
proximal point, but it is not required by the preceding theorem.  With
\[
    \eta_{{\rm PG},n}
    :=
    \frac{2}{L_n+\mu_n},
    \qquad
    \rho_{{\rm PG},n}
    :=
    \frac{L_n-\mu_n}{L_n+\mu_n},
    \qquad
    \rho_{\rm PG}:=\sup_n\rho_{{\rm PG},n}<1,
\]
the projected-gradient map is uniformly contractive.

\begin{proposition}[fixed projected-gradient work]
\label{prop:fixed-pg-work}
Suppose Assumptions~\ref{ass:uniform-metrics} and
\ref{ass:fp-sharpness} hold, $q<1$, and
$N_{\min}\le N_n\le N_{\max}<\infty$.
Warm-start each proximal subproblem at the current reflected-Halpern primal state,
$u^{n,k,0}=\widehat h_x^{n,k}$.  There are local constants
$C_{\rm op},C_w>0$, independent of $n$, $k$, and $J\ge1$, such that exactly
$J$ projected-gradient steps per proximal problem give
\begin{equation}
\label{eq:theory-fixed-j-error}
    E_n
    \le
    C_{\rm op}N_{\max}C_w\rho_{\rm PG}^Jd_n.
\end{equation}
Consequently, with
\begin{equation}
\label{eq:fixed-pg-contraction}
    q_J
    :=
    q+
    2\rho\chi C_{\rm op}N_{\max}C_w\rho_{\rm PG}^J,
\end{equation}
any fixed $J$ satisfying $q_J<1$ gives local Q-linear convergence of the KKT
distance and R-linear convergence of the anchors, provided
$\overline B_{R_{\eta_J}}(a^0)\subseteq U$, where
$\eta_J:=C_{\rm op}N_{\max}C_w\rho_{\rm PG}^J$.
\end{proposition}

The proof is deferred to Appendix~\ref{app:proof-fixed-pg-work}.
Theorem~\ref{thm:inexact-local-linear} is an oracle result: its proof does not
use the update formula of a particular inner solver.  Projected gradient is
only one example for which the conditions can be verified explicitly; the
theorem does not claim convergence for arbitrary uncontrolled errors.  The
algorithm always acts on the original conic QP \eqref{eq:cqp}.  The rotated-SOC
lifting in Section~\ref{sec:conic-geometry} is solely an analytical device for
verifying smoothed-gap quadratic growth.
\fi

\section{Uniform Local Quadratic Growth}
\label{sec:conic-geometry}

\ifdefined\AVERAGEDPDHG
Section~\ref{sec:averaged-restart-theory} showed that uniform local
quadratic growth of the smoothed gap is sufficient for local linear
convergence.  This section characterizes that property through explicit
primal and dual error bounds.  The proof first establishes the product
geometry of the KKT set and an exact decomposition of the smoothed gap, which
identifies separate primal and dual error-bound certificates.  It then uses a
RSOC lifting to verify the primal error bound and a weighted projection
residual to verify the dual error bound.  Combining these two bounds yields
uniform local quadratic growth.  The lifting is used only to justify a primal
error bound.  Algorithm~\ref{alg:pdhcq1} continues to operate directly on the
original conic QP \eqref{eq:cqp}.  All proofs of this section are deferred to Appendix~\ref{app:sec3-proofs}.  Throughout this section, $\xi>0$ is fixed.
\else
Section~\ref{sec:rh-restart-theory} establishes local linear convergence of
restarted reflected-Halpern PDHG from fixed-point residual sharpness.  This
section develops a complementary geometric result: it characterizes uniform
local quadratic growth of the smoothed primal-dual gap through explicit
primal and dual error bounds.  The two properties are related through the
regularity of the KKT system, but smoothed-gap quadratic growth is not used as
a substitute for fixed-point sharpness in Section~\ref{sec:rh-restart-theory}.
The proof first establishes the product geometry of the KKT set and an exact
decomposition of the smoothed gap, which identifies separate primal and dual
error-bound certificates.  It then uses an RSOC lifting to verify the
primal error bound and a weighted projection residual to verify the dual
error bound.  Combining these two bounds yields uniform local quadratic
growth.
The lifting is used only to justify a primal error bound.  Algorithm
\ref{alg:pdhcq1} continues to operate directly on the original conic QP
\eqref{eq:cqp}.  All proofs of this section are deferred to Appendix~\ref{app:sec3-proofs}.

For completeness, recall the gap function used in this section.  For
$z=(x,y)\in K\times Y$ and
$\widehat z=(\widehat x,\widehat y)\in K\times Y$, the two-point
primal-dual gap is
\begin{equation}
\label{eq:ordinary-duality-gap}
    \mathcal Q(z,\widehat z)
    :=
    \mathcal L(x,\widehat y)-\mathcal L(\widehat x,y).
\end{equation}
The ordinary gap obtained by maximizing this expression may be infinite on
unbounded domains.  Following the smoothed-gap framework of
\cite{fercoq2022quadratic}, and its use in restarted first-order QP methods
\cite{lu2023practical,huang2025restarted}, fix $\xi>0$ and define
\begin{equation}
\label{eq:smoothed-gap}
    G_\xi(z;\dot z)
    :=
    \sup_{\widehat x\in K,\ \widehat y\in Y}
    \left\{
        \mathcal Q(z,\widehat z)
        -\frac{\xi}{2}\|\widehat x-\dot x\|^2
        -\frac{\xi}{2}\|\widehat y-\dot y\|^2
    \right\},
\end{equation}
where $\dot z=(\dot x,\dot y)\in E\times Y$ is the prescribed center.  The
quadratic penalty makes the supremum finite and permits a local quadratic
error bound.
\fi

\subsection{Exact smoothed-gap decomposition}

For a fixed $x^\star\in X^\star$, introduce the affine slack set and the
complementary face
\[
    \mathcal D_Q
    :=Qx^\star+c-\operatorname{range}A^*,
    \qquad
    \mathcal F_{x^\star}
    :=K^*\cap(x^\star)^\perp.
\]

\begin{lemma}[geometry of the KKT solution set]
\label{lem:kkt-solution-geometry}
Every $x^\star\in X^\star$ can be paired with every
$y^\star\in Y^\star$.  Consequently,
\begin{equation}
\label{eq:solution-product}
    Z^\star=X^\star\times Y^\star,
    \qquad
    \operatorname{dist}^2((x,y),Z^\star)
    =
    \operatorname{dist}^2(x,X^\star)
    +
    \operatorname{dist}^2(y,Y^\star).
\end{equation}
Then $\mathcal D_Q$ is independent of the chosen $x^\star$, and the optimal
slack set is
\begin{equation}
\label{eq:optimal-slack-intersection}
    \mathcal S^\star
    =
    \mathcal D_Q\cap\mathcal F_{x^\star}.
\end{equation}
\end{lemma}

Fix a KKT center $z^\star=(x^\star,y^\star)$ and set
$s^\star=Qx^\star+c-A^*y^\star$.  Define
\[
    P_\xi(x;z^\star)
    :=
    G_\xi((x,y^\star);z^\star),
    \qquad
    D_\xi(y;z^\star)
    :=
    G_\xi((x^\star,y);z^\star).
\]
These are the primal and dual components of the standard smoothed-gap
decomposition \cite[Lemma~5]{lu2023practical}, i.e.,
$$
G_\xi(z;z^\star) = P_\xi(x;z^\star) + D_\xi(y;z^\star).
$$
The exact identities needed
below are also derived in the proof of Theorem~\ref{thm:qg}.
Fix a reference point $\bar z=(\bar x,\bar y)\in Z^\star$.  All local error
bounds and regularity assumptions below are imposed near this reference
point, and their constants are required to be uniform over all KKT centers
in the local solution stratum.

For the primal component, define the complementarity-curvature error
\begin{equation}
\label{eq:primal-curvature-error}
    E_P(x;z^\star)
    :=
    \frac12\|x-x^\star\|_Q^2+\langle s^\star,x\rangle.
\end{equation}
Both terms are nonnegative for $x\in K$.
\begin{definition}[uniform local primal error bound]
\label{def:primal-eb}
We say that the primal error bound holds uniformly locally around $\bar z$
if there are neighborhoods $V_P$ of $\bar z$ and $V_x$ of $\bar x$, and
constants $a_P,b_P>0$, such that, for every
$z^\star=(x^\star,y^\star)\in Z^\star\cap V_P$ and every
$x\in K\cap V_x$,
\begin{equation}
\label{eq:primal-eb}
    \operatorname{dist}^2(x,X^\star)
    \le
    a_PE_P(x;z^\star)
    +b_P\|Ax-b\|^2.
\end{equation}
\end{definition}

For the dual component, set
\[
    M_\xi:=Q+\xi I\succ0,
    \qquad
    s^\star(y):=Qx^\star+c-A^*y,
\]
and let
\[
    \Pi_K^{M_\xi}(v)
    :=
    \arg\min_{u\in K}\frac12\|u-v\|_{M_\xi}^2.
\]
The corresponding projection point and residual are
\[
    u_y
    :=
    \Pi_K^{M_\xi}
    \left(x^\star-M_\xi^{-1}s^\star(y)\right),
    \qquad
    r_y:=x^\star-u_y
    =:R_\xi(y;z^\star).
\]
\begin{definition}[uniform local dual error bound]
\label{def:dual-eb}
We say that the dual error bound holds uniformly locally around $\bar z$
if there are neighborhoods $V_D$ of $\bar z$ and $V_y$ of $\bar y$, and a
constant $\kappa_D>0$, such that, for every
$z^\star=(x^\star,y^\star)\in Z^\star\cap V_D$ and every
$y\in Y\cap V_y$,
\begin{equation}
\label{eq:dual-eb}
    \operatorname{dist}(y,Y^\star)
    \le
    \kappa_D\|R_\xi(y;z^\star)\|_{M_\xi}.
\end{equation}
\end{definition}

\begin{theorem}[characterization of uniform local quadratic growth]
\label{thm:qg}
Uniform local quadratic growth of the smoothed gap holds around $\bar z$;
that is, there are a neighborhood $V$ of $\bar z$ and a constant
$\alpha_\xi>0$ such that
\begin{equation}
\label{eq:uniform-gap-qg}
    G_\xi(z;z^\star)
    \ge
    \alpha_\xi\operatorname{dist}^2(z,Z^\star)
\end{equation}
for every $z\in(K\times Y)\cap V$ and $z^\star\in Z^\star\cap V$, if and
only if the uniform local primal and dual error bounds in
Definitions~\ref{def:primal-eb} and \ref{def:dual-eb} both hold around
$\bar z$.
\end{theorem}

Theorem~\ref{thm:qg} reduces uniform local quadratic growth exactly to the
two error bounds.  The following subsection verifies them under the
corresponding one-sided strict-complementarity conditions.

\subsection{Primal and dual error bound under strict complementarity}
\label{sec:one-sided-sc-error-bounds}

\paragraph{Primal and dual strict complementarity.}
Let $(x^\star,s^\star)\in K\times K^*$ be a complementary pair.  We say
that it satisfies, respectively,
\[
\begin{aligned}
    \text{primal strict complementarity (P-SC)}
    &:\quad
    x^\star\in\operatorname{ri}\bigl(K\cap(s^\star)^\perp\bigr),\\
    \text{dual strict complementarity (D-SC)}
    &:\quad
    s^\star\in\operatorname{ri}\bigl(K^*\cap(x^\star)^\perp\bigr).
\end{aligned}
\]
Two-sided, or facial, strict complementarity means that both conditions
hold; for product cones, these definitions are understood blockwise.  Our
terminology names the component required to lie in a relative interior.
The terminology of \cite[Definition~2]{ding2023strict} is reversed: their
``dual strict complementarity'' is P-SC above.  We therefore state the
relevant side explicitly whenever that result is invoked.

For $K=\mathbb R_+^n$, complementarity makes P-SC and D-SC equivalent to
$x_i^\star+s_i^\star>0$ for every $i$.  The Goldman-Tucker theorem \cite{goldman1956theory} guarantees the
existence of such an optimal pair for LP whenever both optimal sets are
nonempty, but not that every optimal pair is strictly complementary. More generally, D-SC implies P-SC when the minimal face of $K$ containing
$x^\star$ is exposed; the reverse implication follows from the analogous
property of $K^*$.  Hence the two sides coincide when both $K$ and $K^*$ are
facially exposed \cite{chua2008invariance}.  In particular, for symmetric
cones they are equivalent to
$x^\star+s^\star\in\operatorname{int}K$.  This covers nonnegative,
second-order-cone (SOC) and rotated-SOC, and PSD cones.

For cones with nonexposed faces, P-SC and D-SC may differ; the exponential
cone is a standard example \cite{lindstrom2023exponential}. P-SC and D-SC concern a particular complementary pair.  The uniform bounds
below require the relevant condition over a local certificate family:
the primal error bound is verified under P-SC at every
local KKT center, whereas
Corollary~\ref{cor:dual-sc-slack-regularity} requires each local primal
center to admit a D-SC optimal slack.  These conditions are distinct from
primal and dual nondegeneracy \cite{alizadeh1997complementarity}.

\subsubsection{Primal error bound from a rotated-SOC lifting}

\paragraph{Rotated-SOC lifting.}
We verify the primal error bound in Definition~\ref{def:primal-eb} through an
epigraph lifting.  Choose $B$ such that $Q=B^*B$, and define the rotated SOC as 
\[
    \mathcal Q_r
    :=
    \{(u,v,w):u\ge0,\ v\ge0,\ 2uv\ge\|w\|^2\}.
\]
Since $u\ge\frac12\|Bx\|^2$ is equivalent to $(u,1,Bx)\in\mathcal Q_r$,
problem \eqref{eq:cqp} is equivalent to
\begin{equation}
\label{eq:soc-lifting}
\begin{aligned}
    \min_{x,u,v,w}\quad
        &u+\langle c,x\rangle\\
    \mathrm{s.t.}\quad
        &Ax=b,\qquad x\in K,\\
        &v=1,\qquad w=Bx,\\
        &(u,v,w)\in\mathcal Q_r.
\end{aligned}
\end{equation}
Write $\chi:=(x,u,v,w)$, denote the optimal solution set of
\eqref{eq:soc-lifting} by $\mathcal X_L^\star$, and collect the lifted
equality constraints as
\[
    \mathcal A_L(x,u,v,w):=(Ax,v,w-Bx),
    \qquad
    \widetilde b:=(b,1,0).
\]
At every optimum the epigraph inequality is tight.  Attaching multipliers
$(y,\alpha,t)$ to the equality residuals $b-Ax$, $1-v$, and $Bx-w$ shows
that the lifted dual slacks have the form
\[
    p=(1,-\alpha,-t)\in\mathcal Q_r,
\]
with lifted dual objective $\langle b,y\rangle+\alpha$.  Finally, for every
$x\in K$, define the lifted comparison point
\[
    q(x):=\left(\frac12\|Bx\|^2,1,Bx\right),
    \qquad
    \widehat\chi(x):=(x,q(x)).
\]

\paragraph{Transfer of strict complementarity.}
For any KKT center $z^\star=(x^\star,y^\star)$, 
set
$s^\star:=Qx^\star+c-A^*y^\star$, and define the canonical lift
\[
    u^\star:=\frac12\|Bx^\star\|^2,
    \qquad
    q^\star:=(u^\star,1,Bx^\star),
    \qquad
    t^\star:=Bx^\star,
    \qquad
    \alpha^\star:=-u^\star,
    \qquad
    p^\star:=(1,u^\star,-Bx^\star),
\]
together with
\[
    \chi^\star:=(x^\star,q^\star),\qquad
    \lambda^\star:=(y^\star,\alpha^\star,t^\star),\qquad
    \zeta^\star:=(s^\star,p^\star).
\]
The point $\chi^\star$ is primal feasible, $\zeta^\star$ is the lifted
slack at $\lambda^\star$ and is dual feasible, and
\[
    \langle\chi^\star,\zeta^\star\rangle
    =
    \langle x^\star,s^\star\rangle
    +2u^\star-\|Bx^\star\|^2
    =0,
\]
so $(\chi^\star,\lambda^\star,\zeta^\star)$ is a lifted KKT certificate. Besides, the added RSOC
block is always strictly complementary because
\[
    q^\star+p^\star
    =(u^\star+1,u^\star+1,0)
    \in\operatorname{ri}\mathcal Q_r.
\]
Consequently, the lifted center satisfies P-SC exactly when the original
center does since
\begin{align*}
    \chi^\star = (x^\star,q^\star)   \in \operatorname{ri}\bigl(K \cap(s^\star)^\perp\bigr) \times \operatorname{ri}\bigl(\mathcal Q_r \cap(p^\star)^\perp\bigr) = \operatorname{ri}\bigl((K\times\mathcal Q_r)\cap(s^\star,p^\star)^\perp\bigr),
\end{align*}
where the last equality holds because relative interiors commute with products. Similarly, D-SC also transfers
blockwise between the original and the lifted certificates.  

The verification is organized around the following residual error bound for
the lifted problem.

\begin{assumption}[uniform lifted residual error bound]
\label{ass:lifted-residual-eb}
There are a neighborhood $V_{\mathrm{SC}}$ of $\bar z$, a common open
neighborhood $\mathcal N$ of $\widehat\chi(\bar x)$, and constants
$\kappa_L,\gamma_L>0$ such that, for every
$z^\star\in Z^\star\cap V_{\mathrm{SC}}$ with canonical lifted slack
$\zeta^\star=(s^\star,p^\star)$,
\begin{equation}
\label{eq:lifted-residual-error-bound}
    \operatorname{dist}(\chi,\mathcal X_L^\star)
    \le
    \kappa_L\langle\zeta^\star,\chi\rangle^{1/2}
    +\gamma_L\|\mathcal A_L\chi-\widetilde b\|,
    \qquad
    \chi\in(K\times\mathcal Q_r)\cap\mathcal N.
\end{equation}
\end{assumption}

\begin{proposition}[verification of the lifted residual error bound]
\label{prop:lifted-eb}
Let $K$ be a finite Cartesian product of nonnegative cones, second-order
cones (including rotated second-order cones), and positive-semidefinite
cones.  If every $z^\star\in Z^\star\cap V_{\mathrm{SC}}$ satisfies P-SC,
for some neighborhood $V_{\mathrm{SC}}$ of $\bar z$, then
Assumption~\ref{ass:lifted-residual-eb} holds.
\end{proposition}

Proposition~\ref{prop:lifted-eb} covers the rotated block $\mathcal Q_r$
because the rotated second-order cone is the image of the standard
second-order cone $\{(t,\xi):t\ge\|\xi\|\}$ under the orthogonal map
$(t,\xi_1,\xi_2)\mapsto
\bigl((t+\xi_1)/\sqrt2,\,(t-\xi_1)/\sqrt2,\,\xi_2\bigr)$, a $45^\circ$
rotation of the $(t,\xi_1)$-plane.  Orthogonal maps preserve distances,
faces, and relative interiors, so properties of second-order cones---in
particular the error bounds established in \cite{ding2023strict}---hold
verbatim, with identical constants, for rotated second-order-cone blocks.

The corollary below therefore only needs to compare the original and
lifted solution-set distances.

\begin{corollary}[verification of the uniform local primal error bound]
\label{cor:primal-eb-strict-complementarity}
Under Assumption~\ref{ass:lifted-residual-eb}, the uniform local primal
error bound in Definition~\ref{def:primal-eb} holds with
\[
    a_P=2\kappa_L^2,
    \qquad
    b_P=2\gamma_L^2.
\]
\end{corollary}

The linear lifted feasibility residual in
\eqref{eq:lifted-residual-error-bound} is essential: a square-root residual
would yield $\|Ax-b\|$, rather than $\|Ax-b\|^2$, after squaring and would
not imply Definition~\ref{def:primal-eb}.

\subsubsection{Dual error bound}

We next verify the dual error bound in Definition~\ref{def:dual-eb}.  The
weighted projection defining $R_\xi(y;z^\star)$ yields
$v_y:=M_\xi r_y-s^\star(y)\in N_K(u_y)$.  Normal-cone calmness and
$v_y+s^\star(y)=M_\xi r_y$ control the distance to the complementary face;
linear regularity then gives the distance to the optimal slack set, and a
pseudoinverse of $A^*$ lifts this estimate to the multiplier space.  The key
identity is
\[
    \mathcal S^\star
    =
    \mathcal D_Q\cap\mathcal F_{x^\star}
\]
so we first state the two local regularity properties used in this chain.

\paragraph{Linear regularity.}
Let $C_1$ and $C_2$ be nonempty closed convex sets and let
$\bar s\in C_1\cap C_2$.  The pair $\{C_1,C_2\}$ is locally linearly
regular at $\bar s$ if there are a neighborhood $W$ of $\bar s$ and a
constant $\kappa>0$ such that
\begin{equation}
\label{eq:general-linear-regularity}
    \operatorname{dist}(s,C_1\cap C_2)
    \le
    \kappa\left[
        \operatorname{dist}(s,C_1)+\operatorname{dist}(s,C_2)
    \right],
    \qquad s\in W.
\end{equation}
Uniform local linear regularity means that the neighborhood and modulus can
be chosen independently of the parameter.  Applied to the two sets defining
$\mathcal S^\star$, it gives the following assumption.

\begin{assumption}[uniform local slack regularity]
\label{ass:slack-regularity}
Set $\bar s:=Q\bar x+c-A^*\bar y$.  There are neighborhoods $V_A$ of
$\bar z$ and $W_A$ of $\bar s$, and a constant $\kappa_A>0$, such that,
for every $z^\star=(x^\star,y^\star)\in Z^\star\cap V_A$ and every
$s\in W_A$,
\begin{equation}
\label{eq:slack-regularity}
    \operatorname{dist}(s,\mathcal S^\star)
    \le
    \kappa_A\left[
        \operatorname{dist}(s,\mathcal D_Q)
        +
        \operatorname{dist}(s,\mathcal F_{x^\star})
    \right].
\end{equation}
\end{assumption}

This is the uniform local linear regularity of
$\{\mathcal D_Q,\mathcal F_{x^\star}\}$ along the local KKT stratum.  Since
$s^\star(y)\in\mathcal D_Q$, the first distance in
\eqref{eq:slack-regularity} vanishes along the residual path.  D-SC gives a
direct verification of this property.

\begin{corollary}[dual strict complementarity implies slack regularity]
\label{cor:dual-sc-slack-regularity}
Fix a KKT center $z^\star=(x^\star,y^\star)$.  If its optimal slack satisfies
D-SC,
\begin{equation}
\label{eq:dual-side-strict-complementarity}
    s^\star\in\operatorname{ri}\mathcal F_{x^\star},
\end{equation}
then $\{\mathcal D_Q,\mathcal F_{x^\star}\}$ is locally linearly regular at
$s^\star$,
so the corresponding pointwise slack-regularity bound holds.  If every
$x^\star\in X^\star\cap V_{X,\mathrm{SC}}$, for some neighborhood
$V_{X,\mathrm{SC}}$ of $\bar x$, admits a D-SC optimal slack, then
Assumption~\ref{ass:slack-regularity} holds uniformly on the corresponding
local KKT stratum.
\end{corollary}

We next impose the normal-cone calmness used in this projection estimate.

\begin{assumption}[uniform normal-cone calmness]
\label{ass:normal-calmness}
There are neighborhoods $V_B$ of $\bar z$ and $W_B$ of
$(\bar x,-\bar s)$, and a constant $\kappa_B>0$, such that, for every
$z^\star=(x^\star,y^\star)\in Z^\star\cap V_B$ and every
$(u,v)\in\operatorname{gph}N_K\cap W_B$,
\begin{equation}
\label{eq:normal-calmness}
    \operatorname{dist}(v,N_K(x^\star))
    \le
    \kappa_B\|u-x^\star\|.
\end{equation}
\end{assumption}

A standard pointwise sufficient condition is $C^2$-cone reducibility.  A
closed convex set $\Omega$ is $C^2$-cone reducible at $\bar x\in\Omega$ if,
locally, it has the representation
\[
    \Omega=\{x:\ \Xi(x)\in\mathcal C\},
\]
where $\Xi$ is twice continuously differentiable, $\Xi(\bar x)=0$,
$D\Xi(\bar x)$ is surjective, and $\mathcal C$ is a pointed closed convex cone.
If $K$ is $C^2$-cone reducible at $x^\star$, then $N_K$ is calm at every
normal graph point, including $(x^\star,-s^\star)$
\cite[Theorem~2.1]{liu2019graphical}.  This class includes polyhedral,
Lorentz, positive-semidefinite, and Ky Fan $k$-norm epigraph cones, and is
preserved by finite products.  For polyhedral cones, calmness also follows
directly from piecewise polyhedrality
\cite{robinson1981continuity}; for exponential and power cones it must be
verified separately unless a suitable reducibility result is available.

Table~\ref{tab:dual-regularity-verification} summarizes these pointwise
verification routes; rotated SOCs are included with Lorentz cones.

\begin{table}[H]
\centering
\small\setstretch{1.0}
\caption{Typical pointwise verification routes for the two dual regularity
conditions.}
\label{tab:dual-regularity-verification}
\renewcommand{\arraystretch}{1.12}
\begin{tabularx}{\textwidth}{>{\raggedright\arraybackslash}p{0.19\textwidth}
    >{\raggedright\arraybackslash}X
    >{\raggedright\arraybackslash}X}
\toprule
Cone class
& Slack-set linear regularity
& Normal-cone calmness \\
\midrule
Polyhedral cones, including $\mathbb R_+^n$
& Hoffman bound; often global.
& Polyhedral multifunction \cite{robinson1981continuity}. \\
Lorentz/SOC and rotated SOC
& D-SC:
  $s^\star\in\operatorname{ri}\mathcal F_{x^\star}$.
& $C^2$-cone reducibility \cite{liu2019graphical}. \\
Positive-semidefinite cone
& D-SC, equivalently the rank condition.
& $C^2$-cone reducibility \cite{liu2019graphical}. \\
Finite Cartesian products
& Blockwise D-SC.
& Reducibility preserved by products \cite{liu2019graphical}. \\
General closed convex cone
& Verify linear regularity directly.
& Verify calmness directly. \\
\bottomrule
\end{tabularx}
\end{table}

The table gives pointwise conditions.  If each local primal center admits a
D-SC optimal slack, Corollary~\ref{cor:dual-sc-slack-regularity} makes slack
regularity uniform.  The next proposition verifies uniform normal-cone
calmness in the same regime for the common cone classes.

\begin{proposition}[verification of uniform normal-cone calmness]
\label{prop:uniform-normal-calmness}
Let $K$ be a finite Cartesian product of nonnegative cones, second-order
cones (including rotated second-order cones), and positive-semidefinite
cones.  If every $x^\star\in X^\star\cap V_{X,\mathrm{SC}}$, for some
neighborhood $V_{X,\mathrm{SC}}$ of $\bar x$, admits a blockwise D-SC
optimal slack, then Assumption~\ref{ass:normal-calmness} holds.
\end{proposition}

Without the constant-face structure provided by D-SC, the calmness modulus
can blow up as centers approach a smaller face, so the stratification is
what makes uniformity available.  We now combine the two properties.

\begin{proposition}[verification of the uniform local dual error bound]
\label{prop:dual-eb}
Under Assumptions~\ref{ass:slack-regularity} and
\ref{ass:normal-calmness}, the uniform local dual error bound in
Definition~\ref{def:dual-eb} holds with constant
\[
    \kappa_D
    =
    \|(A^*)^\dagger\|\,\kappa_A
    \frac{L_\xi+\kappa_B}{\sqrt{\lambda_\xi}},
\]
where $L_\xi:=\|M_\xi\| = \|Q\| + \xi$, $\lambda_\xi:=\lambda_{\min}(M_\xi) \geq \xi$, and the
pseudoinverse norm is restricted to $\operatorname{range}A^*$.
\end{proposition}

\subsection{Quadratic growth for common-cone QPs under strict complementarity}

The preceding verification routes yield the following direct consequence
for the cone classes most commonly used in conic optimization.  The
uniformity clause is stated explicitly because the corresponding result at
a single strictly complementary KKT point is only pointwise.

\begin{corollary}[common-cone conic QPs]
\label{cor:common-cone-qg}
Fix $\xi>0$, and suppose that $K$ is a finite Cartesian product of
nonnegative cones, Lorentz cones (including rotated Lorentz cones), and
positive-semidefinite cones.  Let $\bar z\in Z^\star$, and assume strict
complementarity along the local KKT stratum: there is a neighborhood
$V_{\mathrm{SC}}$ of $\bar z$ such that every
$z^\star=(x^\star,y^\star)\in Z^\star\cap V_{\mathrm{SC}}$ satisfies
$x^\star+s^\star\in\operatorname{int}K$.
Then,
there are a neighborhood $V$ of $\bar z$ and a constant $\alpha_\xi>0$ such
that
\[
    G_\xi(z;z^\star)
    \ge
    \alpha_\xi\operatorname{dist}^2(z,Z^\star)
\]
for all $z\in(K\times Y)\cap V$ and
$z^\star\in Z^\star\cap V$. With the constants $\kappa_L$ and $\gamma_L$
from Assumption~\ref{ass:lifted-residual-eb}, $\kappa_A$ from Assumption~\ref{ass:slack-regularity}, and $\kappa_B$ from Assumption~\ref{ass:normal-calmness}, one may take
\begin{equation}
\label{eq:common-cone-qg-constant}
\begin{aligned}
    \alpha_\xi
    &=
    \min\left\{
        \frac{1}{2\kappa_L^2},
        \frac{1}{4\xi\gamma_L^2},
        \frac{\xi}
        {2\|(A^*)^\dagger\|^2\kappa_A^2(\|Q\| + \xi+\kappa_B)^2}
    \right\}.
\end{aligned}
\end{equation}
\end{corollary}

\ifdefined\AVERAGEDPDHG
Theorem~\ref{thm:qg} verifies Assumption~\ref{ass:smoothed-qg}; it is not a
separate convergence argument.  The exact averaged PDHG output lies in
$K\times Y$ because every exact primal iterate lies in $K$ and $K$ is convex.
Therefore the theorem applies directly to the epoch average whenever it lies
in the stated neighborhood.
\else
Theorem~\ref{thm:qg} establishes a smoothed-gap quadratic-growth property; it
is not, by itself, a proof of Assumption~\ref{ass:fp-sharpness}.  In
particular, reflected-Halpern states may lie outside $K\times Y$, whereas the
growth estimate above is stated on $K\times Y$, and its merit function is not
the fixed-point residual $R_n$.  When the same conic regularity also yields
the KKT metric-subregularity estimate \eqref{eq:kkt-msr}, the resolvent
identity \eqref{eq:pdhg-residual-kkt} verifies fixed-point sharpness and the
local convergence theorems of Section~\ref{sec:rh-restart-theory} apply.
\fi

\section{Practical Implementation}
\label{sec:implementation}

This section describes the numerical components used to implement and
accelerate the PDHG iteration of Algorithm~\ref{alg:pdhcq1}.  We first discuss the
primal proximal problem, then present the reflected-Halpern scheme, problem
rescaling, inner-solver enhancements, adaptive restart, and the primal-weight
update.  Finally, we describe the distributed implementation for multi-GPU
acceleration.

The implementation accepts problems in the following standard form:
\begin{equation}
\label{eq:solver-input}
\begin{aligned}
    \min_{x\in\mathbb R^n}\quad
    &\frac12x^\top Qx+c^\top x\\
    \mathrm{s.t.}\quad
    &l^a\le Ax\le u^a,\\
    &x_{\mathcal I_0}\in
      \mathcal X_0:=[l^x,u^x],\\
    &x_{\mathcal I_j}\in\mathcal K_j,
      \qquad j=1,\ldots,J,
\end{aligned}
\end{equation}
Here $A\in\mathbb R^{m\times n}$, $Q\in\mathbb S_+^n$,
$c\in\mathbb R^n$, and
$\mathcal X:=\mathcal X_0\times\prod_{j=1}^J\mathcal K_j$.  Let
$\overline{\mathbb R}:=\mathbb R\cup\{-\infty,+\infty\}$.  The affine
bounds satisfy $l^a,u^a\in\overline{\mathbb R}^m$ with $l^a\le u^a$, and
the box bounds satisfy
$l^x,u^x\in\overline{\mathbb R}^{|\mathcal I_0|}$ with $l^x\le u^x$.
Thus equal finite endpoints represent equalities, while infinite endpoints
represent one-sided or absent bounds.  The index sets
$\{\mathcal I_j\}_{j=0}^J$ form a partition of $\{1,\ldots,n\}$, so every
primal coordinate belongs to exactly one block.

The implementation supports nonnegative blocks and the vector cones
\begin{align*}
    \mathcal K_{\mathrm{soc}}
    &:=\{(v,t):\|v\|_2\le t\},\\
    \mathcal K_{\mathrm{rsoc}}
    &:=\{(v,s,t):\|v\|_2^2\le2st,\ s,t\ge0\},\\
    \mathcal K_{\mathrm{exp}}
    &:=\operatorname{cl}\{(r,s,t):s>0,\ s\exp(r/s)\le t\},\\
    \mathcal K_{\mathrm{pow}}^\alpha
    &:=\{(r,s,t):r,s\ge0,\ r^\alpha s^{1-\alpha}\ge|t|\},
    \qquad \alpha\in(0,1).
\end{align*}
PSD cone blocks are excluded because every projection would require an
eigenvalue decomposition.

For a closed convex set $\mathcal C$ and a positive diagonal matrix
$M\succ0$, define the projection in the $M$-norm by
\begin{equation}
\label{eq:metric-block-projection}
    \Pi_{\mathcal C}^{M}(r)
    :=
    \arg\min_{u\in\mathcal C}\frac12\|u-r\|_M^2,
    \qquad
    \|z\|_M^2:=\langle z,Mz\rangle.
\end{equation}
We write $\Pi_{\mathcal C}:=\Pi_{\mathcal C}^{I}$ for the Euclidean
projection. Because $\mathcal X$ is a Cartesian product and $M$ is diagonal,
$\Pi_{\mathcal X}^M$ decomposes into independent box and cone projections.
Table~\ref{tab:cone-projection-methods} summarizes the projection methods for
the box and cone blocks, and Appendix~\ref{app:cone-projections} gives the
corresponding formulas under the Euclidean norm and the $M$-norm.

\subsection{Primal Proximal Problem Solving}
\label{sec:conic}
\label{sec:cone-projections}

Each PDHG iteration requires solving the following strongly convex primal
proximal subproblem, where $\tau_n$ is the current stepsize and
$v:=x+\tau_nA^*y$:
\begin{equation}
\label{eq:prox-subproblem}
    p^+
    =
    \arg\min_{u\in\mathcal X}
    \left\{
        \frac12\langle u,Qu\rangle+\langle c,u\rangle
        +\frac{1}{2\tau_n}\|u-v\|^2
    \right\},
\end{equation}
whose smooth Hessian $Q+\tau_n^{-1}I\succ0$ makes the solution unique even
for singular $Q$.  The solution method depends on whether $Q$ is diagonal or
a general positive-semidefinite operator.

\paragraph{General Quadratic Objective.}

For a general matrix-free $Q$, we solve \eqref{eq:prox-subproblem} by a
projected Barzilai-Borwein method \cite{dai2005projected}. Let
\begin{equation}
\label{eq:bb-inner-gradient}
    g^j
    :=Qu^j+c+\tau_n^{-1}(u^j-v).
\end{equation}
For $j\ge1$, define $s_j:=u^j-u^{j-1}$ and
$\delta g_j:=g^j-g^{j-1}$. The BB1 stepsize and projected update are
\begin{equation}
\label{eq:bb-step}
    \alpha_j^{\mathrm{BB}}
    :=\frac{\|s_j\|_2^2}{\langle s_j,\delta g_j\rangle},
    \qquad
    u^{j+1}
    =\Pi_{\mathcal X}\left(u^j-\alpha_j^{\mathrm{BB}}g^j\right).
\end{equation}
Each inner iteration requires one application of $Q$ and one product-set
projection. Products with a structured operator may be evaluated as
$Q=P+R^\top D R$, where $D$ is symmetric and $Q\succeq0$, without
materializing $Q$.

\paragraph{Diagonal Quadratic Objective.}

Suppose $Q=\operatorname{diag}(q)$ with $q_i\ge0$. More generally, let
$M_{\mathrm{sc}}\succ0$ denote the positive diagonal matrix induced by the
cone scaling described in Section~\ref{sec:problem-rescaling}; its cone-block
components are defined in \eqref{eq:scaling-induced-metric}.
When no cone scaling is applied, $M_{\mathrm{sc}}=I$. The resulting projection
metric and center are
\begin{equation}
\label{eq:diagonal-q-metric}
    M_n:=M_{\mathrm{sc}}+\tau_nQ,
    \qquad
    \widetilde x
    :=M_n^{-1}\bigl(M_{\mathrm{sc}}x-\tau_n(c-A^*y)\bigr).
\end{equation}
The proximal subproblem is therefore the single metric projection
\begin{equation}
\label{eq:diagonal-q-metric-projection}
    p^+=\Pi_{\mathcal X}^{M_n}(\widetilde x).
\end{equation}
Thus diagonal quadratic curvature and diagonal cone scaling are handled
by the same projection definition \eqref{eq:metric-block-projection}; no inner
projected-gradient iteration is required.

\subsection{Algorithm Enhancement}
\label{sec:algorithm-enhancement}

The following enhancements improve practical solution efficiency through
reflected-Halpern acceleration, problem rescaling, reduced inner work, timely
restarts, and primal-dual balance.

\subsubsection{Reflected-Halpern Acceleration}
\label{sec:rh-pdhg}

The practical solver applies an anchored reflected-Halpern acceleration to
the PDHG operator. This acceleration has demonstrated practical effectiveness
in cuPDLPx \cite{lu2025cupdlpx}; its convergence and nonergodic $O(1/k)$
residual bounds in a more general setting are established in
\cite{liu2026reflected}. During epoch $n$, let $\mathcal T_n$ denote the PDHG
update operator with fixed
$\tau_n,\sigma_n$, set the anchor $a^n=z^{n,0}$, and store the candidate
\begin{equation}
\label{eq:stored-pdhg-candidate}
    w^{n,k+1}:=\mathcal T_n(z^{n,k}).
\end{equation}
For a reflection coefficient $\rho\in[1/2,1]$ and $k=0,\ldots,N_n-2$, the
anchored update is
\begin{equation}
\label{eq:rh-update}
    z^{n,k+1}
    =
    \frac{k+1}{k+2}
    \bigl(2\rho w^{n,k+1}+(1-2\rho)z^{n,k}\bigr)
    +
    \frac{1}{k+2}a^n .
\end{equation}
The choice $\rho=1/2$ gives standard Halpern iteration on $\mathcal T_n$,
whereas $\rho=1$ gives Halpern iteration on the reflected operator
$2\mathcal T_n-I$.  A reflected-Halpern state need not
have a primal component in the cone.  Consequently, KKT residuals are
evaluated at the stored PDHG candidates, and an accepted epoch restarts from
the final candidate
\begin{equation}
\label{eq:restart-output}
    z^{n+1,0}
    :=
    w^{n,N_n}
    =
    \mathcal T_n(z^{n,N_n-1}),
\end{equation}
or from its inexact counterpart in the implementation.
Section~\ref{sec:rh-restart-theory} analyzes the restarted averaged scheme with
inexact primal proximal evaluations; the reflected-Halpern recursion is used
here as an implementation enhancement built on the same PDHG operator.

\subsubsection{Problem Rescaling}
\label{sec:problem-rescaling}

Following the preconditioned-instance convention of
cuPDLPx \cite{lu2025cupdlpx}, we combine $\ell_\infty$ Ruiz
rescaling \cite{ruiz2001scaling}, Pock-Chambolle
rescaling
\cite{pock2011diagonal}, and bound-objective rescaling. To describe
the accumulated diagonal transformation, let $S_x\succ0$ and $S_y\succ0$ be
the variable and constraint scalings and set $\widetilde x=S_xx$. The main
problem data become
\begin{equation}
\label{eq:problem-rescaling}
    \widetilde A=S_yAS_x^{-1},
    \qquad
    \widetilde Q=S_x^{-\top}QS_x^{-1},
    \qquad
    \widetilde c=S_x^{-\top}c,
\end{equation}
with all bounds transformed consistently. Bound-objective rescaling adds
positive scalar normalizations of the bound and objective data.

For primal cone blocks, we support two treatments of the variable scaling.
Let $\mathcal I_j$ index a cone block and, at the current rescaling stage, let
the coordinatewise Ruiz and Pock-Chambolle candidates be
\begin{equation}
\label{eq:cone-scaling-candidates}
    r_i
    :=\left(\max_{\ell}|\widetilde A_{\ell i}|\right)^{1/2},
    \qquad
    p_i
    :=\left(\sum_{\ell}|\widetilde A_{\ell i}|^{2-\alpha}\right)^{1/2}.
\end{equation}
The \emph{cone-preserving block scaling} mode replaces these candidates on
$\mathcal I_j$ by
\begin{equation}
\label{eq:cone-preserving-scaling}
    r_{\mathcal I_j}:=\max_{i\in\mathcal I_j}r_i,
    \qquad
    p_{\mathcal I_j}
    :=\left(\frac{1}{|\mathcal I_j|}
        \sum_{i\in\mathcal I_j}p_i^2\right)^{1/2},
\end{equation}
respectively. Thus Ruiz retains its blockwise $\ell_\infty$ safeguard, whereas
Pock-Chambolle aggregates the column energies by their root mean square. All
coordinates of a cone block receive the resulting common factor, so
$S_j=d_jI$ and $S_j\mathcal K_j=\mathcal K_j$. Nonconic variable scalings and
all constraint scalings remain coordinatewise.

Alternatively, the \emph{rescaled-cone} mode retains the candidates in
\eqref{eq:cone-scaling-candidates} coordinatewise. A cone block then becomes
$S_j\mathcal K_j$, where
$S_j=\operatorname{diag}((S_x)_{\mathcal I_j})$. This applies the Ruiz and
Pock-Chambolle scalings without blockwise aggregation and invokes the
rescaled-cone projections described by \cite{lin2025pdcs}. Equivalently, in
the original cone coordinates,
\begin{equation}
\label{eq:scaling-induced-metric}
    S_j^{-1}\Pi_{S_j\mathcal K_j}(S_jr)
    =\Pi_{\mathcal K_j}^{M_{\mathrm{sc},j}}(r),
    \qquad
    M_{\mathrm{sc},j}:=S_j^\top S_j.
\end{equation}
Constraint
scalings use the same coordinatewise rule in both modes. As in cuPDLPx,
termination criteria are evaluated after recovering the candidate for the
original, unpreconditioned problem, so rescaling does not affect the reported
accuracy.
The default rescaling settings are reported in
Table~\ref{tab:practical-defaults}.

\subsubsection{Adaptive Inner Accuracy}
\label{sec:adaptive-inner-accuracy}

For a general $Q$, let $\Delta u_{n,k}$ denote the final projected update of
the inner solve at outer step $(n,k)$. For $k\ge1$, the inner tolerance is
\begin{equation}
\label{eq:monotone-primal-inner-tol}
    \epsilon_{n,k}^{\mathrm{in}}
    =
    \min\left\{
        \epsilon_{n,k-1}^{\mathrm{in}},
        \max\left\{
            \gamma_{\mathrm{in}}
            \frac{\|x^{n,k}-x^{n,k-1}\|}{\tau_n},
            \epsilon_{\min}
        \right\}
    \right\},
\end{equation}
and the recorded stopping test is
\begin{equation}
\label{eq:current-inner-step-stop}
    \|\Delta u_{n,k}\|_2
    \le \epsilon_{n,k}^{\mathrm{in}}.
\end{equation}
By projection optimality, the final update supplies a stationarity residual
proportional to $\|\Delta u_{n,k}\|_2$, so
\eqref{eq:general-residual-to-error} converts this displacement into a bound
on the proximal error of the returned point.
The first tolerance of an epoch is inherited from the preceding epoch, with
$\epsilon_{0}^{\mathrm{in}}$ used at initialization.

\subsubsection{Jacobi Inner Preconditioning}
\label{sec:jacobi-inner-preconditioning}

For a general $Q$, we accelerate the Euclidean projected-BB method in
\eqref{eq:bb-step} with the Jacobi metric
\begin{equation}
\label{eq:jacobi-inner-metric}
    J_n
    :=\operatorname{diag}\!\left(Q+\tau_n^{-1}I\right)
    =\operatorname{diag}(Q)+\tau_n^{-1}I.
\end{equation}
The corresponding diagonally scaled projected-BB step is
\begin{equation}
\label{eq:jacobi-bb-step}
    \alpha_j^{\mathrm{JBB}}
    =
    \frac{\langle s_j,J_ns_j\rangle}
         {\langle s_j,\delta g_j\rangle},
    \qquad
    u^{j+1}
    =\Pi_{\mathcal X}^{J_n}
      \left(u^j-\alpha_j^{\mathrm{JBB}}J_n^{-1}g^j\right),
\end{equation}
where $g^j$ denotes the gradient of the smooth part of
\eqref{eq:prox-subproblem} at $u^j$, $s_j:=u^{j}-u^{j-1}$, and
$\delta g_j:=g^{j}-g^{j-1}$.  This is a scaled gradient-projection step with
a spectral steplength in the sense of \cite{bonettini2009scaled}.  Since
$J_n$ is diagonal, its inverse is applied elementwise, and the metric
projection is evaluated using the formulas in
Appendix~\ref{app:cone-projections}.

The Jacobi metric is used only to define the inner direction and projection;
the stopping test is restored to the unpreconditioned Euclidean coordinates.
Indeed, for
$\widehat{\Delta u}_{n,k}:=J_n^{1/2}\Delta u_{n,k}$,
\begin{equation}
\label{eq:jacobi-stop-unscale}
    \|\widehat{\Delta u}_{n,k}\|_2
    =\|\Delta u_{n,k}\|_{J_n},
    \qquad
    \|\Delta u_{n,k}\|_2
    =\|J_n^{-1/2}\widehat{\Delta u}_{n,k}\|_2.
\end{equation}
The second quantity is the one used in
\eqref{eq:current-inner-step-stop}.
Appendix~\ref{app:inner-accuracy-ablation} reports an ablation of the adaptive
inner-accuracy rule and Jacobi inner preconditioning.

\subsubsection{Adaptive Stepsize and Restart}

Following cuPDLPx \cite{lu2025cupdlpx}, we define the primal and dual
stepsizes using a step size $\eta$ and a primal weight $\omega_n$:
\begin{equation}
\label{eq:practical-stepsizes}
    \tau_n=\frac{\eta}{\omega_n},
    \qquad
    \sigma_n=\eta\omega_n,
    \qquad
    \eta:=\frac{0.998}{\|A\|_2}.
\end{equation}
The primal weight $\omega_n$ is updated when a restart is triggered.

\paragraph{Primal Weight Update.}
For the accepted endpoint $w^{n,N_n}$, the proportional-integral-derivative
(PID) update uses the movement ratio and exponentially discounted integral
\begin{equation}
\label{eq:pid-error}
    e_n
    =
    \log\left(
        \frac{\|w_y^{n,N_n}-y^{n,0}\|}
        {\omega_n\|w_x^{n,N_n}-x^{n,0}\|}
    \right),
    \qquad
    I_n=\beta_I I_{n-1}+e_n.
\end{equation}
We update the primal weight by
\begin{equation}
\label{eq:pid-weight}
    \log\omega_{n+1}
    =
    \log\omega_n
    +K_Pe_n+K_I I_n+K_D(e_n-e_{n-1}).
\end{equation}

\paragraph{Adaptive Restart.}
For the stored PDHG candidate, define
\[
    d_x^{n,k}:=x^{n,k}-w_x^{n,k+1},
    \qquad
    d_y^{n,k}:=y^{n,k}-w_y^{n,k+1}.
\]
The fixed-point error is
\begin{equation}
\label{eq:practical-fixed-point-movement}
    E_{n,k}^2
    =
    \omega_n\|d_x^{n,k}\|^2
    +\omega_n^{-1}\|d_y^{n,k}\|^2
    +2\eta\langle A d_x^{n,k},d_y^{n,k}\rangle.
\end{equation}
Using our thresholds, a restart occurs under sufficient decay,
necessary decay with no local progress, or an artificial iteration limit:
\begin{equation}
\label{eq:practical-restart-rule}
\begin{aligned}
    E_{n,k}&\le\beta_{\mathrm{s}}E_{n,0},
    \quad\text{or}\\
    E_{n,k}&\le\beta_{\mathrm{n}}E_{n,0}
        \ \text{ and }\ E_{n,k}>E_{n,k-1},
    \quad\text{or}\\
    k&\ge\beta_{\mathrm{a}}K_{\mathrm{tot}},
\end{aligned}
\end{equation}
where $K_{\mathrm{tot}}$ is the cumulative number of outer iterations.
Table~\ref{tab:practical-defaults} gives the default constants.

\begin{table}[t]
    \centering
    \papertablesetup
    \caption{Default parameters for problem rescaling and algorithm
    enhancements.}
    \label{tab:practical-defaults}
    \medskip
    \begin{tabular}{lll}
        \toprule
        Component & Parameters & Default values \\
        \midrule
        Ruiz rescaling
        & Iterations
        & $10$ \\
        Pock-Chambolle rescaling
        & $\alpha$
        & $1.00$ \\
        Cone scaling
        & Mode
        & Cone-preserving \\
        Bound-objective rescaling
        & Status
        & Enabled \\
        \midrule
        Inner accuracy
        & $(\epsilon_{0}^{\mathrm{in}},\gamma_{\mathrm{in}},\epsilon_{\min})$
        & $(10^{-3},5\!\times\!10^{-4},10^{-9})$ \\
        Jacobi preconditioning
        & Status
        & Enabled \\
        Reflected-Halpern acceleration
        & $\rho$
        & $1.00$ \\
        Restart
        & $(\beta_{\mathrm{s}},\beta_{\mathrm{n}},\beta_{\mathrm{a}})$
        & $(0.20,0.80,0.36)$ \\
        Primal weight
        & $(\beta_I,K_P,K_I,K_D)$
        & $(0.30,0.99,0.01,0.00)$ \\
        \bottomrule
    \end{tabular}
\end{table}

\subsection{Multi-GPU Implementation}
\label{sec:distributed}

\name distributes the products with $A$, $A^\top$, and $Q$
across multiple GPUs. The constraint operator uses the two-dimensional
partition of D-PDLP \cite{li2026dpdlp}, while the quadratic operator follows
the same primal-vector partition. This increases aggregate memory capacity
without fully replicating the primal and dual vectors.

\subsubsection{Distributed Storage}

The constraint matrix $A$ is sharded into a $P_r \times P_c$ grid of blocks.
Writing
\[
    A=
    \begin{bmatrix}
        A_{11}&\cdots&A_{1P_c}\\
        \vdots&&\vdots\\
        A_{P_r1}&\cdots&A_{P_rP_c}
    \end{bmatrix},
    \qquad
    x=(x_1,\ldots,x_{P_c}),
    \qquad
    y=(y_1,\ldots,y_{P_r}),
\]
rank $(i,j)$ stores $A_{ij}$, the corresponding slices $x_j$ and $y_i$, and
their local work vectors. Thus $x_j$ is shared only within process column $j$
and $y_i$ only within process row $i$.

The quadratic operator uses the same partition
$x=(x_1,\ldots,x_{P_c})$. For the structured representation
$Q=P+R^\top D R$, write
\begin{equation}
\label{eq:distributed-q-storage}
    P=\begin{bmatrix}P_{:1}&\cdots&P_{:P_c}\end{bmatrix},
    \qquad
    R=\begin{bmatrix}R_1&\cdots&R_{P_c}\end{bmatrix},
\end{equation}
where $P_{:j}$ and $R_j$ contain the columns multiplying $x_j$. Rank $(i,j)$
stores these two column shards; because they depend only on $j$, they are
replicated within process column $j$. The middle matrix $D$ and the
rank-dimensional work vectors are replicated. For diagonal $Q$, rank $(i,j)$
stores only the local diagonal block $q_j$.

\subsubsection{Distributed Computation}

The two constraint products are assembled from local sparse products as
\begin{equation}
\label{eq:distributed-products}
    (Ax)_i=\sum_{j=1}^{P_c}A_{ij}x_j,
    \qquad
    (A^\top y)_j=\sum_{i=1}^{P_r}A_{ij}^\top y_i.
\end{equation}
The first sum is reduced across each process row and the second across each
process column. For the partition in
\eqref{eq:distributed-q-storage}, the quadratic product is assembled as
\begin{equation}
\label{eq:distributed-q-product}
\begin{aligned}
    p&:=\sum_{\ell=1}^{P_c}P_{:\ell}x_\ell,
    &h&:=\sum_{\ell=1}^{P_c}R_\ell x_\ell,\\
    (Qx)_j&=p_j+R_j^\top D h,
    &&j=1,\ldots,P_c.
\end{aligned}
\end{equation}
The sums defining $p$ and $h$ are reduced across each process row; every rank
then extracts the $j$th primal slice $p_j$ and applies $R_j^\top$. The sparse and
low-rank terms may be used independently by setting the other term to zero.
For diagonal $Q$, the product is local,
$(Qx)_j=q_j\odot x_j$. Global scalar quantities in
Section~\ref{sec:algorithm-enhancement} are obtained by summing local
inner-product and squared-norm contributions. The Fisher study in
Section~\ref{sec:fisher-multigpu-exp} reports the resulting end-to-end scaling.

\section{Numerical Experiments}
\label{sec:experiments}

We evaluate \name on standard and large-scale QP, public convex QCQP,
Mittelmann SOCP, and large-scale quasilinear Fisher equilibrium benchmarks.
We study multi-GPU scaling on both large-scale QP and conic instances.

\subsection{Experimental Setup}
\label{sec:exp-setup}

\paragraph{Selected baselines.}
We consider a broad set of candidate solvers spanning first-order and
interior-point methods, and compare \name with the selected baselines. The
selected first-order baselines are PDQP~\cite{lu2023practical},
PDHCG~\cite{huang2025restarted}, HPR-QP and HPR-SOCP~\cite{chen2025hpr},
SCS~\cite{o2016conic,scs}, PDCS~\cite{lin2025pdcs}, and
OSQP~\cite{stellato2020osqp}. The selected interior-point baselines are
MOSEK~\cite{mosek2025manual}, Gurobi~\cite{gurobi2024manual},
COPT~\cite{ge2022cardinal}, and Clarabel~\cite{goulart2024clarabel}. All
solvers are compared at the same target accuracy $\epsilon$.

\paragraph{Computing environment.}
Experiments are conducted on a server with eight NVIDIA H100 GPUs, each with
80GB HBM3, an Intel Xeon Platinum 8469C CPU at 2.60GHz, and 512GB RAM. Unless
otherwise stated, each run uses a single GPU.

\paragraph{Evaluation metrics.}
We report the number of solved instances and shifted geometric mean (SGM)
runtime
with shift 10; the QP tables also report arithmetic mean runtime:
\begin{equation}
    \operatorname{SGM}_{10}
    =
    \exp\left(\frac1N\sum_{i=1}^N \log(t_i+10)\right)-10 .
\end{equation}
Any unsuccessful run is charged the time limit when it is included in an
aggregate runtime metric.

\subsection{Standard Convex QP Benchmarks}

We first evaluate the QP specialization of \name on two standard benchmark repositories.
\begin{itemize}
    \item \textbf{Maros-M{\'e}sz{\'a}ros benchmark} \cite{maros1999repository}: 134 convex quadratic programming instances, tested at tolerance $\epsilon=10^{-6}$ with a 1000 second time limit.
    \item \textbf{Mittelmann benchmark} \cite{mittelmann2021decision}: 21 continuous convex QP instances, tested at tolerances $\epsilon=10^{-6}$ and $\epsilon=10^{-8}$ with a 3600 second time limit.
\end{itemize}

\begin{table}[H]
  \centering
  \papertablesetup
  \caption{Performance on 134 Maros-M{\'e}sz{\'a}ros instances with a
  1000-second time limit and tolerance $\epsilon=10^{-6}$. The best results
  are marked in \textbf{bold}, and the second best are
  \underline{underlined}.}
  \label{tab:comparison_134}
  \medskip
  \begin{tabular}{lcccc}
    \toprule
    \textbf{Solver} & \textbf{Total} & \textbf{Solved} & \textbf{SGM$_{10}$ (s)} & \textbf{Average (s)} \\
    \midrule
    PDQP & \multirow{4}{*}{134} & 118 & 28.53 & 160.70 \\
    PDHCG & & 111 & 33.51 & 210.53 \\
    HPR-QP & & \underline{124} & \textbf{10.56} & \underline{93.32} \\
    \name & & \textbf{126} & \underline{10.95} & \textbf{83.49} \\
    \bottomrule
  \end{tabular}
\end{table}

Table~\ref{tab:comparison_134} shows that \name improves robustness over
the earlier PDHCG implementation and is competitive with HPR-QP: with Jacobi
preconditioning on the non-diagonal subset it solves 126 instances, attains
the best arithmetic mean runtime, and is second best in shifted geometric
mean.

\begin{table}[htbp]
  \centering
  \papertablesetup
  \caption{Performance on 21 Mittelmann instances with a 3600-second time
  limit. The best results are marked in \textbf{bold}, and the second best are
  \underline{underlined}.}
  \label{tab:solver_comparison_21}
  \medskip
  \begin{tabular}{lcccc}
    \toprule
    \textbf{Solver} & \textbf{Total} & \textbf{Solved} & \textbf{SGM$_{10}$ (s)} & \textbf{Average (s)} \\
    \midrule
    \multicolumn{5}{c}{\textit{Tolerance $\epsilon = 10^{-6}$}} \\
    \midrule
    PDQP & \multirow{4}{*}{21} & 11 & 410.34 & 1942.89 \\
    PDHCG & & 10 & 486.81 & 1987.11 \\
    HPR-QP & & \underline{14} & \underline{126.88} & \underline{1282.34} \\
    \name & & \textbf{17} & \textbf{72.74} & \textbf{766.94} \\
    \midrule
    \multicolumn{5}{c}{\textit{Tolerance $\epsilon = 10^{-8}$}} \\
    \midrule
    HPR-QP & \multirow{2}{*}{21} & \underline{14} & \underline{137.42} & \underline{1318.61} \\
    \name & & \textbf{17} & \textbf{91.22} & \textbf{894.32} \\
    \bottomrule
  \end{tabular}
\end{table}

On the Mittelmann benchmark, \name solves 17 of the 21 instances at both
tolerances, the most among the tested solvers, and tightening the tolerance
from $10^{-6}$ to $10^{-8}$ increases its runtime only moderately; the
aggregate comparison does not isolate the individual effects of reflection,
restart, and inner accuracy.

\subsection{Public Convex QCQP Benchmark}
\label{sec:qcqp-exp}

We next consider 21 public convex QCQP instances assembled from QPLIB
\cite{furini2019qplib} and the Mittelmann benchmark collection
\cite{mittelmann2021decision}. Write an instance as
\begin{equation}
\label{eq:qcqp-benchmark-form}
\begin{aligned}
    \min_{x\in\mathcal X}\quad
        & \frac12 x^\top Q_0x+c_0^\top x \\
    \mathrm{s.t.}\quad
        & \frac12 x^\top Q_ix+c_i^\top x\le d_i,
        \qquad i=1,\ldots,p,
\end{aligned}
\end{equation}
where $\mathcal X$ collects the linear constraints and variable bounds and
$Q_i\succeq0$.  Choosing factors $Q_i=R_i^\top R_i$ and using the rotated
second-order cone from Section~\ref{sec:cone-projections}, each quadratic
constraint is equivalent to
\begin{equation}
\label{eq:qcqp-to-rsoc}
    (d_i-c_i^\top x,\,1,\,R_ix)\in\mathcal K_{\mathrm{rsoc}}.
\end{equation}
Retaining the quadratic objective in
\eqref{eq:qcqp-benchmark-form} and replacing only the quadratic constraints by
\eqref{eq:qcqp-to-rsoc} gives a quadratic SOCP (QSOCP). A solver requiring a
linear objective receives the equivalent SOCP
\begin{equation}
\label{eq:qcqp-to-socp}
\begin{aligned}
    \min_{x\in\mathcal X,\,t}\quad & t+c_0^\top x \\
    \mathrm{s.t.}\quad
        & (t,\,1,\,R_0x)\in\mathcal K_{\mathrm{rsoc}},\\
        & (d_i-c_i^\top x,\,1,\,R_ix)\in\mathcal K_{\mathrm{rsoc}},
          \qquad i=1,\ldots,p.
\end{aligned}
\end{equation}
Thus the SOCP path adds an objective-epigraph cone, whereas the QSOCP path
retains $Q_0$ explicitly. The formulation used by each solver is recorded in
Table~\ref{tab:qcqp-public}; COPT GPU is evaluated through its native QCQP
interface. Each run has a 3600 second limit. The table reports the number of
runs declared solved at the target tolerance and $\operatorname{SGM}_{10}$ over
all 21 instances, with every unsuccessful run charged the time limit.

\begin{table}[htbp]
    \centering
  \papertablesetup
    \caption{Aggregate performance on 21 public convex QCQP instances. The
    formulation column gives the actual input form used in each run. Times are
    in seconds and the time limit is 3600 seconds. Within each metric column,
    the best result is marked in \textbf{bold} and the second best is
    \underline{underlined}; tied values share the same marking.}
    \label{tab:qcqp-public}
    \medskip
    \begin{tabular}{llrrrrrr}
        \toprule
        & & \multicolumn{2}{c}{$\epsilon=10^{-4}$}
        & \multicolumn{2}{c}{$\epsilon=10^{-6}$}
        & \multicolumn{2}{c}{$\epsilon=10^{-8}$}\\
        \cmidrule(lr){3-4}\cmidrule(lr){5-6}\cmidrule(lr){7-8}
        Solver & Formulation & Solved & SGM$_{10}$ & Solved & SGM$_{10}$ & Solved & SGM$_{10}$\\
        \midrule
        \multicolumn{8}{l}{\textit{First-order methods}}\\
        \name                    & QSOCP & \textbf{21} & \underline{8.11} & \textbf{21} & 21.64 & \textbf{20} & 42.63\\
        HPR-SOCP                  & QSOCP & 19 & 36.00 & 19 & 54.56 & \underline{19} & 82.06\\
        SCS indirect             & SOCP  & 13 & 209.18 & 10 & 1107.13 & 5 & 2744.60\\
        SCS direct               & SOCP  & 16 & 62.70 & 14 & 215.37 & 11 & 625.47\\
        PDCS                     & SOCP  & 13 & 152.94 & 13 & 249.69 & 12 & 664.11\\
        \addlinespace
        \multicolumn{8}{l}{\textit{Interior-point methods}}\\
        MOSEK                    & QCQP  & \underline{20} & 16.40 & \underline{20} & 16.54 & \textbf{20} & 16.76\\
        Gurobi                   & QCQP  & \textbf{21} & \textbf{6.06} & \underline{20} & \underline{11.74} & \textbf{20} & \underline{14.71}\\
        COPT GPU                 & QCQP  & \underline{20} & 9.68 & \underline{20} & \textbf{9.78} & \textbf{20} & \textbf{9.94}\\
        Clarabel                 & SOCP  & \textbf{21} & 14.66 & \textbf{21} & 19.97 & 17 & 50.54\\
        \bottomrule
    \end{tabular}
\end{table}

\name solves all 21 instances at $10^{-4}$ and $10^{-6}$ and 20 at
$10^{-8}$. It remains the most robust and fastest first-order solver in
aggregate at every target tolerance. Against HPR-SOCP, \name records
21 wins and no losses, 19 wins and 2 losses, and 17 wins, 3 losses, and 1 tie
at $10^{-4}$, $10^{-6}$, and $10^{-8}$, respectively. At $10^{-4}$, \name
attains the second-lowest $\operatorname{SGM}_{10}$ overall, behind only
Gurobi; the commercial QCQP interfaces retain an aggregate runtime advantage
at the tighter tolerances.

\subsection{Public SOCP Benchmark}
\label{sec:socp-exp}

The SOCP study uses the 18 public instances in the Mittelmann collection
\cite{mittelmann2024socp}.

\begin{table}[htbp]
    \centering
  \papertablesetup
    \caption{Aggregate performance on 18 public SOCP instances. Times are in
    seconds and the time limit is 3600 seconds. Within each metric column,
    the best result is marked in \textbf{bold} and the second best is
    \underline{underlined}; tied values share the same marking.}
    \label{tab:socp-public}
    \medskip
    \begin{tabular}{lrrrr}
        \toprule
        & \multicolumn{2}{c}{$\epsilon=10^{-4}$}
        & \multicolumn{2}{c}{$\epsilon=10^{-6}$}\\
        \cmidrule(lr){2-3}\cmidrule(lr){4-5}
        Solver & Solved & SGM$_{10}$ & Solved & SGM$_{10}$\\
        \midrule
        \multicolumn{5}{l}{\textit{First-order methods}}\\
        \name             & \textbf{18} & \underline{18.76} & \textbf{18} & 69.25\\
        PDCS              & \underline{14} & 613.98 & 12 & 614.03\\
        SCS direct        & \underline{14} & 475.45 & 9 & 1166.90\\
        SCS indirect      & 4 & 2479.75 & 3 & 2644.17\\
        HPR-SOCP          & \textbf{18} & \textbf{18.11} & \textbf{18} & 51.19\\
        \addlinespace
        \multicolumn{5}{l}{\textit{Interior-point methods}}\\
        COPT GPU          & \textbf{18} & 19.00 & \textbf{18} & \textbf{19.01}\\
        Gurobi            & \textbf{18} & 22.27 & \textbf{18} & \underline{19.85}\\
        MOSEK             & \textbf{18} & 30.97 & \textbf{18} & 32.44\\
        Clarabel          & \underline{14} & 601.98 & \underline{14} & 602.04\\
        \bottomrule
    \end{tabular}
\end{table}

Both \name and HPR-SOCP solve all 18 instances at both tolerances. HPR-SOCP
has the lower aggregate $\operatorname{SGM}_{10}$, although the two solvers are
close at $10^{-4}$, with values of 18.11 and 18.76. Moreover, \name remains
competitive at the instance level, with 8 wins and 10 losses at $10^{-4}$ and
5 wins and 13 losses at $10^{-6}$. Thus the updated SOCP results show matching
robustness but favor HPR-SOCP in aggregate runtime. In contrast, the QCQP
results indicate that \name benefits more clearly when the quadratic
objective is retained and handled directly.

\subsection{Real-World Large-Scale Sparse QPs}
\label{sec:realworldqp}

We also evaluate large-scale Lasso instances derived from LIBSVM \cite{chang2011libsvm} and the SuiteSparse Matrix Collection \cite{davis2011university}. The Lasso problem
\begin{equation}
    \min_x \; \|Ax-b\|_2^2+\lambda\|x\|_1
\end{equation}
is solved through its standard QP reformulation
\begin{equation}
\begin{aligned}
    \min_{x,y,t}\quad & y^\top y+\lambda\mathbf 1^\top t\\
    \mathrm{s.t.}\quad & y=Ax-b,\\
    & -t\le x\le t .
\end{aligned}
\end{equation}
We set $\lambda=0.01\|A^\top b\|_\infty$ and use tolerance $\epsilon=10^{-6}$.
The dimensions and densities of the nine instances are listed in
Table~\ref{tab:lasso-instances} of Appendix~\ref{app:instance-statistics}.

\begin{table}[htbp]
  \centering
  \papertablesetup
  \caption{Single-GPU solve times on large-scale Lasso-derived QPs. Times
  are in seconds, with a 7200-second limit. The best results are marked in
  \textbf{bold}, the second best are \underline{underlined}, and ``f''
  denotes an unsuccessful run.}
  \label{real_lasso}
  \medskip
    \begin{tabular}{lrrrrrrr}
    \toprule
    & \multicolumn{6}{c}{First-order methods}
      & \multicolumn{1}{c}{Interior-point} \\
    \cmidrule(lr){2-7}\cmidrule(lr){8-8}
    Problem & \textbf{\name} & HPR-QP & PDHCG & PDQP & SCS GPU & OSQP & COPT \\
    \midrule
    SLS & \textbf{1.96} & \underline{2.09} & 3.35 & 7.30 & 345.09 & 80.32 & 88.21 \\
    rcv1\_test & \textbf{2.37} & \underline{6.21} & 7.12 & 19.54 & f & f & f \\
    avazu-site.tr & \textbf{1337.05} & 4911.41 & \underline{1377.54} & 5124.82 & f & f & f \\
    avazu-app & \textbf{217.70} & \underline{753.65} & 1429.55 & 5557.97 & f & f & f \\
    avazu-site & \textbf{1642.69} & \underline{3213.38} & 4224.95 & f & f & f & f \\
    kddb2010\_test & \underline{11.59} & 26.87 & \textbf{11.10} & 46.49 & 490.66 & 255.81 & 69.57 \\
    kdda2010\_test & \textbf{9.90} & \underline{29.36} & 61.00 & 148.01 & f & f & f \\
    kddb2010\_train & \underline{703.36} & 1971.24 & \textbf{387.00} & 1715.50 & f & f & f \\
    kdda2010\_train & \textbf{341.16} & \underline{842.02} & 2705.94 & f & f & f & f \\
    \bottomrule
  \end{tabular}
\end{table}

Table~\ref{real_lasso} shows that \name is the fastest solver on seven of
the nine instances and is especially effective at the largest scales: on
\texttt{avazu-app} it is $3.46\times$ faster than HPR-QP, while the
remaining solvers are slower still or fail.

\begin{table}[htbp]
    \centering
  \papertablesetup
    \caption{Multi-GPU scaling of \name on the Lasso-derived QPs. Times are in
    seconds, and $S_p=t_1/t_p$ is the speedup over one GPU. The fastest time in
    each row is marked in \textbf{bold}, and the second fastest is
    \underline{underlined}.}
    \label{tab:lasso-scaling}
    \medskip
    \begin{tabular}{lrrrrrrr}
        \toprule
        & \multicolumn{4}{c}{Solve time (s)}
        & \multicolumn{3}{c}{Speedup}\\
        \cmidrule(lr){2-5}\cmidrule(lr){6-8}
        Problem & 1 GPU & 2 GPUs & 4 GPUs & 8 GPUs
            & $S_2$ & $S_4$ & $S_8$\\
        \midrule
        SLS               & 1.96    & 1.10   & \underline{0.56}   & \textbf{0.50}   & 1.78 & 3.50 & 3.92\\
        rcv1\_test        & 2.37    & 1.48   & \underline{1.06}   & \textbf{0.81}   & 1.60 & 2.24 & 2.93\\
        avazu-site.tr     & 1337.05 & 953.02 & \underline{547.28} & \textbf{376.90} & 1.40 & 2.44 & 3.55\\
        avazu-app         & 217.70  & 116.92 & \underline{65.86}  & \textbf{46.30}  & 1.86 & 3.31 & 4.70\\
        avazu-site        & 1642.69 & 993.62 & \underline{765.46} & \textbf{503.99} & 1.65 & 2.15 & 3.26\\
        kddb2010\_test    & 11.59   & 11.18  & \underline{8.79}   & \textbf{7.07}   & 1.04 & 1.32 & 1.64\\
        kdda2010\_test    & 9.90    & 8.50   & \underline{6.77}   & \textbf{5.61}   & 1.16 & 1.46 & 1.76\\
        kddb2010\_train   & 703.36  & 466.28 & \underline{317.69} & \textbf{233.44} & 1.51 & 2.21 & 3.01\\
        kdda2010\_train   & 341.16  & 237.19 & \underline{167.62} & \textbf{124.89} & 1.44 & 2.04 & 2.73\\
        \bottomrule
    \end{tabular}
\end{table}

All nine Lasso-derived QPs benefit from multi-GPU execution. On
\texttt{avazu-app}, eight GPUs reduce the solve time from 217.70 to 46.30
seconds, a $4.70\times$ speedup; the corresponding speedup on
\texttt{avazu-site.tr} is $3.55\times$.

\subsection{Large-Scale Fisher Equilibrium}
\label{sec:fisher-multigpu-exp}

Consider a quasilinear Fisher market with $n$ buyers and $m$ divisible goods.
Buyer $i$ has budget $w_i>0$, good $j$ has supply $b_j>0$, and $u_{ij}$ is
buyer $i$'s value for good $j$. Let
$\mathcal E:=\{(i,j):u_{ij}>0\}$ be the sparse valuation graph. The quasilinear
extension of the Eisenberg-Gale program~\cite{gao2023infinite} is
\begin{equation}
\label{eq:fisher-qleg}
\begin{aligned}
    \max_{x,\delta}\quad
        & \sum_{i=1}^n w_i
          \log\!\left(\sum_{j:(i,j)\in\mathcal E}u_{ij}x_{ij}+\delta_i\right)
          -\sum_{i=1}^n\delta_i\\
    \mathrm{s.t.}\quad
        & \sum_{i:(i,j)\in\mathcal E}x_{ij}=b_j,
          \qquad j=1,\ldots,m,\\
        & x_{ij}\ge0,\quad (i,j)\in\mathcal E,
          \qquad \delta_i\ge0,\quad i=1,\ldots,n,
\end{aligned}
\end{equation}
where $\delta_i$ represents the money retained by buyer $i$. Introducing
$z_i$ and $t_i$ gives the equivalent exponential-cone formulation used in
our experiments:
\begin{equation}
\label{eq:fisher-conic}
\begin{aligned}
    \min_{x,\delta,z,t}\quad
        & \sum_{i=1}^n(\delta_i-w_i z_i)\\
    \mathrm{s.t.}\quad
        & \sum_{i:(i,j)\in\mathcal E}x_{ij}=b_j,
          \qquad j=1,\ldots,m,\\
        & t_i=\sum_{j:(i,j)\in\mathcal E}u_{ij}x_{ij}+\delta_i,
          \qquad i=1,\ldots,n,\\
        & (z_i,1,t_i)\in\mathcal K_{\mathrm{exp}},
          \qquad i=1,\ldots,n,\\
        & x_{ij}\ge0,\quad (i,j)\in\mathcal E,
          \qquad \delta_i\ge0,\quad i=1,\ldots,n.
\end{aligned}
\end{equation}
Let $e:=|\mathcal E|$ and $\rho:=e/(nm)$.  We store allocations only on
$\mathcal E$ and represent each $(z_i,1,t_i)$ by an exponential-cone block
whose middle coordinate is fixed at one.  Thus \eqref{eq:fisher-conic} has
$e+4n$ stored primal coordinates, $m+n$ affine equations, $n$ exponential
cones, and $e+n$ nonnegative coordinates.

We evaluate five instances with $10^3\le n\le10^7$ buyers, $400\le m\le4000$
goods, and valuation density $0.01\le\rho\le0.20$. Nonzero valuations
$u_{ij}$ and budgets $w_i$ are sampled uniformly from $[0.1,1.1]$, every
good has supply $b_j=0.2n$, and all instances use random seed 1. The exact
values of $(n,m,\rho)$ and the resulting conic dimensions are reported in
Table~\ref{tab:fisher-instances} of Appendix~\ref{app:instance-statistics}.
The solver time limit is 3600 seconds.

\begin{table}[htbp]
    \centering
  \papertablesetup
    \caption{Solver comparison on the quasilinear Fisher equilibrium
    instances, with each GPU solver using one GPU. Times are in seconds. The
    best successful time in each row is
    marked in \textbf{bold}, the second best is \underline{underlined}, and
    ``f'' denotes an unsuccessful or omitted run under the 3600-second
    resource limit.}
    \label{tab:fisher-solvers}
    \medskip
    \begin{tabular}{rrrrrrrr}
        \toprule
        & \multicolumn{3}{c}{First-order methods}
        & \multicolumn{4}{c}{Interior-point methods}\\
        \cmidrule(lr){2-4}\cmidrule(lr){5-8}
        $n$ & \name & PDCS & SCS direct & COPT GPU & MOSEK & Gurobi & Clarabel\\
        \midrule
        \multicolumn{8}{c}{\textit{Tolerance $\epsilon=10^{-4}$}}\\
        \midrule
        1,000      & \textbf{0.36} & 30.32 & 7.45 & 0.67 & \underline{0.60} & 157.13 & 1.30\\
        10,000     & \textbf{0.69} & 96.22 & 356.51 & \underline{6.61} & 24.19 & f & 80.82\\
        100,000    & \textbf{4.43} & 96.02 & f & \underline{64.69} & 68.36 & f & 1353.31\\
        1,000,000  & \textbf{46.65} & 1565.44 & f & f & \underline{687.52} & f & f\\
        10,000,000 & \textbf{231.94} & f & f & f & f & f & f\\
        \midrule
        \multicolumn{8}{c}{\textit{Tolerance $\epsilon=10^{-6}$}}\\
        \midrule
        1,000      & 0.85 & 30.96 & 238.16 & \underline{0.50} & \textbf{0.48} & f & 1.81\\
        10,000     & \textbf{0.77} & 502.63 & f & \underline{7.16} & 22.77 & f & 123.97\\
        100,000    & \textbf{4.77} & 1315.10 & f & \underline{66.10} & 83.41 & f & 1929.71\\
        1,000,000  & \textbf{48.32} & f & f & f & \underline{873.38} & f & f\\
        10,000,000 & \textbf{241.53} & f & f & f & f & f & f\\
        \bottomrule
    \end{tabular}
\end{table}

The symbol ``f'' covers time-limit, memory-limit, and numerical terminations.
After a solver failed at a given size and tolerance, larger or tighter runs
with the same solver were omitted.

\begin{table}[htbp]
    \centering
  \papertablesetup
    \caption{Multi-GPU scaling of \name on the quasilinear Fisher equilibrium
    instances. Times are in seconds, and $S_p=t_1/t_p$ is the speedup over one GPU. The fastest time in
    each row is marked in \textbf{bold}, and the second fastest is
    \underline{underlined}.}
    \label{tab:fisher-scaling}
    \medskip
    \begin{tabular}{rrrrrrrr}
        \toprule
        & \multicolumn{4}{c}{Solve time (s)}
        & \multicolumn{3}{c}{Speedup}\\
        \cmidrule(lr){2-5}\cmidrule(lr){6-8}
        $n$ & 1 GPU & 2 GPUs & 4 GPUs & 8 GPUs
            & $S_2$ & $S_4$ & $S_8$\\
        \midrule
        \multicolumn{8}{c}{\textit{Tolerance $\epsilon=10^{-4}$}}\\
        \midrule
        1,000      & \textbf{0.36} & \underline{0.42} & 0.45 & 0.49 & 0.86 & 0.80 & 0.73\\
        10,000     & 0.69 & 0.60 & \textbf{0.50} & \underline{0.55} & 1.15 & 1.38 & 1.25\\
        100,000    & 4.43 & 2.45 & \underline{1.46} & \textbf{1.05} & 1.81 & 3.03 & 4.22\\
        1,000,000  & 46.65 & 22.87 & \underline{11.16} & \textbf{5.85} & 2.04 & 4.18 & 7.97\\
        10,000,000 & 231.94 & 128.38 & \underline{65.93} & \textbf{34.28} & 1.81 & 3.52 & 6.77\\
        \midrule
        \multicolumn{8}{c}{\textit{Tolerance $\epsilon=10^{-6}$}}\\
        \midrule
        1,000      & \textbf{0.85} & \underline{0.97} & 1.01 & 1.11 & 0.88 & 0.84 & 0.77\\
        10,000     & 0.77 & 0.66 & \textbf{0.56} & \underline{0.60} & 1.17 & 1.38 & 1.28\\
        100,000    & 4.77 & 2.61 & \underline{1.55} & \textbf{1.12} & 1.83 & 3.08 & 4.26\\
        1,000,000  & 48.32 & 23.74 & \underline{11.54} & \textbf{6.07} & 2.04 & 4.19 & 7.96\\
        10,000,000 & 241.53 & 133.57 & \underline{68.68} & \textbf{35.69} & 1.81 & 3.52 & 6.77\\
        \bottomrule
    \end{tabular}
\end{table}

The smallest instances do not amortize collective latency, but scaling
becomes effective from $n=10^5$: at $n=10^6$, eight GPUs reduce the solve
time from 46.65 to 5.85 seconds at $10^{-4}$ and from 48.32 to 6.07 seconds
at $10^{-6}$, corresponding to speedups of $7.97\times$ and $7.96\times$,
and at
$n=10^7$, \name is the only solver reported as reaching optimality within the
time and memory limits.

\section{Conclusion}
\label{sec:conclusion}

We introduced \name, a matrix-free restarted primal-dual solver for
large-scale conic convex quadratic programming on GPUs.  The practical
method combines primal-first PDHG updates with an anchored reflected-Halpern
acceleration, matrix-free quadratic proximal solves, specialized cone
projections. \name supports both single- and multi-GPU execution.
For the underlying restarted averaged scheme, we proved local linear
convergence under uniform quadratic growth of the smoothed duality gap, with
exact and inexact proximal evaluations, including a fixed amount of
projected-gradient work per subproblem; the quadratic-growth property itself
was verified under strict complementarity.

The experiments cover standard and large-scale sparse QPs, public convex
QCQP and SOCP benchmarks, and large-scale quasilinear Fisher equilibrium
instances. The Fisher study shows strong scaling through eight GPUs and
robustness at $10^7$ buyers, corresponding to more than $4.4\times10^8$
stored primal coordinates, and the Lasso-derived QPs confirm multi-GPU
acceleration on large sparse diagonal-$Q$ instances.

\bibliographystyle{plainnat}
\bibliography{main,ref}

\clearpage
\appendix

\section{Proof in Section~\ref{sec:averaged-restart-theory}}

\subsection{Proof of Theorem~\ref{thm:exact-local-linear}}
\label{app:proof-exact-local-linear}

The exact contraction is obtained in four steps.  We first establish a
one-step energy inequality.  Summing this inequality over one epoch and using
convexity gives the ordinary gap bound for the Ces\`aro average.  The ordinary
gap estimate is then converted into a smoothed-gap estimate.  Finally,
Assumption~\ref{ass:smoothed-qg} converts the smoothed-gap bound into a
contraction of the distance to the KKT set.

Define the KKT operator
\begin{equation}
\label{eq:theory-kkt-operator}
    \mathcal F_{\rm KKT}(x,y)
    :=
    \begin{bmatrix}
        \partial F(x)-A^*y\\
        Ax-b
    \end{bmatrix}.
\end{equation}
Then $Z^\star=\mathcal F_{\rm KKT}^{-1}(0)$.

\begin{lemma}[one-step energy inequality]
\label{lem:one-step-energy}
For every comparison point
$\widehat z=(\widehat x,\widehat y)\in K\times Y$,
\begin{equation}
\label{eq:one-step-gap}
    \mathcal Q(z^{k+1},\widehat z)
    \le
    \frac12\|z^k-\widehat z\|_P^2
    -
    \frac12\|z^{k+1}-\widehat z\|_P^2
    -
    \frac12\|z^{k+1}-z^k\|_P^2.
\end{equation}
In particular, for every $z^\star\in Z^\star$,
\begin{equation}
\label{eq:inner-fejer}
    \|z^{k+1}-z^\star\|_P
    \le
    \|z^k-z^\star\|_P,
    \qquad
    \|z^k-z^\star\|
    \le
    \kappa_P\|z^0-z^\star\|.
\end{equation}
\end{lemma}

\begin{proof}
Optimality of the primal proximal step gives
\[
    \tau^{-1}(x^k-x^{k+1})
    \in
    \partial F(x^{k+1})-A^*y^k.
\]
After adding $A^*(y^k-y^{k+1})$,
\[
    \tau^{-1}(x^k-x^{k+1})
    +
    A^*(y^k-y^{k+1})
    \in
    \partial F(x^{k+1})-A^*y^{k+1}.
\]
The dual step similarly gives
\[
    A(x^k-x^{k+1}) + \sigma^{-1}(y^k-y^{k+1}) = A(x^k-x^{k+1}) + A(2x^{k+1}-x^k)-b = Ax^{k+1}-b.
\]
Together, these two relations show that
\[
    g^{k+1}
    :=
    P(z^k-z^{k+1})
    \in
    \mathcal F_{\rm KKT}(z^{k+1}).
\]
Convexity of $F$ implies
\[
    \mathcal Q(z^{k+1},\widehat z)
    \le
    \langle g^{k+1},z^{k+1}-\widehat z\rangle.
\]
The three-point identity in the $P$-inner product now gives
\eqref{eq:one-step-gap}.  At a saddle point,
$\mathcal Q(z^{k+1},z^\star)\ge0$.  Dropping this term and the final
nonnegative squared norm proves Fej\'er monotonicity in
\eqref{eq:inner-fejer}.  Finally,
\[
    \|z^k-z^\star\|
    \le
    \lambda_-^{-1/2}\|z^k-z^\star\|_P
    \le
    \lambda_-^{-1/2}\|z^0-z^\star\|_P
    \le
    \kappa_P\|z^0-z^\star\|.
\]
\end{proof}

\begin{proposition}[exact one-epoch estimates]
\label{prop:one-epoch-estimates}
Let
\[
    \bar z^T
    :=
    \frac1T\sum_{k=0}^{T-1}z^{k+1}.
\]
Then, for every $z^\star\in Z^\star$,
\begin{equation}
\label{eq:average-boundedness}
    \|z^k-z^\star\|
    \le
    \kappa_P\|z^0-z^\star\|,
    \qquad
    \|\bar z^T-z^\star\|
    \le
    \kappa_P\|z^0-z^\star\|.
\end{equation}
If, in addition, $\xi>0$ and $T\ge2\lambda_+/\xi$, then, for every center
$\dot z\in E\times Y$,
\begin{equation}
\label{eq:smoothed-gap-rate}
    G_\xi(\bar z^T;\dot z)
    \le
    \frac{\lambda_+}{T}\|z^0-\dot z\|^2.
\end{equation}
\end{proposition}

\begin{proof}
Summing \eqref{eq:one-step-gap} from $k=0$ to $T-1$ and discarding
nonpositive terminal terms gives
\[
    \sum_{k=0}^{T-1}\mathcal Q(z^{k+1},\widehat z)
    \le
    \frac12\|z^0-\widehat z\|_P^2.
\]
The function $\mathcal L(\cdot,\widehat y)$ is convex and
$\mathcal L(\widehat x,\cdot)$ is affine.  Then, for every
$\widehat z\in K\times Y$, Jensen's inequality shows that
\begin{equation}
\label{eq:ordinary-gap-rate}
    \mathcal Q(\bar z^T,\widehat z)
    \le
    \frac{1}{2T}\|z^0-\widehat z\|_P^2
    \le
    \frac{\lambda_+}{2T}\|z^0-\widehat z\|^2,
\end{equation}
where the second inequality follows from $P\preceq\lambda_+I$.  The first
estimate in \eqref{eq:average-boundedness} follows from
Lemma~\ref{lem:one-step-energy}; convexity of the norm and averaging give the
second.

For the smoothed-gap estimate, set
$a=\lambda_+/(2T)$ and $b_\xi=\xi/2$.  Then
\begin{equation}
\label{eq:smoothed-gap-maximization}
    G_\xi(\bar z^T;\dot z)
    \le
    \sup_{\widehat z\in K\times Y}
    \left\{
        a\|z^0-\widehat z\|^2
        -
        b_\xi\|\widehat z-\dot z\|^2
    \right\}.
\end{equation}
Relaxing $\widehat x\in K$ to $\widehat x\in E$ gives an upper bound.  For
$a<b_\xi$, completion of the square gives the unrestricted supremum
\[
    \frac{ab_\xi}{b_\xi-a}\|z^0-\dot z\|^2.
\]
Condition $T\ge2\lambda_+/\xi$ gives $a\le b_\xi/2$ and hence
$b_\xi/(b_\xi-a)\le2$.  Substitution into
\eqref{eq:smoothed-gap-maximization} proves
\eqref{eq:smoothed-gap-rate}.
\end{proof}

\begin{proof}[Proof of Theorem~\ref{thm:exact-local-linear}]
We prove by induction that the restart points remain in the retention ball
and that \eqref{eq:main-linear-rate} holds.  Both claims are immediate for
$n=0$.  Suppose they hold through outer iteration $D$.  For every
$t\le D$, choose
\[
    z_t^\star
    \in
    \Pi_{Z^\star}(z^{t,0}).
\]
Proposition~\ref{prop:one-epoch-estimates} gives
\[
\begin{aligned}
    \|z^{D+1,0}-z^{0,0}\|
    &\le
    \sum_{t=0}^{D}\|z^{t+1,0}-z^{t,0}\|\\
    &\le
    \sum_{t=0}^{D}
    \left(
        \|z^{t+1,0}-z_t^\star\|
        +
        \|z_t^\star-z^{t,0}\|
    \right)\\
    &\le
    (1+\kappa_P)\sum_{t=0}^{D}d_t\\
    &\le
    (1+\kappa_P)\sum_{t=0}^{D}e^{-t}d_0
    \le
    R_0.
\end{aligned}
\]
Thus the new restart point stays in the ball on which quadratic growth is
valid.  The same estimate with the last term omitted shows that the moving
center $z_D^\star$ also belongs to that ball.  Applying
Assumption~\ref{ass:smoothed-qg} and
Proposition~\ref{prop:one-epoch-estimates} gives
\[
    d_{D+1}^2 \le \alpha_\xi^{-1} G_\xi(z^{D+1,0};z_D^\star) \le \frac{\lambda_+}{\alpha_\xi T} \|z^{D,0}-z_D^\star\|^2 = \frac{\lambda_+}{\alpha_\xi T}d_D^2 \le e^{-2}d_D^2 \le e^{-2(D+1)}d_0^2,
\]
where the penultimate inequality is precisely the second condition in
\eqref{eq:main-restart-length}.  This closes both parts of the induction.

The proof follows a deliberate order: boundedness is established first, so
that the local quadratic-growth constant is valid at the new epoch output;
only then is quadratic growth used to obtain contraction.  This avoids
assuming local retention as an unstated premise.
\end{proof}

\subsection{Proof of Theorem~\ref{thm:inexact-local-linear}}
\label{app:proof-inexact-local-linear}

We first establish the finite-epoch perturbation estimate used to compare an
inexact epoch with its exact shadow orbit, and then prove the theorem by
combining this estimate with the exact one-epoch contraction.

\begingroup
\renewcommand{\thetheorem}{\epochperturbationnumber}
\begin{lemma}[perturbation of one finite epoch]
\label{lem:theory-epoch-perturbation}
For fixed $T$, $\tau$, and $\sigma$, there is
$C_{\rm ep}<\infty$, independent of the epoch number, the inner algorithm,
and the error values, such that
\begin{equation}
\label{eq:theory-epoch-perturbation}
    \|\widetilde{\mathcal E}_T(z)-\mathcal E_T(z)\|
    \le
    C_{\rm ep}\sum_{k=0}^{T-1}\delta_k,
\end{equation}
where $\mathcal E_T$ and $\widetilde{\mathcal E}_T$ are the exact and inexact
epoch averages starting from the same point.
\end{lemma}
\endgroup

\begin{proof}
Let $z_e^k=(x_e^k,y_e^k)$ be the exact orbit from the same initial point, and
set
\[
    a_k:=\|\widetilde x^k-x_e^k\|,
    \qquad
    b_k:=\|\widetilde y^k-y_e^k\|,
    \qquad
    a_0=b_0=0.
\]
The definition of $\delta_k$ and nonexpansiveness of the exact proximal map
give
\begin{equation}
\label{eq:theory-primal-perturbation}
    a_{k+1}
    \le
    a_k+\tau\|A\|b_k+\delta_k.
\end{equation}
The extrapolated-primal difference is bounded by
$2a_{k+1}+a_k$, and hence
\begin{equation}
\label{eq:theory-dual-perturbation}
    b_{k+1}
    \le
    b_k+\sigma\|A\|(2a_{k+1}+a_k).
\end{equation}
With $h_k=a_k+b_k$, these inequalities imply
\[
    h_{k+1}
    \le
    L_{\rm rec}h_k+c_\delta\delta_k
\]
for constants depending only on $\tau$, $\sigma$, and $\|A\|$.  Expanding
this recursion over the fixed number $T$ bounds every $h_k$, and therefore
the difference of the epoch averages, by a constant times
$\sum_{k=0}^{T-1}\delta_k$.
\end{proof}

\begin{proof}[Proof of Theorem~\ref{thm:inexact-local-linear}]
For each $n$, choose
$z_n^\star\in\Pi_{Z^\star}(z^{n,0})$ and define the exact shadow average
\[
    z_e^{n+1}
    :=
    \mathcal E_T(z^{n,0}).
\]
When $z_e^{n+1}$ and $z_n^\star$ lie in the quadratic-growth neighborhood,
Assumption~\ref{ass:smoothed-qg} and
Proposition~\ref{prop:one-epoch-estimates} give
\[
    \operatorname{dist}^2(z_e^{n+1},Z^\star) \le \alpha_\xi^{-1}G_\xi(z_e^{n+1};z_n^\star) \le \frac{\lambda_+}{\alpha_\xi T} \|z^{n,0}-z_n^\star\|^2 = q_0^2d_n^2.
\]
Consequently, the triangle inequality and
Lemma~\ref{lem:theory-epoch-perturbation} yield
\[
    d_{n+1} \le \|z^{n+1,0}-z_e^{n+1}\| + \operatorname{dist}(z_e^{n+1},Z^\star) \le C_{\rm ep}\varepsilon_n+q_0d_n,
\]
so, whenever the local estimate is applicable,
\begin{equation}
\label{eq:theory-inexact-recurrence}
    d_{n+1}
    \le
    q_0d_n+C_{\rm ep}\varepsilon_n.
\end{equation}

We next prove local retention under
\eqref{eq:general-relative-error}--\eqref{eq:inexact-retention-ball}.  At
$n=0$, the rate bound is an equality and the displacement bound below holds
because its defining sum is empty.  Suppose inductively that
$d_t\le q_\eta^td_0$ through index $n$ and that
\[
    \|z^{n,0}-z^{0,0}\|
    \le
    \beta_\eta\sum_{t=0}^{n-1}d_t.
\]
The center and the exact shadow average satisfy
\[
    \|z_n^\star-z^{0,0}\|
    \le
    \beta_\eta\sum_{t=0}^{n-1}d_t+d_n,
\]
and
\[
    \|z_e^{n+1}-z^{0,0}\|
    \le
    \beta_\eta\sum_{t=0}^{n-1}d_t
    +(1+\kappa_P)d_n.
\]
Here the second estimate uses
Proposition~\ref{prop:one-epoch-estimates}.  Since
$\beta_\eta\ge1+\kappa_P$, both right-hand sides are at most
\[
    \beta_\eta\sum_{t=0}^{n}q_\eta^td_0
    \le
    R_\eta.
\]
Thus \eqref{eq:inexact-retention-ball} permits the use of
\eqref{eq:theory-inexact-recurrence}, which gives
\[
    d_{n+1}
    \le
    (q_0+C_{\rm ep}\eta)d_n
    =
    q_\eta d_n.
\]
Furthermore,
\[
    \|z^{n+1,0}-z^{n,0}\| \le \|z^{n+1,0}-z_e^{n+1}\| + \|z_e^{n+1}-z_n^\star\| + \|z_n^\star-z^{n,0}\| \le C_{\rm ep}\varepsilon_n+(1+\kappa_P)d_n \le \beta_\eta d_n.
\]
Adding this estimate to the preceding displacement bound gives
\[
    \|z^{n+1,0}-z^{0,0}\|
    \le
    \beta_\eta\sum_{t=0}^{n}d_t.
\]
This closes the retention argument and proves
\eqref{eq:general-relative-rate}.
\end{proof}

\subsection{Proof of Proposition~\ref{prop:fixed-pg-work}}
\label{app:proof-fixed-pg-work}

\begin{proof}
For a proximal input $v$, let $u^0\in K$ and define the projected-gradient
iteration
\[
    u^{j+1}
    =
    \Pi_K\left(
        u^j-\eta_{\rm PG}
        [Qu^j+c+\tau^{-1}(u^j-v)]
    \right).
\]
Projection onto $K$ is nonexpansive and
\[
    \|I-\eta_{\rm PG}H_\tau\|
    =
    \max_{\lambda\in[\mu_\tau,L_\tau]}
    |1-\eta_{\rm PG}\lambda|
    =
    \rho_{\rm PG}.
\]
The exact proximal point is the unique fixed point of the PG map, and hence
\begin{equation}
\label{eq:theory-pg-contraction}
    \|u^J-p(v)\|
    \le
    \rho_{\rm PG}^J\|u^0-p(v)\|.
\end{equation}
If $T_v^J(u)$ denotes $J$ such steps, the same nonexpansiveness argument
gives
\begin{equation}
\label{eq:theory-pg-stability}
    \|T_v^J(u)-T_{v'}^J(u')\| \le \rho_{\rm PG}^J\|u-u'\| + \eta_{\rm PG}\tau^{-1} \frac{1-\rho_{\rm PG}^J}{1-\rho_{\rm PG}} \|v-v'\|.
\end{equation}

The difference between consecutive PG iterates also supplies a computable
stationarity certificate.  Let
$\Delta u^j:=u^{j+1}-u^j$ and define
\[
    r^{j+1}
    :=
    \nabla H_v(u^{j+1})
    +
    \eta_{\rm PG}^{-1}(u^j-u^{j+1})
    -
    \nabla H_v(u^j).
\]
The projection step implies
\[
    \eta_{\rm PG}^{-1}(u^j-u^{j+1})
    -
    \nabla H_v(u^j)
    \in
    N_K(u^{j+1}),
\]
and therefore
\begin{equation}
\label{eq:pg-step-to-stationarity}
    r^{j+1}
    \in
    \nabla H_v(u^{j+1})+N_K(u^{j+1}),
    \qquad
    \|r^{j+1}\|
    \le
    (L_\tau+\eta_{\rm PG}^{-1})\|\Delta u^j\|.
\end{equation}
Combining this estimate with
\eqref{eq:general-residual-to-error} yields
\[
    \|u^{j+1}-p(v)\|
    \le
    \mu_\tau^{-1}\|r^{j+1}\|.
\]

It remains to control the warm start uniformly over a local epoch.  At every
KKT point the warm start is the exact proximal point, so the approximate PG
step fixes the KKT point.  The affine dual update and extrapolation are
Lipschitz, while \eqref{eq:theory-pg-stability} has a
parameter-Lipschitz constant bounded uniformly in $J$.  Induction over the $T$ steps of one epoch therefore gives
\[
    \|\widetilde z^{n,k}-z_n^\star\|
    \le
    C_{\rm in}d_n,
    \qquad
    z_n^\star\in\Pi_{Z^\star}(z^{n,0}).
\]
Nonexpansiveness of the exact proximal map then yields
\[
    \|\widetilde x^{n,k}-p_{\rm sh}^{n,k+1}\|
    \le
    (2+\tau\|A\|)C_{\rm in}d_n
    =:
    C_wd_n.
\]
Applying \eqref{eq:theory-pg-contraction} to each proximal subproblem gives
$\delta_{n,k}\le C_w\rho_{\rm PG}^Jd_n$.  Summing over the $T$ inner steps
proves
\[
    \varepsilon_n
    \le
    T C_w\rho_{\rm PG}^Jd_n,
\]
which is \eqref{eq:theory-fixed-j-error}.  Theorem
\ref{thm:inexact-local-linear} with
$\eta=T C_w\rho_{\rm PG}^J$ has contraction factor
\[
    q_J
    :=
    q_0+C_{\rm ep}T C_w\rho_{\rm PG}^J.
\]
Thus $q_J<1$ gives local Q-linear convergence.  For
$0<\rho_{\rm PG}<1$, the explicit sufficient condition
\begin{equation}
\label{eq:theory-fixed-j-choice}
    J
    >
    \frac{
        \log\!\left(C_{\rm ep}T C_w/(1-q_0)\right)
    }{
        -\log\rho_{\rm PG}
    }
\end{equation}
ensures $q_J<1$.  If $\rho_{\rm PG}=0$, one PG step is exact.
\end{proof}

\section{Proofs in Section~\ref{sec:conic-geometry}}
\label{app:sec3-proofs}

\subsection{Proof of Lemma~\ref{lem:kkt-solution-geometry}}
\label{app:proof-kkt-solution-geometry}

\begin{proof}
Take $x_1^\star,x_2^\star\in X^\star$ and set
$d=x_1^\star-x_2^\star$.  Since the feasible set is convex and both
points are optimal, the objective is constant on the segment joining them.
The second directional derivative of $f$ along this segment is
$\langle d,Qd\rangle$, and hence $\langle d,Qd\rangle=0$.  Because
$Q\succeq0$, one has $Q^{1/2}d=0$ and therefore $Qd=0$, proving
that $Qx_1^\star=Qx_2^\star$.

Take two KKT pairs $(x_1^\star,y_1^\star)$ and
$(x_2^\star,y_2^\star)$.  The invariance just established gives
\[
    s(x_2^\star,y_1^\star)
    =
    Qx_2^\star+c-A^*y_1^\star
    =
    s(x_1^\star,y_1^\star)\in K^*.
\]
Let $d=x_2^\star-x_1^\star$.  Feasibility gives $Ad=0$, while equality of
the two objective values and $Qd=0$ give
\[
    0 = f(x_2^\star)-f(x_1^\star)= \langle Qx_1^\star+c,d\rangle = \langle s(x_1^\star,y_1^\star),d\rangle .
\]
Together with complementarity at $x_1^\star$, this proves complementarity
at $x_2^\star$.  Hence $(x_2^\star,y_1^\star)$ is also a KKT pair.
Cross-pairing follows, and the distance identity follows from the Euclidean
product norm.

The same invariance makes $\mathcal D_Q$ independent of
$x^\star$.  Every optimal slack belongs to the intersection in
\eqref{eq:optimal-slack-intersection}.  Conversely, if
$s=Qx^\star+c-A^*y$ belongs to this intersection, then primal feasibility,
dual feasibility, and complementarity hold.  Thus $(x^\star,y)$ is a KKT
pair and $s$ is an optimal slack.
\end{proof}

\subsection{Proof of Theorem~\ref{thm:qg}}
\label{app:proof-qg}

\begin{proof}
We first derive the component identities used in both directions.  Insert
and subtract $\mathcal L(x^\star,y^\star)$ in
\eqref{eq:smoothed-gap}.  The maximizations over $\widehat y$ and
$\widehat x$ separate, and the unused part of each maximization has value
zero at the corresponding component of $z^\star$.  Hence
\begin{equation}
\label{eq:gap-separation}
    G_\xi(z;z^\star)
    =
    P_\xi(x;z^\star)+D_\xi(y;z^\star).
\end{equation}
Completing the square in $\widehat y$ gives
\[
    P_\xi(x;z^\star)
    =
    f(x)-f(x^\star)
    +\langle y^\star,b-Ax\rangle
    +\frac1{2\xi}\|Ax-b\|^2.
\]
Using $Ax^\star=b$, $s^\star=Qx^\star+c-A^*y^\star$, and
$\langle x^\star,s^\star\rangle=0$, we obtain
\begin{equation}
\label{eq:primal-smoothed-gap}
    P_\xi(x;z^\star)
    =
    E_P(x;z^\star)+\frac1{2\xi}\|Ax-b\|^2.
\end{equation}

For the dual component, writing $\widehat x=x^\star+d$ yields
\[
    D_\xi(y;z^\star)
    =
    \sup_{x^\star+d\in K}
    \left\{
        -\langle s^\star(y),d\rangle
        -\frac12\|d\|_{M_\xi}^2
    \right\}.
\]
The unique maximizer is $u_y$, and therefore
\begin{equation}
\label{eq:dual-gap-exact}
    D_\xi(y;z^\star)
    =
    \langle s^\star(y),r_y\rangle
    -\frac12\|r_y\|_{M_\xi}^2.
\end{equation}
Weighted projection optimality gives
\[
    M_\xi r_y-s^\star(y)\in N_K(u_y).
\]
Testing this normal-cone relation with $x^\star\in K$ shows that
$\langle s^\star(y),r_y\rangle\ge\|r_y\|_{M_\xi}^2$, and hence
\begin{equation}
\label{eq:dual-gap-lower-bound}
    D_\xi(y;z^\star)
    \ge
    \frac12\|r_y\|_{M_\xi}^2.
\end{equation}
Moreover, $D_\xi(\cdot;z^\star)$ is convex and differentiable.  Danskin's
theorem and the uniqueness of $u_y$ give
\begin{equation}
\label{eq:dual-gap-gradient}
    \nabla D_\xi(y;z^\star)
    =
    A(u_y-x^\star)
    =
    -Ar_y.
\end{equation}
The same projection condition identifies the zero set of the residual:
\[
    R_\xi(y;z^\star)=0 \Longleftrightarrow -s^\star(y)\in N_K(x^\star) \Longleftrightarrow s^\star(y)\in\mathcal D_Q\cap\mathcal F_{x^\star} =\mathcal S^\star \Longleftrightarrow y\in Y^\star.
\]
In particular, \eqref{eq:dual-gap-exact} gives
$D_\xi(y;z^\star)=0$ for every $y\in Y^\star$.

Suppose first that the two uniform local error bounds hold.  Choose witness
neighborhoods $V_P,V_x,V_D,V_y$ from
Definitions~\ref{def:primal-eb} and \ref{def:dual-eb}, and choose $V$
contained in $V_P\cap V_D\cap(V_x\times V_y)$, shrinking it if necessary so
that both bounds apply to every $z\in(K\times Y)\cap V$ and every
$z^\star\in Z^\star\cap V$.  From \eqref{eq:primal-eb} and
\eqref{eq:primal-smoothed-gap},
\[
    P_\xi(x;z^\star)
    \ge
    \min\left\{\frac1{a_P},\frac1{2\xi b_P}\right\}
    \operatorname{dist}^2(x,X^\star).
\]
Similarly, \eqref{eq:dual-eb} and \eqref{eq:dual-gap-lower-bound} give
\[
    D_\xi(y;z^\star)
    \ge
    \frac1{2\kappa_D^2}
    \operatorname{dist}^2(y,Y^\star).
\]
Thus \eqref{eq:gap-separation} and \eqref{eq:solution-product} prove
\eqref{eq:uniform-gap-qg} with
\[
    \alpha_\xi
    =
    \min\left\{
        \frac1{a_P},
        \frac1{2\xi b_P},
        \frac1{2\kappa_D^2}
    \right\}.
\]

Conversely, suppose \eqref{eq:uniform-gap-qg} holds on a neighborhood
$V_Q$ with constant $\alpha_\xi>0$.  Choose product neighborhoods $U_x$ of
$\bar x$ and $U_y$ of $\bar y$ such that $U_x\times U_y\subseteq V_Q$, and
restrict the KKT centers to $Z^\star\cap(U_x\times U_y)$.  For
$x\in K\cap U_x$, apply \eqref{eq:uniform-gap-qg} to
$z=(x,y^\star)$.  By \eqref{eq:solution-product},
\eqref{eq:primal-smoothed-gap}, and the definition of $P_\xi$,
\[
    \operatorname{dist}^2(x,X^\star) \le \frac1{\alpha_\xi}P_\xi(x;z^\star) = \frac1{\alpha_\xi}E_P(x;z^\star) +\frac1{2\xi\alpha_\xi}\|Ax-b\|^2.
\]
This is the uniform local primal error bound with
$a_P=\alpha_\xi^{-1}$ and $b_P=(2\xi\alpha_\xi)^{-1}$.

For $y\in Y\cap U_y$, apply \eqref{eq:uniform-gap-qg} to
$z=(x^\star,y)$ and let $\widehat y\in\Pi_{Y^\star}(y)$.  By
Lemma~\ref{lem:kkt-solution-geometry}, $(x^\star,\widehat y)$ is a KKT pair, and the
zero-set characterization above gives
$D_\xi(\widehat y;z^\star)=0$.  Convexity and
\eqref{eq:dual-gap-gradient} therefore yield
\[
    \alpha_\xi\operatorname{dist}^2(y,Y^\star) \le D_\xi(y;z^\star) \le \langle\nabla D_\xi(y;z^\star),y-\widehat y\rangle \le \|A M_\xi^{-1/2}\|\, \|r_y\|_{M_\xi}\operatorname{dist}(y,Y^\star).
\]
After dividing when the distance is nonzero, we obtain the uniform local
dual error bound with any positive constant satisfying
\[
    \kappa_D
    \ge
    \frac{\|A M_\xi^{-1/2}\|}{\alpha_\xi}.
\]
When the distance is zero the bound is immediate.  This proves the reverse
implication.  The center neighborhood $U_x\times U_y$, together with $U_x$
and $U_y$ for the primal and dual variables, provides the common witness
neighborhoods required in Definitions~\ref{def:primal-eb} and
\ref{def:dual-eb}.
\end{proof}

\subsection{Proof of Proposition~\ref{prop:lifted-eb}}
\label{app:proof-lifted-eb}

\begin{proof}
The lifted cone $K\times\mathcal Q_r$ is again a finite product of the
listed classes, and the canonical certificates constructed above exhibit
lifted KKT points, so the lifted problem \eqref{eq:soc-lifting} satisfies
the standing strong-duality hypotheses of \cite{ding2023strict}.  By the
transfer of strict complementarity, P-SC of every center in the stratum
becomes lifted P-SC of its canonical certificate, which is the
strict-complementarity hypothesis of \cite[Corollary~1]{ding2023strict}.

Fix $z^\star\in Z^\star\cap V_{\mathrm{SC}}$.  Remark~4 of
\cite{ding2023strict} assembles the violation-of-complementarity function
of the product cone $K\times\mathcal Q_r$ as the sum of the blockwise
functions of their Lemma~3; each summand is linear or of square-root type
in $\langle\zeta^\star,\chi\rangle$, with coefficients determined by the
blockwise spectra of the components of $\zeta^\star$ (smallest nonzero
entries or eigenvalues; the RSOC block enters through
$p^\star=(1,u^\star,-Bx^\star)$, whose norm is at least one).  Remark~1
bounds the norm arguments of these functions on the bounded neighborhood
$\mathcal N$, and \cite[Corollary~1]{ding2023strict} then yields, for every
$\chi\in(K\times\mathcal Q_r)\cap\mathcal N$,
\[
    \operatorname{dist}(\chi,\mathcal X_L^\star)
    \le
    c_1(z^\star)\langle\zeta^\star,\chi\rangle
    +c_2(z^\star)\langle\zeta^\star,\chi\rangle^{1/2}
    +c_3(z^\star)\|\mathcal A_L\chi-\widetilde b\|.
\]
Since $\langle\zeta^\star,\chi\rangle$ is bounded on $\mathcal N$, the
linear term is absorbed into the square-root term, which is
\eqref{eq:lifted-residual-error-bound} at the fixed center.

It remains to choose the constants uniformly.  The mirror image of the
cross-complementarity argument of
Corollary~\ref{cor:dual-sc-slack-regularity} shows that P-SC at every
center fixes the primal complementary faces $K\cap(s^\star)^\perp$ along
the stratum, and with them the blockwise supports, boundary rays, and
ranks of the slacks $s^\star$.  The blockwise smallest nonzero entries and
eigenvalues therefore vary continuously with $z^\star$ and remain bounded
below after shrinking $V_{\mathrm{SC}}$, while the complementary
subspaces entering $c_3$, the map $\mathcal A_L$, and the norm bound on
$\mathcal N$ are fixed.  Hence
$\sup_{z^\star\in Z^\star\cap V_{\mathrm{SC}}}c_i(z^\star)<\infty$ for
$i=1,2,3$, and the constants in
\eqref{eq:lifted-residual-error-bound} can be chosen uniformly.
\end{proof}

\subsection{Proof of Corollary~\ref{cor:primal-eb-strict-complementarity}}
\label{app:proof-primal-eb-strict-complementarity}

\begin{proof}
Lemma~\ref{lem:kkt-solution-geometry} and $Q=B^*B$ imply that $Bx^\star$
is constant on $X^\star$.  Denote this value by $\bar w$ and set
$\bar u:=\frac12\|\bar w\|^2$.  Tightness of the lifted epigraph gives
\begin{equation}
\label{eq:lifted-solution-set}
    \mathcal X_L^\star
    =
    X^\star\times\{(\bar u,1,\bar w)\},
\end{equation}
and hence
\begin{equation}
\label{eq:lifted-distance-projection}
    \operatorname{dist}^2(\widehat\chi(x),\mathcal X_L^\star)
    =
    \operatorname{dist}^2(x,X^\star)
    +\|q(x)-(\bar u,1,\bar w)\|^2
    \ge
    \operatorname{dist}^2(x,X^\star).
\end{equation}

The map $\widehat\chi$ is continuous, so
$V_x:=\widehat\chi^{-1}(\mathcal N)$ is a common neighborhood of $\bar x$.
For $x\in K\cap V_x$, the lifted point is cone feasible and
\[
    \mathcal A_L\widehat\chi(x)-\widetilde b=(Ax-b,0,0).
\]
For every $x\in K$,
\begin{equation}
\label{eq:objective-complementarity}
    \langle p^\star,q(x)\rangle = \frac12\|Bx\|^2+\frac12\|Bx^\star\|^2 -\langle Bx^\star,Bx\rangle = \frac12\|x-x^\star\|_Q^2,
\end{equation}
and hence
\begin{equation}
\label{eq:lifted-complementarity-ep}
    \langle\zeta^\star,\widehat\chi(x)\rangle
    =
    \langle s^\star,x\rangle+\langle p^\star,q(x)\rangle
    =
    E_P(x;z^\star).
\end{equation}
Combining \eqref{eq:lifted-complementarity-ep},
\eqref{eq:lifted-distance-projection}, and
\eqref{eq:lifted-residual-error-bound} gives
\[
    \operatorname{dist}(x,X^\star)
    \le
    \kappa_L\sqrt{E_P(x;z^\star)}
    +\gamma_L\|Ax-b\|.
\]
Squaring and using $(a+b)^2\le2a^2+2b^2$ yields
\[
    \operatorname{dist}^2(x,X^\star)
    \le
    2\kappa_L^2E_P(x;z^\star)
    +2\gamma_L^2\|Ax-b\|^2.
\]
Taking $V_P:=V_{\mathrm{SC}}$ and $V_x$ as above proves the stated
constants uniformly over the local stratum.
\end{proof}

\subsection{Proof of Corollary~\ref{cor:dual-sc-slack-regularity}}
\label{app:proof-dual-sc-slack-regularity}

\begin{proof}
Every optimal slack lies in the affine set $\mathcal D_Q$, whose relative
interior is itself.  Thus \eqref{eq:dual-side-strict-complementarity} gives
$\operatorname{ri}\mathcal D_Q\cap
\operatorname{ri}\mathcal F_{x^\star}\neq\emptyset$, which implies bounded,
and hence local, linear regularity \cite{bauschke1999strong}.

For the uniform claim, take two primal centers $x_1^\star,x_2^\star$ and
D-SC optimal slacks
$s_i^\star\in\operatorname{ri}\mathcal F_{x_i^\star}$.  The
cross-complementarity in Lemma~\ref{lem:kkt-solution-geometry} gives
$s_1^\star\in\mathcal F_{x_2^\star}$ and
$s_2^\star\in\mathcal F_{x_1^\star}$.  Since each slack is a
relative-interior point of its face, these inclusions imply
\[
    \mathcal F_{x_1^\star}\subseteq\mathcal F_{x_2^\star},
    \qquad
    \mathcal F_{x_2^\star}\subseteq\mathcal F_{x_1^\star}.
\]
Hence all complementary faces on the stratum coincide with a fixed face
$\mathcal F$.  Bounded linear regularity of the fixed pair
$\{\mathcal D_Q,\mathcal F\}$ supplies common $W_A$ and $\kappa_A$, proving
the uniform claim.
\end{proof}

\subsection{Proof of Proposition~\ref{prop:uniform-normal-calmness}}
\label{app:proof-uniform-normal-calmness}

\begin{proof}
The graph of $N_K$ is the product of the blockwise graphs and squared
distances add across blocks, so it suffices to produce common neighborhoods
and moduli for each block; the product then takes the smallest radius and
the largest modulus.  For a polyhedral block,
$\operatorname{gph}N_{K_j}$ is a finite union of polyhedra, so the
upper-Lipschitz modulus of the polyhedral multifunction $N_{K_j}$ is
determined by its finitely many affine pieces and is valid, with a common
neighborhood, at every reference point \cite{robinson1981continuity}.

For the remaining blocks, the target set does not move: for a cone,
$N_{K_j}(x_j^\star)
=-\bigl(K_j^*\cap(x_j^\star)^\perp\bigr)
=-\mathcal F_{x_j^\star}$
depends on $x_j^\star$ only through its complementary face, and the
cross-complementarity argument of
Corollary~\ref{cor:dual-sc-slack-regularity} shows that blockwise D-SC
fixes these faces along the stratum: all local centers share one face
$\mathcal F_j$ per block.  Blocks with $\bar x_j$ interior to $K_j$, or
with $\bar s_j$ interior to $K_j^*$, are trivial: in the first case nearby
graph points satisfy $v_j=0\in N_{K_j}(x_j^\star)$, and in the second case
$x_j^\star=0$, so
$v_j\in N_{K_j}(u_j)\subseteq-K_j^*=N_{K_j}(x_j^\star)$; in both cases the
left-hand side of \eqref{eq:normal-calmness} vanishes.

Consider a PSD block on the boundary.  Constancy of $\mathcal F_j$ fixes
$\ker x_j^\star$ and hence the rank, so the smallest nonzero eigenvalue
$\lambda^+_{\min}(x_j^\star)$ is bounded below by some $\lambda_j>0$ after
shrinking the stratum neighborhood.  A graph point $(u_j,v_j)$ satisfies
$u_j\succeq0$, $v_j\preceq0$, and $u_jv_j=0$.  Decompose both matrices in
blocks adapted to
$\operatorname{range}(x_j^\star)\oplus\ker(x_j^\star)$ and let
$v_j^{11},v_j^{12},v_j^{22}$ denote the blocks of $v_j$.  Zeroing
$v_j^{11}$ and $v_j^{12}$ produces a matrix of $-\mathcal F_j$, so
$\operatorname{dist}(v_j,-\mathcal F_j)\le\|v_j^{11}\|+2\|v_j^{12}\|$.  For
$\|u_j-x_j^\star\|\le\lambda_j/2$, the restriction of $u_j$ to
$\operatorname{range}(x_j^\star)$ is invertible with inverse bounded by
$2/\lambda_j$, and the first block row of the identity $u_jv_j=0$
give
\[
    \|v_j^{11}\|+2\|v_j^{12}\|
    \le
    \frac{C\|v_j\|}{\lambda_j}\,\|u_j-x_j^\star\|
\]
for an absolute constant $C$.  Since $\|v_j\|$ is bounded on the graph
neighborhood of $(\bar x_j,-\bar s_j)$, this is
\eqref{eq:normal-calmness} with a modulus depending only on $\lambda_j$
and the neighborhood.  A boundary SOC block is the same computation with
$\lambda^+_{\min}(x_j^\star)$ replaced by the distance of $x_j^\star$ from
the cone vertex, which is bounded below because the fixed face pins the
boundary ray carrying $x_j^\star$; a rotated second-order cone reduces to
this case through a linear isometry.  Taking the smallest blockwise radius
and the largest blockwise modulus completes the proof.
\end{proof}

\subsection{Proof of Proposition~\ref{prop:dual-eb}}
\label{app:proof-dual-eb}

\begin{proof}
For the projection point and residual defined above, recall that
$v_y=M_\xi r_y-s^\star(y)$.  The maps
$(z^\star,y)\mapsto s^\star(y)$ and
$(z^\star,y)\mapsto(u_y,v_y)$ are jointly continuous and, at $y=y^\star$,
equal $s^\star$ and $(x^\star,-s^\star)$, respectively.  After shrinking
common neighborhoods $V_D$ of $\bar z$ and $V_y$ of $\bar y$, both
assumptions therefore apply to every
$z^\star\in Z^\star\cap V_D$ and $y\in Y\cap V_y$.

Weighted projection optimality gives
$v_y\in N_K(u_y)$.  Since
$N_K(x^\star)=-\mathcal F_{x^\star}$,
\[
    \operatorname{dist}(s^\star(y),\mathcal F_{x^\star}) = \operatorname{dist}(-s^\star(y),N_K(x^\star)) \le \|M_\xi r_y\|+\kappa_B\|r_y\| \le \frac{L_\xi+\kappa_B}{\sqrt{\lambda_\xi}} \|r_y\|_{M_\xi}.
\]

The vector $s^\star(y)$ already lies in $\mathcal D_Q$.  Therefore
Assumption~\ref{ass:slack-regularity} gives
\[
    \operatorname{dist}(s^\star(y),\mathcal S^\star)
    \le
    \kappa_A
    \frac{L_\xi+\kappa_B}{\sqrt{\lambda_\xi}}
    \|r_y\|_{M_\xi}.
\]
Let $\widehat s\in\Pi_{\mathcal S^\star}(s^\star(y))$.  Both slacks belong to
$\mathcal D_Q$, so their difference lies in $\operatorname{range}A^*$.
Define
\[
    d:=(A^*)^\dagger(s^\star(y)-\widehat s),
    \qquad
    \widehat y:=y+d.
\]
Then $s^\star(\widehat y)=\widehat s$, and hence
$\widehat y\in Y^\star$.  Consequently,
\[
    \operatorname{dist}(y,Y^\star)
    \le
    \|d\|
    \le
    \|(A^*)^\dagger\|
    \operatorname{dist}(s^\star(y),\mathcal S^\star),
\]
which proves \eqref{eq:dual-eb}.
\end{proof}

\subsection{Proof of Corollary~\ref{cor:common-cone-qg}}
\label{app:proof-common-cone-qg}

\begin{proof}
The listed blocks are symmetric cones, up to the linear isometry carrying
rotated Lorentz cones onto Lorentz cones, so P-SC and D-SC are each
equivalent, blockwise, to the displayed condition
$x^\star+s^\star\in\operatorname{int}K$
\cite{chua2008invariance,faraut1994analysis}.  Hence every local KKT
center satisfies P-SC, and every local primal center admits a blockwise
D-SC optimal slack: by Lemma~\ref{lem:kkt-solution-geometry}, each
$x^\star\in X^\star$ near $\bar x$ pairs with $\bar y$ to form a KKT
center in the stratum, whose optimal slack is then D-SC.  The first
property and Proposition~\ref{prop:lifted-eb} supply the lifted residual
error bound of Assumption~\ref{ass:lifted-residual-eb}, and
Corollary~\ref{cor:primal-eb-strict-complementarity} converts it into the
uniform primal error bound with $a_P=2\kappa_L^2$ and
$b_P=2\gamma_L^2$.  The second property,
Corollary~\ref{cor:dual-sc-slack-regularity}, and
Proposition~\ref{prop:uniform-normal-calmness} verify
Assumptions~\ref{ass:slack-regularity} and \ref{ass:normal-calmness}.
Proposition~\ref{prop:dual-eb} then gives the
uniform dual error bound, and Theorem~\ref{thm:qg} proves the asserted
growth estimate.  Substituting the primal constants and the value of
$\kappa_D$ from Proposition~\ref{prop:dual-eb} gives
\eqref{eq:common-cone-qg-constant}.
\end{proof}

\section{Cone Projections under the Euclidean Norm and the
\texorpdfstring{$M$}{M}-Norm}
\label{app:cone-projections}

This section gives the projections under the Euclidean norm and the $M$-norm
used in Section~\ref{sec:cone-projections}.  Let
$M=\operatorname{diag}(m)\succ0$.
For any closed convex cone $\mathcal K$, the membership conditions
\[
    r\in\mathcal K
    \qquad\text{and}\qquad
    -Mr\in\mathcal K^*
\]
give $\Pi_{\mathcal K}^M(r)=r$ and
$\Pi_{\mathcal K}^M(r)=0$, respectively.  In all remaining cases, the
projected point lies on the boundary of $\mathcal K$.
For a diagonal metric, the change of variables $w=M^{1/2}u$ gives
\begin{equation}
\label{eq:metric-rescaled-cone-equivalence}
    \Pi_{\mathcal K}^{M}(r)
    =M^{-1/2}\Pi_{M^{1/2}\mathcal K}(M^{1/2}r).
\end{equation}
Consequently, the $M$-norm projections onto the second-order and exponential
cones are precisely the Euclidean projections onto the diagonally rescaled
cones analyzed in PDCS~\cite{lin2025pdcs}. We record their formulas below in
the original cone coordinates. For the power cone, we derive the Euclidean
and $M$-norm scalar equations and their bisection procedure explicitly.

\begin{table}[H]
    \centering
    \papertablesetup
    \caption{Projection methods for the box and cone blocks under the
    Euclidean norm and the $M$-norm.}
    \label{tab:cone-projection-methods}
    \medskip
    \begin{tabular*}{\textwidth}{@{\extracolsep{\fill}}lcc@{}}
        \toprule
        Set or cone
        & \shortstack{Euclidean\\projection}
        & \shortstack{$M$-norm\\projection} \\
        \midrule
        Box constraints              & Closed form & Closed form \\
        Nonnegative cone             & Closed form & Closed form \\
        Second-order cone            & Closed form & Scalar Root-finding \\
        Rotated second-order cone    & Closed form & Scalar Root-finding \\
        Exponential cone             & Scalar Root-finding & Scalar Root-finding \\
        Three-dimensional power cone & Scalar Root-finding & Scalar Root-finding \\
        \bottomrule
    \end{tabular*}
\end{table}

\subsection{Box and Nonnegative Cone}

The Euclidean and $M$-norm projection problems both separate by coordinate
and have the same solution:
\begin{equation}
\label{eq:box-nonnegative-metric-projection}
    [\Pi_{\mathbb R_+^d}^{M}(r)]_i=\max\{r_i,0\},
    \qquad
    [\Pi_{[\ell,u]}^{M}(r)]_i
    =\min\{u_i,\max\{\ell_i,r_i\}\}.
\end{equation}
Thus their $M$-norm and Euclidean projections coincide.

\subsection{Second-Order Cone}

\paragraph{Euclidean projection.}
For $r=(r_v,r_t)\in\mathbb R^d\times\mathbb R$, let
$a:=\|r_v\|_2$.  The Euclidean projection is
\begin{equation}
\label{eq:soc-euclidean-projection}
    \Pi_{\mathcal K_{\mathrm{soc}}}(r_v,r_t)
    =
    \begin{cases}
        (r_v,r_t), & a\le r_t,\\
        0, & a\le-r_t,\\
        \displaystyle
        \left(\frac{a+r_t}{2a}r_v,\frac{a+r_t}{2}\right),
        & \text{otherwise}.
\end{cases}
\end{equation}

\paragraph{\texorpdfstring{$M$}{M}-norm projection.}
For $M=\operatorname{diag}(m_{v,1},\ldots,m_{v,d},m_t)$, the nontrivial
boundary case is determined by a scalar $\lambda\ge0$. In the original cone
coordinates, the rescaled-cone characterization of PDCS~\cite{lin2025pdcs}
becomes
\begin{equation}
\label{eq:soc-metric-projection}
    v_i(\lambda)
    =\frac{m_{v,i}(r_v)_i}{m_{v,i}+\lambda},
    \qquad
    t(\lambda)=\frac{m_tr_t}{m_t-\lambda},
    \qquad
    \|v(\lambda)\|_2^2-t(\lambda)^2=0.
\end{equation}
If $r_t>0$, the admissible root lies in $(0,m_t)$; if $r_t<0$, it lies in
$(m_t,\infty)$.  When $r_t=0$, the singular case $\lambda=m_t$ gives
$v_i=m_{v,i}(r_v)_i/(m_{v,i}+m_t)$ and $t=\|v\|_2$.  The scalar equation is
solved by bisection; the endpoint signs and the resulting bisection intervals
are established in~\cite{lin2025pdcs}.

\subsection{Rotated Second-Order Cone}

\paragraph{Euclidean projection.}
Define the orthogonal map
\[
    U(v,s,t)
    :=
    \left(\left(v,\frac{s-t}{\sqrt2}\right),
          \frac{s+t}{\sqrt2}\right).
\]
It maps $\mathcal K_{\mathrm{rsoc}}$ onto a second-order cone, and hence
\begin{equation}
\label{eq:rsoc-via-soc}
    \Pi_{\mathcal K_{\mathrm{rsoc}}}(r)
    =U^{-1}\Pi_{\mathcal K_{\mathrm{soc}}}(Ur).
\end{equation}

\paragraph{\texorpdfstring{$M$}{M}-norm projection.}
For $r=(r_v,r_s,r_t)$ and
$M=\operatorname{diag}(m_{v,1},\ldots,m_{v,d},m_s,m_t)$, a nonzero boundary
projection is determined by $\xi\ge0$ through
\begin{equation}
\label{eq:rsoc-metric-projection}
\begin{aligned}
    v_i(\xi)
    &=\frac{m_{v,i}(r_v)_i}{m_{v,i}+2\xi},\\
    \begin{bmatrix}m_s&-2\xi\\-2\xi&m_t\end{bmatrix}
    \begin{bmatrix}s(\xi)\\t(\xi)\end{bmatrix}
    &=
    \begin{bmatrix}m_sr_s\\m_tr_t\end{bmatrix},\\
    0&=\|v(\xi)\|_2^2-2s(\xi)t(\xi).
\end{aligned}
\end{equation}
The $2\times2$ system is singular at
$2\xi=q:=\sqrt{m_sm_t}$.  When its right-hand side is consistent, namely,
$\sqrt{m_s}r_s+\sqrt{m_t}r_t=0$, the singular solution is obtained explicitly.
Define
\[
    \delta:=\sqrt{m_s}r_s=-\sqrt{m_t}r_t,
    \qquad
    v_i:=\frac{m_{v,i}(r_v)_i}{m_{v,i}+q},
    \qquad
    \chi:=\sqrt{\delta^2+2q\|v\|_2^2}.
\]
Then
\begin{equation}
\label{eq:rsoc-metric-singular-projection}
    s=\frac{\delta+\chi}{2\sqrt{m_s}},
    \qquad
    t=\frac{-\delta+\chi}{2\sqrt{m_t}}
\end{equation}
satisfies $\|v\|_2^2=2st$ and also covers projections onto either coordinate ray.
For nonsingular candidates, the admissible root may lie on either side of
$q/2$.  After the membership and zero-projection tests, the implementation
brackets a sign change on an interval satisfying $s(\xi),t(\xi)\ge0$ and
applies bisection. Restricting the search to the lower branch can miss valid
projections.

\subsection{Exponential Cone}

Let $r=(r_0,s_0,t_0)$. In addition to the membership and dual-cone cases,
$r_0\le0$ and $s_0\le0$ give the same projection in both metrics:
\[
    \Pi_{\mathcal K_{\mathrm{exp}}}(r)
    =\Pi_{\mathcal K_{\mathrm{exp}}}^{M}(r)
    =(r_0,0,\max\{t_0,0\}).
\]

\paragraph{Euclidean projection.}
In the remaining case, a point on the smooth boundary is parameterized as
$\eta_I(\rho)(\rho,1,e^\rho)$ with $\eta_I(\rho)>0$, where
\begin{equation}
\label{eq:exp-euclidean-projection}
\begin{aligned}
    \eta_I(\rho)
    &:={
    \frac{\rho r_0+s_0+e^\rho t_0}
         {\rho^2+1+e^{2\rho}}},\\
    h_I(\rho)
    &:={\rho\eta_I(\rho)-r_0
      +e^\rho\bigl(e^\rho\eta_I(\rho)-t_0\bigr)}.
\end{aligned}
\end{equation}
The Euclidean projection is obtained by finding a root of $h_I$ and returning
$\eta_I(\rho)(\rho,1,e^\rho)$.

\paragraph{\texorpdfstring{$M$}{M}-norm projection.}
Let $M=\operatorname{diag}(m_r,m_s,m_t)$. For the smooth-boundary case, define
\begin{equation}
\label{eq:exp-metric-projection}
\begin{aligned}
    \eta_M(\rho)
    &:={
    \frac{m_r\rho r_0+m_ss_0+m_te^\rho t_0}
         {m_r\rho^2+m_s+m_te^{2\rho}}},\\
    h_M(\rho)
    &:={m_r\bigl(\rho\eta_M(\rho)-r_0\bigr)
      +m_te^\rho\bigl(e^\rho\eta_M(\rho)-t_0\bigr)}.
\end{aligned}
\end{equation}
The projection is obtained from an admissible root of $h_M$ satisfying
$\eta_M(\rho)>0$ and equals $\eta_M(\rho)(\rho,1,e^\rho)$. This is the
original-coordinate form of the diagonally rescaled exponential-cone
projection in PDCS~\cite{lin2025pdcs}, which gives an explicit admissible
interval and the corresponding bisection procedure. A safeguarded Newton step
can be used within the same bracket.

\subsection{Three-Dimensional Power Cone}

Let
\[
    \mathcal K_{\mathrm{pow}}^\alpha
    :=\{(x,y,z):x\ge0,\ y\ge0,\ x^\alpha y^{1-\alpha}\ge |z|\},
    \qquad \alpha\in(0,1).
\]
Its dual cone is
\[
    (\mathcal K_{\mathrm{pow}}^\alpha)^*
    =\left\{(p,q,w):p,q\ge0,\
    \left(\frac{p}{\alpha}\right)^\alpha
    \left(\frac{q}{1-\alpha}\right)^{1-\alpha}\ge |w|\right\}.
\]
Let $r=(r_x,r_y,r_z)$. If $r_z=0$, the Euclidean and $M$-norm projections
coincide:
\[
    \Pi_{\mathcal K_{\mathrm{pow}}^\alpha}(r)
    =\Pi_{\mathcal K_{\mathrm{pow}}^\alpha}^{M}(r)
    =\bigl(\max\{r_x,0\},\max\{r_y,0\},0\bigr).
\]
Set $R:=|r_z|$ and suppose below that $R>0$. The membership and
zero-projection tests are applied separately for each metric.

\paragraph{Euclidean projection.}
Suppose that $r\notin\mathcal K_{\mathrm{pow}}^\alpha$ and
$-r\notin(\mathcal K_{\mathrm{pow}}^\alpha)^*$.
The nonzero boundary point is parameterized by
$\rho=|u_z|\in(0,R)$ through
\begin{equation}
\label{eq:power-euclidean-projection}
\begin{aligned}
    x_I(\rho)
    &=\frac12\left(
        r_x+\sqrt{r_x^2+4\alpha\rho(R-\rho)}
      \right),\\
    y_I(\rho)
    &=\frac12\left(
        r_y+\sqrt{r_y^2+4(1-\alpha)\rho(R-\rho)}
      \right),\\
    g_I(\rho)
    &:=\alpha\log x_I(\rho)
      +(1-\alpha)\log y_I(\rho)-\log\rho.
\end{aligned}
\end{equation}
If $\rho_I^*$ is the zero of $g_I$ specified below, then
\begin{equation}
\label{eq:power-euclidean-projection-recovery}
    \Pi_{\mathcal K_{\mathrm{pow}}^\alpha}(r)
    =\bigl(x_I(\rho_I^*),y_I(\rho_I^*),
      \operatorname{sign}(r_z)\rho_I^*\bigr).
\end{equation}

\paragraph{\texorpdfstring{$M$}{M}-norm projection.}
For $M=\operatorname{diag}(m_x,m_y,m_z)$, suppose that
$r\notin\mathcal K_{\mathrm{pow}}^\alpha$ and
$-Mr\notin(\mathcal K_{\mathrm{pow}}^\alpha)^*$. The same parameterization gives
\begin{equation}
\label{eq:power-projection-root}
\begin{aligned}
    x_M(\rho)
    &=\frac12\left(
        r_x+\sqrt{r_x^2
        +4\alpha\frac{m_z}{m_x}\rho(R-\rho)}
      \right),\\
    y_M(\rho)
    &=\frac12\left(
        r_y+\sqrt{r_y^2
        +4(1-\alpha)\frac{m_z}{m_y}\rho(R-\rho)}
      \right),\\
    g_M(\rho)
    &:=\alpha\log x_M(\rho)
      +(1-\alpha)\log y_M(\rho)-\log\rho.
\end{aligned}
\end{equation}
For stable numerical evaluation of the positive quadratic roots, define
\begin{equation}
\label{eq:power-cone-stable-quadratic-root}
    q(s,c):=
    \begin{cases}
        \dfrac{s+\sqrt{s^2+4c}}{2}, & s\ge0,\\[6pt]
        \dfrac{2c}{\sqrt{s^2+4c}-s}, & s<0.
    \end{cases}
\end{equation}
Thus, for example,
$x_M(\rho)=q(r_x,\alpha(m_z/m_x)\rho(R-\rho))$, with analogous expressions
for $y_M$, $x_I$, and $y_I$. The second branch avoids cancellation when the
input coordinate is negative and $\rho$ is close to an endpoint.

If $\rho_M^*$ is the zero of $g_M$ specified below, then
\begin{equation}
\label{eq:power-metric-projection-recovery}
    \Pi_{\mathcal K_{\mathrm{pow}}^\alpha}^{M}(r)
    =\bigl(x_M(\rho_M^*),y_M(\rho_M^*),
      \operatorname{sign}(r_z)\rho_M^*\bigr).
\end{equation}

\begin{proposition}[Bisection for the power-cone projection]
\label{prop:power-cone-bisection}
Let $M=\operatorname{diag}(m_x,m_y,m_z)\succ0$, $R=|r_z|>0$,
$r\notin\mathcal K_{\mathrm{pow}}^\alpha$, and
$-Mr\notin(\mathcal K_{\mathrm{pow}}^\alpha)^*$. Then $g_M$ has a unique zero
$\rho_M^*\in(0,R)$. Initialize $\rho_L=0$ and $\rho_U=R$, using the endpoint
limits established below. Each bisection step sets
$\rho=(\rho_L+\rho_U)/2$ and updates
\begin{equation}
\label{eq:power-cone-bisection-update}
    \begin{cases}
        \rho_L\leftarrow\rho, & g_M(\rho)>0,\\
        \rho_U\leftarrow\rho, & g_M(\rho)<0.
    \end{cases}
\end{equation}
For a prescribed scalar tolerance $\epsilon_{\mathrm{bis}}>0$, the iterations
stop when $\rho_U-\rho_L\le
\epsilon_{\mathrm{bis}}\max\{1,R\}$, and the midpoint is used in
\eqref{eq:power-metric-projection-recovery}.
The Euclidean procedure is obtained by setting $M=I$ and replacing
$g_M$ by $g_I$.
\end{proposition}

\begin{proof}
Symmetry in the third coordinate allows a nonzero boundary candidate to be
written as $u_z=\operatorname{sign}(r_z)\rho$, where
$x,y,\rho>0$ and $x^\alpha y^{1-\alpha}=\rho$. Its KKT conditions with
multiplier $\lambda>0$ are
\begin{equation}
\label{eq:power-cone-projection-kkt}
\begin{aligned}
    m_x(x-r_x)&=\lambda\alpha\frac{\rho}{x},\\
    m_y(y-r_y)&=\lambda(1-\alpha)\frac{\rho}{y},\\
    m_z(\rho-R)+\lambda&=0,\\
    x^\alpha y^{1-\alpha}&=\rho.
\end{aligned}
\end{equation}
Eliminating $\lambda=m_z(R-\rho)$ and taking the positive roots of the first
two quadratic equations yields $x_M(\rho)$ and $y_M(\rho)$ in
\eqref{eq:power-projection-root}. Thus every zero of $g_M$ in $(0,R)$
satisfies the KKT conditions and recovers the projection.

At the upper endpoint, with $\log 0:=-\infty$,
\begin{equation}
\label{eq:power-cone-upper-endpoint-limit}
    \lim_{\rho\uparrow R}g_M(\rho)
    =\log\left(\frac{(r_x)_+^\alpha(r_y)_+^{1-\alpha}}{R}\right)<0,
\end{equation}
where the strict inequality follows from
$r\notin\mathcal K_{\mathrm{pow}}^\alpha$. As $\rho\downarrow0$,
$x_M(\rho)$, and analogously $y_M(\rho)$, is of order $1$, $\sqrt\rho$, or
$\rho$ according as $r_x$ is positive, zero, or negative. Hence
$g_M(\rho)\to+\infty$ unless both
$r_x<0$ and $r_y<0$. In that remaining case,
\begin{equation}
\label{eq:power-cone-lower-endpoint-limit}
    \lim_{\rho\downarrow0}g_M(\rho)
    =\log\left(
    \frac{m_zR}
    {\left(\frac{-m_xr_x}{\alpha}\right)^\alpha
     \left(\frac{-m_yr_y}{1-\alpha}\right)^{1-\alpha}}
    \right)>0.
\end{equation}
The final inequality is exactly the failure of the zero-projection test
$-Mr\in(\mathcal K_{\mathrm{pow}}^\alpha)^*$. Thus $[0,R]$ is a
sign-changing bracket in the limiting sense. Finally, the weighted projection
problem has a strictly convex objective, so its KKT point is unique.
Consequently, $g_M$ has a unique zero in $(0,R)$, and bisection converges to
it. Setting $M=I$ proves the Euclidean statement.
\end{proof}

\section{Large-Scale Instance Statistics}
\label{app:instance-statistics}

Table~\ref{tab:lasso-instances} reports the dimensions and densities of
the large-scale Lasso-derived QP instances used in
Section~\ref{sec:realworldqp}; $m$ and $n$ denote the number of samples
and features of the underlying design matrix $A$.

\begin{table}[H]
  \centering
  \papertablesetup
  \caption{Instance statistics of the large-scale Lasso-derived QPs.}
  \label{tab:lasso-instances}
  \medskip
\begin{tabular}{lrrr}
    \toprule
    Problem & $m$ & $n$ & Density \\
    \midrule
    SLS & 1,748,122 & 62,729 & $6.21\times10^{-5}$ \\
    rcv1\_test & 677,399 & 47,236 & $1.55\times10^{-3}$ \\
    avazu-site.tr & 23,567,843 & 1,000,000 & $1.50\times10^{-5}$ \\
    avazu-app & 40,428,967 & 1,000,000 & $1.50\times10^{-5}$ \\
    avazu-site & 25,832,830 & 1,000,000 & $1.50\times10^{-5}$ \\
    kddb2010\_test & 748,401 & 1,163,024 & $7.74\times10^{-6}$ \\
    kdda2010\_test & 510,302 & 20,216,830 & $1.87\times10^{-5}$ \\
    kddb2010\_train & 19,264,097 & 1,163,024 & $7.97\times10^{-6}$ \\
    kdda2010\_train & 8,407,752 & 20,216,830 & $1.80\times10^{-6}$ \\
    \bottomrule
  \end{tabular}
\end{table}

Table~\ref{tab:fisher-instances} reports the dimensions of the quasilinear
Fisher equilibrium instances used in
Section~\ref{sec:fisher-multigpu-exp}. Here $n$ and $m$ are the numbers of
buyers and goods, $e=|\mathcal E|$ is the number of nonzero buyer-good
valuations, $N=e+4n$ is the number of stored primal coordinates, and
$M=m+n$ is the number of affine equations in \eqref{eq:fisher-conic}.

\begin{table}[H]
  \centering
  \papertablesetup
  \caption{Instance statistics of the quasilinear Fisher equilibrium
  benchmark.}
  \label{tab:fisher-instances}
  \medskip
  \begin{tabular}{rrrrrr}
    \toprule
    Buyers $n$ & Goods $m$ & Density $\rho$ & Edges $e$ & Variables $N$ & Equations $M$ \\
    \midrule
    1,000      & 400   & 0.20 & 79,956      & 83,956      & 1,400 \\
    10,000     & 4,000 & 0.02 & 798,310     & 838,310     & 14,000 \\
    100,000    & 4,000 & 0.02 & 7,999,211   & 8,399,211   & 104,000 \\
    1,000,000  & 4,000 & 0.02 & 79,996,397  & 83,996,397  & 1,004,000 \\
    10,000,000 & 4,000 & 0.01 & 400,022,237 & 440,022,237 & 10,004,000 \\
    \bottomrule
  \end{tabular}
\end{table}

\section{Ablation: Inner Accuracy and Preconditioning}
\label{app:inner-accuracy-ablation}

We isolate the effect of the monotone inner tolerance
\eqref{eq:monotone-primal-inner-tol} and of Jacobi inner preconditioning on
the non-diagonal Maros-M{\'e}sz{\'a}ros subset, whose instances invoke the
inner solver, at tolerance $10^{-6}$ with a 1000 second time limit. The PDHCG
baseline sets the inner tolerance to $0.05$ times the current KKT residual.

\begin{table}[H]
    \centering
    \papertablesetup
    \caption{Ablation of adaptive inner tolerance and inner preconditioning
    on non-diagonal Maros-M{\'e}sz{\'a}ros problems.}
    \label{tab:adaptive_inner}
    \medskip
    \begin{tabular}{lrrr}
        \toprule
        Solver & Solved & Time SGM$_{10}$ (s) & Iteration SGM$_{10}$ \\
        \midrule
        PDHCG baseline        & 76          & 11.04         & 18,434 \\
        \name without Jacobi preconditioning
                                        & \underline{78} & \underline{6.15} & \underline{6,495} \\
        \name with Jacobi preconditioning
                                        & \textbf{79} & \textbf{5.14} & \textbf{5,690} \\
        \bottomrule
    \end{tabular}
\end{table}

The monotone adaptive tolerance sharply reduces runtime and iteration count
relative to the PDHCG baseline, and Jacobi inner preconditioning solves one
additional hard instance while lowering the runtime
$\operatorname{SGM}_{10}$ from 6.15 to 5.14 seconds.

\end{document}